\documentclass[reqno]{amsart}            
\usepackage{amsthm,amssymb,amsfonts,graphicx,cite,color}
\usepackage[bookmarks,bookmarksnumbered,bookmarksopen,colorlinks,backref,linkcolor=blue,citecolor=red]{hyperref}%
\usepackage{xcolor}
\usepackage{graphicx}
\usepackage{mathrsfs}

\usepackage{tikz}
\usetikzlibrary{matrix}

\newtheorem{theorem}{Theorem}[section]
\newtheorem{definition}{Definition}[section]

\newtheorem{corollary}{Corollary}[section]
\newtheorem{remark}{Remark}[section]
\newtheorem{proposition}{Proposition}[section]
\numberwithin{equation}{section}

\begin{document}
\hyphenation{ap-pro-xi-ma-tion}
\title[Approximation for the solution to a parabolic problem  in a thin star-shaped junction]
{Asymptotic approximation for the solution to a semi-linear parabolic problem  in a thin star-shaped junction}
\author[A.V. Klevtsovskiy and  T.A.~Mel'nyk]{Arsen V. Klevtsovskiy and Taras A. Mel'nyk
}
\address{\hskip-12pt  Faculty of Mathematics and Mechanics, Department of Mathematical Physics\\
Taras Shevchenko National University of Kyiv\\
Volodymyrska str. 64,\ 01601 Kyiv,  \ Ukraine
}
\email{avklevtsovskiy@gmail.com, \ \ melnyk@imath.kiev.ua}

\begin{abstract}
A semi-linear parabolic problem  is considered  in a thin $3D$
star-shaped junction that consists of several thin curvilinear cylinders that are joined through a domain (node) of diameter $\mathcal{O}(\varepsilon).$

The purpose is to study the asymptotic behavior of the solution $u_\varepsilon$ as $\varepsilon \to 0,$ i.e. when the
star-shaped junction is transformed in a graph. In addition, the passage to the limit is accompanied by special intensity factors $\{\varepsilon^{\alpha_i}\}$ and $\{\varepsilon^{\beta_i}\}$ in
nonlinear perturbed Robin boundary conditions.

We establish qualitatively different cases in the asymptotic behaviour of the solution depending on the value of the parameters $\{{\alpha_i}\}$ and $\{\beta_i\}.$
Using the multi-scale analysis, the asymptotic approximation for the solution is constructed and justified as the parameter $\varepsilon \to 0.$
Namely, in each case we derive the limit problem  $(\varepsilon =0)$ on the graph
with the corresponding Kirchhoff transmission conditions (untypical in some cases) at the vertex, \, define other terms of the asymptotic approximation and prove
appropriate asymptotic estimates that justify these coupling conditions at the vertex and show  the impact of the local geometric heterogeneity of the node and physical processes in the node on some properties of the solution.

\end{abstract}

\keywords{Approximation,  semi-linear parabolic problem,  nonlinear perturbed boundary condition,
asymptotic estimate, thin star-shaped junction.
\\
\hspace*{9pt} {\it MOS subject classification:} \  35K57, 35K55, 35B40, 35B25, 74K30
}

\maketitle
\tableofcontents
\section{Introduction}\label{Sect1}
We are interested in the study of evolution phenomena in junctions composed of several thin curvilinear cylinders that are joined through a domain of diameter $\mathcal{O}(\varepsilon)$ (see Fig.~\ref{f1}).
Mathematical models those are described by semi-linear parabolic equations that allow to model a variety of  biological and physical phenomena (reaction and diffusion processes in biology and biochemistry, heat-mass transfer, etc.) in channels, junctions and networks.

 As we can see from Fig.~\ref{f1},
 a thin junction is shrunk into a graph  as the small parameter $\varepsilon,$  characterizing  thickness of the thin cylinders and domain connecting them, tends to zero.  Thus, the aim is to find the corresponding limit problem in this graph and prove the estimate for the difference between the solutions  of these two problems.

\begin{figure}[htbp]
\begin{center}
\includegraphics[width=5cm]{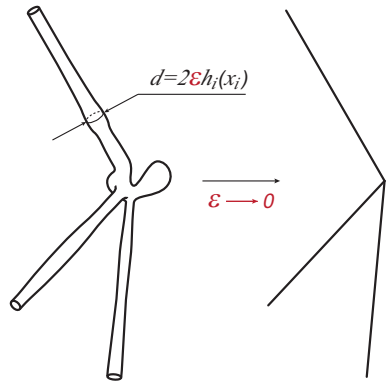}
\end{center}
\caption{Transformation of a thin star-shaped junction into a graph}\label{f1}
\end{figure}

A large amount of physical and mathematical articles and books dedicated to different models on graphs, has been published for the last three decades, e.g. \cite{KowalSivanEntinWohlmanImry1990,AvishaiLuck1992,Kuchment2004,Below1988, Mehmeti1994, PokornyBook2004,BandaHertyKlar2006,KostrykinPotthoffSchrader2012,Manko2014,FlyudGolovaty2017,Mugnolo,Mug-Rom-2007,Pel-Schneider}.
The main question arising in problems on graphs is point interactions at nodes of networks, i.e., the type of coupling conditions at vertices of the graph.

Also there is increasing interest in the investigation of the influence of a local geometric heterogeneity
 in vessels on the blood flow. This is both an aneurysm (a pathological extension of an artery like a bulge) and a stenosis (a pathological restriction of an artery). In \cite{Mardal} the authors classified 12 different aneurysms and proposed a numerical approach for this study. The aneurysm models have been meshed with 800,000 -- 1,200,000 tetrahedral cells containing three boundary layers. However, as was noted by the authors, {\it the question how to model blood flow with sufficient accuracy  is still open}.

Because of those point interactions and local geometric irregularities, the reaction-diffusion processes,  heat-mass transfer and flow motions
in networks posses many distinguishing features. A natural approach to explain the meaning of point interactions at vertices is the use of the limiting procedure mentioned above.

There are several asymptotic approaches to study such problems.
As far as we known, the paper \cite{Freid-Went-1993} was the first paper, where
convergence results for linear diffusion processes in a region with narrow tubes were obtained with the help of
 the martingal-problem method of proving weak convergence. As a result, the standard gluing conditions (or so-called "Kirchhoff" transmission  conditions) at the vertices of the graph were derived. Then this probabilistic approach was generalized in \cite{Albe-Kusu-2012}.

The method of the partial asymptotic domain decomposition was proposed in \cite{Pan-decom-1998} and then it applied
to different problems under the following assumptions:
the uniform boundary conditions on the lateral surfaces of thin rectilinear cylinders, the right-hand sides depend only on
the longitudinal variable in the direction of the corresponding cylinder and they are constant in some neighbourhoods of the nodes and vertices
(see \cite{C-C-P-2006,P-book,P-P-Stokes-1-2015,P-P-Stokes-2-2015,GauPanPia2016}).
It  follows from these papers that the main difficulty is the identification of the behaviour of solutions in neighbourhoods of the nodes.

To overcome this difficulty and to construct the leading terms of the
elastic field asymptotics for the solution of the equations of anisotropic elasticity on junctions of
thin three dimensional beams, the following assumptions were made in \cite{N-S-04}:
 the first terms of the volume force $f$ and surface load $g$ on the rods satisfy special orthogonality conditions (see $(3.5)_1$ and $(3.6))$  and the second term of the volume force $f$ has an identified form and depends only on the longitudinal variable;
 similar orthogonality conditions for the right-hand sides on the nodes are satisfied (see $(3.41)$) and the second term is a piecewise constant vector-function (see $(3.42)$).
By these assumptions, the displacement field at each node can be approximated by a rigid
displacement. As a result, the approximation does not contain boundary layer terms, i.e., the
asymptotic expansion is not complete a priori \cite[ Remark 3.1]{N-S-04}.
Similar approach was used for thin two dimensional junctions in \cite{N-S-02}.

There is a special interest in spectral problems on thin graph-like structures, since such problems have many applications.
A fairly complete review on this topic has been presented in \cite{Kuchment2002}. The main task is to study the possibility of approximating the spectra of different operators by the spectra of appropriate operators on the corresponding graph. The convergence of  spectra for the Laplacians
with different boundary conditions (Neumann, Dirichlet and Robin) at various levels of generality was proved in  \cite{Kuch-Zeng2001,Exner-Post-2005,Rub-Schatz-2001,Duclos-Exner-1995,Post-2005,Mol-Vain-2007,Mol-Vain-2007+,Grieser2008}.
In \cite{Kuch-Zeng2001} the authors took into account large protrusions at the vertices; as a result different Kirchhoff  conditions are appeared depending on the value of the protrusion.
It was demonstrated in \cite{Grieser2008} that the type of the transmission conditions depends crucially on the boundary layer phenomenon in the vicinity of the nodes; in addition the complete asymptotic expansions for
the $k$-th eigenvalue and the eigenfunctions were obtained there, uniformly for $k,$  in terms of scattering data
on a non-compact limit space.
Interesting multifarious transmission conditions are obtained in the limit passage for spectral problems on thin periodic honeycomb lattice
\cite{Kuch-Kunya2002,Naz-Ruo-Uus-2016}.
 Numerical approach to deduce the vertex coupling conditions for the nonlinear Schr\"odinger equation
on two-dimensional thin networks was proposed in \cite{Grieser2015}.

\subsection{Novelty and method of the study}

In the present paper we continue to develop the asymptotic method proposed in our papers~\cite{Mel_Klev_AA-2016,Mel_Klev_AA-2017} for linear elliptic problems, which does not need the above mentioned assumptions. In addition, our approach gives the better estimate for the difference between the solution of the starting  problem and the solution of the corresponding limit problem (compare (1) and (2) in \cite{Mel_Klev_AA-2016}).

Here we have adapted  this method to semi-linear parabolic problems with nonlinear perturbed Robin boundary conditions
\begin{equation}\label{in-1}
    \partial_{\boldsymbol{\nu}} u_\varepsilon
 +  \varepsilon^{\alpha_i}\kappa_i\big({u_\varepsilon}, x_i, t\big)
  =  \varepsilon^{\beta_i}\, {\varphi_{\varepsilon}^{(i)}(x, t)}
\end{equation}
both on the boundaries of the thin curvilinear cylinders $(i\in \{1, 2, 3\})$ and on the boundary of the node $(i=0),$ which
depend on special intensity factors $\varepsilon^{\alpha_i}$ and $\varepsilon^{\beta_i}.$
We study the influence of these factors on the asymptotic behaviour of the solution as $\varepsilon \to 0.$

It turned out that the asymptotic behaviour of the solution depends on
 the parameters $\{\alpha_i\}$ and $\{\beta_i\},$ and essentially on the parameter $\alpha_0$
 that characterizes the intensity of processes at the boundary of the node.
  It is natural to expect that physical processes on the node boundary   provoke crucial changes in the whole process in the thin star-shaped junction, in particular they can reject the traditional Kirchhoff transmission conditions at the vertex in some cases. We discovery three  qualitatively different cases in the asymptotic behaviour of the solution. If $\alpha_0 > 0,\ \beta_0>0, \  \alpha_i, \beta_i\ge 1, \ i\in\{1, 2, 3\},$ then we have classical
  Kirchhoff transmission conditions. In the case $\alpha_0 = 0,\ \beta_0=0, \ \alpha_i, \beta_i\ge 1,  \ i\in\{1, 2, 3\},$
 new   gluing conditions at the vertex $x=0$ of the graph look as follows
\begin{gather*}
\omega_{0}^{(1)}(0, t) = \omega_{0}^{(2)}(0, t) = \omega_{0}^{(3)}(0, t),
\\
 \pi h_1^2 (0) \frac{\partial\omega_{0}^{(1)}}{\partial x_1} (0, t) +  \pi h_2^2 (0) \frac{\partial\omega_{0}^{(2)}}{\partial x_2} (0, t)
 + \pi h_3^2 (0) \frac{\partial\omega_{0}^{(3)}}{\partial x_3} (0, t)
 -   \big|\Gamma_0\big|_2 \, \kappa_0 \big(\omega_0^{(1)}(0, t)\big)
 = - \int_{\Gamma_{0}} \varphi^{(0)}(\xi, t) \, d\sigma_{\xi},
\end{gather*}
where $\big|\Gamma_0\big|_2$ is the Lebesgue measure of the boundary $\Gamma_0$ of the node.
If $\alpha_0 < 0$ the limit problem splits in three independent problems with the Dirichlet conditions.

To construct the asymptotic approximation in each case, we use the  method of matching asymptotic expansions (see  \cite{I1992}) with special cut-off functions. The approximation  consists of two parts, namely, the regular part of the asymptotics  located inside of each thin cylinder and the inner part of the asymptotics discovered in a neighborhood of the node.
The terms of the inner part of the asymptotics are special solutions of boundary-value problems in an unbounded domain with different outlets at infinity. It turns out they have polynomial growth at infinity. Matching these parts, we derive the limit problem  $(\varepsilon =0)$ in the graph and the corresponding coupling conditions at the vertex.

Also we have proved energetic estimates in each case which allow  to identify more precisely
the impact of the local geometric heterogeneity of the node and physical processes in the node on some properties of the solution.
It should be stressed that the error estimates and convergence rate are very important both for  justification of  adequacy of one- or two-dimensional models that aim at  description of actual three-dimensional thin bodies and for the study of boundary effects and effects of local (internal) inhomogeneities in applied problems. In addition, those estimates justify  transmission conditions of Kirchhoff  type for metric graphs.

Thus, our approach makes it possible to take into account various factors (e.g. variable thickness of thin curvilinear cylinders,
 inhomogeneous nonlinear boundary conditions, geometric characteristics of nodes, etc.) in statements of boundary-value problems on graphs.

\smallskip

The rest of this paper is organized as follows.
The statement of the problem and features of the investigation are presented in Section~\ref{statement}.
In Section~\ref{exis-uniq},  the existence and uniqueness of the weak solution is proved for every fixed value $\varepsilon.$ Also
a priori estimates  and auxiliary inequalities are deduced there.
In Section~\ref{Formal asymptotics}
  we formally construct the leading terms both of the regular part of the asymptotics and the inner one
in the case $\alpha_0 \ge 0,$ $\alpha_i \ge 1,$ $ i\in\{1, 2, 3\}.$ Then using the constructed terms we build  the approximation and prove the corresponding asymptotic estimates in Section~\ref{justification}.
Section~\ref{alpha_0<0} shows us what will happen in the case
$\alpha_0 < 0,$ $\alpha_i \ge 1,$ $i\in\{1, 2, 3\}.$ The main novelty is that the limit problem splits into three independent problems
with the uniform Dirichlet condition at the vertex. In addition, the view of asymptotic ansatzes are very sensitive to the parameter $\alpha_0.$
Here we construct the approximation and prove the corresponding estimates for more typical and realistic subcases $\alpha_0 \in (-1,  0)$ and $\alpha_0 =-1;$  general case is only discussed.
In Section~\ref{comments}, we analyze obtained results and discuss research perspectives.

\section{Statement of the problem}\label{statement}

The model thin star-shaped junction  $\Omega_\varepsilon$  consists of three thin curvilinear cylinders
$$
\Omega_\varepsilon^{(i)} =
  \Bigg\{
  x=(x_1, x_2, x_3)\in\Bbb{R}^3 \ : \
  \varepsilon \ell_0<x_i<\ell_i, \quad
  \sum\limits_{j=1}^3 (1-\delta_{ij})x_j^2<\varepsilon^2 h_i^2(x_i)
  \Bigg\}, \quad i=1,2,3,
$$
that are joined through a domain $\Omega_\varepsilon^{(0)}$ (referred in the sequel "node").
Here $\varepsilon$ is a small parameter; $\ell_0\in(0, \frac13), \ \ell_i\geq1, \ i=1,2,3;$
 the positive function $h_i$ belongs to the space $C^1 ([0, \ell_i])$
and it is equal to some constants in neighborhoods of the points $x=0$ and $x_i=1$ $(i= 1, 2, 3)$;
the symbol $\delta_{ij}$ is the Kroneker delta, i.e.,
$\delta_{ii} = 1$ and $\delta_{ij} = 0$ if $i \neq j.$

\begin{figure}[htbp]
\begin{center}
\includegraphics[width=4cm]{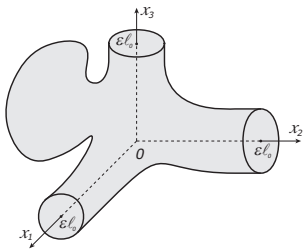}
\end{center}
\caption{The node $\Omega_\varepsilon^{(0)}$}\label{f2}
\end{figure}

The node $\Omega_\varepsilon^{(0)}$ (see Fig.~\ref{f2}) is formed by the homothetic transformation with coefficient $\varepsilon$ from a bounded domain
$\Xi^{(0)}\subset \Bbb R^3$,  i.e.,
$
\Omega_\varepsilon^{(0)} = \varepsilon\, \Xi^{(0)}.
$
In addition, we assume that its boundary contains the disks
$$
\Upsilon_\varepsilon^{(i)} (\varepsilon\ell_0) =
\Bigg\{
 x\in\Bbb{R}^3 \, : \ x_i=\varepsilon \ell_0, \quad
 \sum\limits_{j=1}^3 (1-\delta_{ij})x_j^2<\varepsilon^2 h_i^2(\varepsilon \ell_0)
\Bigg\}, \quad i=1,2,3,
$$
and denote
$
\Gamma_\varepsilon^{(0)} :=
\partial\Omega_\varepsilon^{(0)} \backslash
\left\{
 \overline{\Upsilon_\varepsilon^{(1)} (\varepsilon \ell_0)} \cup
 \overline{\Upsilon_\varepsilon^{(2)} (\varepsilon \ell_0)} \cup
 \overline{\Upsilon_\varepsilon^{(3)} (\varepsilon \ell_0)}
\right\}.
$

Thus the model thin star-shaped junction  $\Omega_\varepsilon$  (see Fig.~\ref{f3})
is   the interior of the union
$
\bigcup_{i=0}^{3}\overline{\Omega_\varepsilon^{(i)}}
$
and we assume that it has the Lipschitz boundary.

\begin{figure}[htbp]
\begin{center}
\includegraphics[width=6cm]{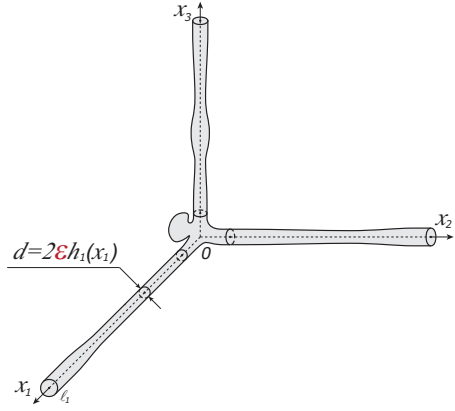}
\end{center}
\caption{The model thin star-shaped junction $\Omega_\varepsilon$}\label{f3}
\end{figure}

\begin{remark}
We can consider more general thin star-shaped junctions with arbitrary orientation of thin cylinders (their number can be also arbitrary).
 But to avoid technical and huge calculations and to demonstrate the main steps of the proposed asymptotic approach
 we consider the case  when the cylinders are placed on the coordinate axes.
\end{remark}

In $\Omega_\varepsilon,$ we consider the following semi-linear parabolic problem:
\begin{equation}\label{probl}
\left\{\begin{array}{rclll}
 \partial_t u_\varepsilon (x, t)
 -  \, \Delta_x u_\varepsilon (x, t)
 +  k\big({u_\varepsilon}(x, t)\big)
 & = & f(x, t), &
    (x, t)\in\Omega_\varepsilon \times (0, T), &
\\
  \partial_{\boldsymbol{\nu}} u_\varepsilon(x, t)
 +  \varepsilon^{\alpha_0}\kappa_0\big({u_\varepsilon}(x, t)\big)
 & = & \varepsilon^{\beta_0}\, {\varphi_{\varepsilon}^{(0)}(x, t)}, &
    (x, t)\in{\Gamma_\varepsilon^{(0)}} \times (0, T), &
\\
  \partial_{\boldsymbol{\nu}} u_\varepsilon(x, t)
 +  \varepsilon^{\alpha_i}\kappa_i\big({u_\varepsilon}(x, t), x_i, t\big)
 & = & \varepsilon^{\beta_i}\, {\varphi_{\varepsilon}^{(i)}(x, t)}, &
    (x, t)\in{\Gamma_\varepsilon^{(i)}} \times (0, T), & i=1,2,3,
\\
    u_\varepsilon(x, t)
 & = & 0, &
    (x, t)\in{\Upsilon_{\varepsilon}^{(i)} (\ell_i)} \times (0, T), & i=1,2,3,
\\
    u_\varepsilon(x, 0)
 & = & 0, & x\in\Omega_\varepsilon,
\end{array}\right.
\end{equation}
where
${\Gamma_\varepsilon^{(i)}} =
\partial\Omega_\varepsilon^{(i)} \cap \{ x\in\Bbb{R}^3 \ : \ \varepsilon \ell_0<x_i<\ell_i \}, \
T>0, \ \partial_{t} = \partial / \partial t, \ \partial_{\boldsymbol{\nu}}$
is the outward normal derivative, and the parameters $\{\alpha_i\}_{i=0}^3 \subset \Bbb R,$ $ \beta_0 \geq 0, \ \beta_i \geq 1, \ i=\overline{1,3}.$
For the given functions $f,$ $k,$ $\{\varphi_{\varepsilon}^{(i)}, \kappa_i\}_{i=0}^3$  we assume the following conditions:
\begin{enumerate}
 \item[{\bf C1.}]
 the function $f$ belongs to the space $C\left(\overline{\Omega_{a_0}} \times [0, T] \right)$
and its restriction on the  curvilinear cylinder  $\Omega_{a_0}^{(i)}$ $(i=1,2,3)$
belong to the space
$
C_{\overline{x}_i}^1
\left( \overline{\Omega_{a_0}^{(i)}} \times [0, T] \right)
$
(the space of all continuous functions having continuous derivatives with respect to variables $\overline{x}_i$ in
$\overline{\Omega_{a_0}^{(i)}} \times [0, T]),$
where $a_0$ is a fixed positive number such that $\Omega_{\varepsilon} \subset \Omega_{a_0}$ for
all values of the small parameter $\varepsilon\in (0, \varepsilon_0)$
and
$$
\overline{x}_i =
\left\{\begin{array}{lr}
(x_2, x_3), & i=1, \\
(x_1, x_3), & i=2, \\
(x_1, x_2), & i=3;
\end{array}\right.
$$
\item[{\bf C2.}] the functions
$
\varphi_{\varepsilon}^{(0)} (x, t) :=
\varphi^{(0)} \left(\dfrac{x}{\varepsilon}, t \right)
$
and
$
\varphi_{\varepsilon}^{(i)} (x, t) :=
\varphi^{(i)} \left(\dfrac{\overline{x}_i}{\varepsilon}, x_i, t \right), \ \ i=1,2,3,
$
belong to the spaces
$
C\left(\overline{\Omega_{a_0}^{(i)}} \times [0, T] \right),
\ i\in \{0,1,2,3\},
$
respectively;
\item[{\bf C3.}]
the functions $\{\kappa_i(s,x_i,t)\}_{i=1}^3,$ $(s,x_i,t)\in \Bbb R \times[0,\ell_i]\times[0,T]$
are continuous in their domains of definition  and  have the partial derivatives with respect to $s,$
$k \in C^1(\Bbb R),$ $\kappa_0 \in C^1(\Bbb R),$ and
  there exists a positive constant $k_+$ such that
\begin{equation}\label{kappa_ineq}
0\leq k^\prime(s) \leq k_+, \quad  0\leq \kappa_0^\prime(s) \leq k_+, \quad
0\leq \partial_s \kappa_i(s,x_i,t)  \leq k_+
\quad \mbox{for} \quad s\in\Bbb{R}
\end{equation}
uniformly with respect to $x_i\in[0, \ell_i]$ and $t\in[0, T],$ respectively;
\begin{enumerate}
  \item
  if $\alpha_0 < 0,$ then in addition, the function $\kappa_0$ is a $C^2$-function with bounded derivatives, there exists  a constant $k_-$ such that
  $ 0< k_- \leq \kappa_0^\prime(s)$ for all $s\in\Bbb{R}$ and $\kappa_0(0)=0$ (so-called condition of
zero-absorption).
\end{enumerate}
\end{enumerate}

\medskip

Denote by
$\mathcal{H}_\varepsilon^*$ the dual space to the Sobolev space
$
\mathcal{H}_\varepsilon = \big\{
u \in H^1(\Omega_\varepsilon) : u|_{\Upsilon_\varepsilon^{(i)}(\ell_i)}=0, \ i=1,2,3
\big\}.
$
Recall that a function
$u_\varepsilon \in L^2 \left(0, T; \, \mathcal{H}_\varepsilon \right),$
with
$\partial_t u_\varepsilon \in L^2 \left(0, T; \, \mathcal{H}_\varepsilon^* \right),$
is called a weak solution to the problem~(\ref{probl}) if it satisfies the integral identity
\begin{multline}\label{int-identity}
     \int_{\Omega_\varepsilon}
     \partial_t u_\varepsilon \, v \, dx
 +   \int_{\Omega_\varepsilon}
     \nabla u_\varepsilon \cdot \nabla v \, dx
 +   \int_{\Omega_\varepsilon} k(u_\varepsilon) \, v \, dx
 +   \varepsilon^{\alpha_0}
     \int_{\Gamma_\varepsilon^{(0)}} \kappa_0(u_\varepsilon) \, v \, d\sigma_x
\\
 +   \sum_{i=1}^3 \varepsilon^{\alpha_i}
     \int_{\Gamma_\varepsilon^{(i)}} \kappa_i(u_\varepsilon, x_i, t) \, v \, d\sigma_x
 =   \int_{\Omega_\varepsilon} f \, v \, dx
 +  \sum_{i=0}^3 \varepsilon^{\beta_i}
     \int_{\Gamma_\varepsilon^{(i)}} \varphi_\varepsilon^{(i)} \, v \, d\sigma_x
\end{multline}
for any function $v\in \mathcal{H}_\varepsilon$ and a.e. $t\in(0, T),$ and $u_\varepsilon|_{t=0}=0.$
It is known (see e.g. \cite{Showalter-book}) that $u_\varepsilon\in C\big([0,T]; L^2(\Omega_\varepsilon)\big),$
and thus the equality $u_\varepsilon|_{t=0}=0$ makes sense.

\medskip

The aim of the present paper is to
\begin{itemize}
  \item
  construct the asymptotic approximation for the solution to the problem  (\ref{probl})
  as the parameter $\varepsilon \to 0;$
\item
   derive the corresponding limit problem $(\varepsilon =0);$
\item
prove the corresponding asymptotic estimates from which the influence of the
local geometric heterogeneity of the node $\Omega_\varepsilon^{(0)}$ and physical processes inside
will be observed;
\item
study the influence of the parameters $\{\alpha_i, \, \beta_i\}_{i=0}^3$
on the asymptotic behavior of the solution.
 \end{itemize}

\subsection{Comments to the statement}\label{coments}
 To our knowledge, the first works on the study of boundary-value problems for reaction-diffusion equations
were papers by Kolmogorov, Petrovskii, Piskunov \cite{KolPetPis} and Fisher \cite{Fisher}.
Standard assumptions for reaction terms of semilinear equations are as follows:
\begin{itemize}
  \item
$\exists\, C>0$ \ $\forall\, s_1, \, s_2 \in \Bbb R$: \  $|k(s_1) - k(s_2)| \le C |s_1 - s_2|;$
  \item
  $\exists\, C_1>0$ \ $C_2 \ge 0$ \ $\forall\, s \in \Bbb R$: \
$k(s) s \ge C_1 s^2 - C_2.$
\end{itemize}
This is sufficient for the existence and uniqueness of the weak solution.
However, many physical processes, especially in chemistry and medicine, have monotonous nature.
Therefore, it is naturally to impose special monotonous conditions for nonlinear terms.
In our case we propose simple conditions (\ref{kappa_ineq}) which are easy to verify.
For instance, the  functions
$$
k(s) = \lambda s + \cos s \quad (\lambda  \ge 1); \qquad
k(s)= \frac{\lambda s}{1 + \nu s} \quad (\lambda, \nu > 0)
$$
satisfy this condition. The last one corresponds to the Michaelis-Menten hypothesis in biochemical reactions
and to the Langmuir kinetics adsorption models (see \cite{Conca,Pao}).

From conditions (\ref{kappa_ineq}) it follows  the following inequalities:
\begin{equation}\label{kappa_ineq+}
\begin{array}{c}
k(0) s \leq k (s) s \leq k_+ \ s^2 + k(0) s, \qquad
\kappa_0(0) s \leq \kappa_0 (s) s \leq k_+ \, s^2 + \kappa_0(0) s,
\\[2mm]
\kappa_i(0, x_i, t) s \leq \kappa_i (s, x_i, t) s \leq k_+ \, s^2 + \kappa_i(0, x_i, t) s,
 \quad \mbox{for} \quad s\in\Bbb{R}
\end{array}
\end{equation}
uniformly with respect to $(x_i,t)\in[0, \ell_i]\times[0, T],$ respectively; $i=1,2,3.$
For the case ${\bf C3}({\rm a})$  we have
\begin{equation}\label{kappa_ineq++}
k_- \, s^2 \leq \kappa_0 (s) s \leq k_+ \, s^2 \qquad \forall \, s\in \Bbb R.
\end{equation}
Doubtless both the function $k$ and $\kappa_0$ may also depend on $x$ and $t.$
However, we have omitted this dependence to avoid cumbersome formulas, leaving it only for the functions $\{\kappa_i\}_{i=1}^3.$

As will be seen from further calculations in the case when some parameter $\alpha_i >1,$ the
condition (\ref{kappa_ineq}) for the corresponding function $\kappa_i$ can be weakened.
In this case it is sufficient that $\kappa_i$  is continuous
and there exist  constants $c_1>0, \ c_2\ge0$ such that
for any $s_1,\,s_2, \, s\in\Bbb{R},$ $ x_i\in [0, \ell_i],$ $ t\in[0, T]:$
$$
 \big(\kappa_i(s_1,x_i,t)-\kappa_i(s_2,x_i,t)\big)(s_1 - s_2) \geq 0, \qquad \big|\kappa_i(s,x_i,t) \big| \leq c_1 \big(1+|s|\big),\qquad \kappa_i(s,x_i,t) s \geq - c_2.
 $$

It should be noted here that the asymptotic behaviour of solutions to the reaction-diffusion equation in different kind of thin domains
with the uniform Neumann conditions was studied in \cite{MarRyb,ArCaPeSi}. The convergence theorems were proved under the following
assumptions for the reaction term $k$:
in \cite{ArCaPeSi} it is a $C^2$-function with bounded derivatives and
\begin{equation}\label{disip-cond}
\liminf_{|s|\to+\infty} \frac{k(s)}{s} > 0;
\end{equation}
in \cite{MarRyb} it is  a $C^1$-function, $|k'(s)| \le C (1 + |s|^q),$ where $q\in (0,+\infty),$  and
the dissipative condition (\ref{disip-cond}) is satisfied. It is easy to see that from (\ref{kappa_ineq++}) it follows (\ref{disip-cond}).


In a typical interpretation the solution to the problem (\ref{probl}) denotes the density of
some quantity (temperature, chemical concentration, the potential of a vector-field, etc.)
within the thin star-shaped junction $\Omega_\varepsilon.$
The nonlinear Robin boundary conditions considered in this problem mean that there is
some interaction between the surrounding density and the density just inside $\Omega_\varepsilon.$
 It is evident from the results we have presented that these conditions
(essentially the condition at the boundary of the node)  have a substantial influence on the asymptotic behavior of
the solution. To study this influence, we introduce special intensity factors
$\varepsilon^{\alpha_i}, \ i\in\{0, 1, 2, 3\}.$ Since in this paper we are more interested in the study
of the boundary interactions at the node, we take the parameter $\alpha_0$ from $\Bbb R$ and the other ones from $[1, +\infty).$ The case when
$ \alpha_i < 1 \ (i\in\{1, 2, 3\})$ is only discussed in Sec.~\ref{comments}.

\section{Existence and uniqueness of the weak solution}\label{exis-uniq}

In order to obtain the operator statement for the problem~(\ref{probl}) we introduce the new norm
$\| \cdot \|_\varepsilon$ in $\mathcal{H}_\varepsilon$, which is generated by the scalar product
$$
(u,v)_\varepsilon = \int_{\Omega_\varepsilon} \nabla{u} \cdot \nabla{v} \, dx,
\quad u,v \in \mathcal{H}_\varepsilon.
$$
Due to the uniform Dirichlet condition on $\Upsilon_\varepsilon^{(i)}(\ell_i), \ i=1,2,3,$
the norm $\| \cdot \|_\varepsilon$ and the ordinary norm $\| \cdot \|_{H^1(\Omega_\varepsilon)}$
are uniformly equivalent, {\it i.e.}, there exist constants $C_1>0$ and $\varepsilon_0>0$
such that for all $\varepsilon\in(0, \varepsilon_0)$ and for all $u\in\mathcal{H}_\varepsilon$
the following estimate hold:
\begin{equation}\label{equivalent_norm}
\| u \|_\varepsilon \leq \| u \|_{H^1(\Omega_\varepsilon)} \leq C_1 \| u \|_\varepsilon.
\end{equation}

\begin{remark}
Here and in what follows all constants $\{C_j\}$ and $\{c_j\}$
in inequalities are independent of the parameter $\varepsilon.$
\end{remark}

Further we will often use the inequalities
\begin{gather}\label{ineq1}
      \varepsilon \int_{\Gamma_\varepsilon^{(i)}} v^2 \, d\sigma_x
 \leq C_2 \Bigg( \varepsilon^2 \int_{\Omega_\varepsilon^{(i)}} |\nabla_{x}v|^2 \, dx
 +    \int_{\Omega_\varepsilon^{(i)}} v^2 \, dx \Bigg),
 \\
\int_{\Omega_\varepsilon^{(i)}} v^2 \, dx \le C_3 \Bigg( \varepsilon^2 \int_{\Omega_\varepsilon^{(i)}} |\nabla_{x}v|^2 \, dx +
  \varepsilon \int_{\Gamma_\varepsilon^{(i)}} v^2 \, d\sigma_x\Bigg)
\qquad \forall \, v\in H^1(\Omega_\varepsilon^{(i)}), \ \ (i\in \{1, 2, 3\})\label{ineq-2}
\end{gather}
proved in \cite{M-MMAS-2008}.  Let us prove similar inequalities for the node $(i=0).$

\begin{proposition} Let $Q$ be a bounded domain in $\Bbb{R}^3$ with the smooth boundary $\partial Q.$
Then there exists a positive constant $C_2>0$ that is independent of $\varepsilon$ such that for any function $v$
from the space $H^1(Q_\varepsilon)$ the following inequalities hold:
\begin{equation}\label{traceineq0}
      \varepsilon \int\limits_{\partial Q_\varepsilon} v^2 \, d\sigma_x
 \leq C_2 \Bigg( \varepsilon^2 \int\limits_{Q_\varepsilon} |\nabla_{x}v|^2 \, dx
 +    \int\limits_{Q_\varepsilon} v^2 \, dx \Bigg) \quad \text{and} \quad
\int\limits_{Q_\varepsilon} v^2 \, dx \le C_3 \Bigg( \varepsilon^2 \int\limits_{Q_\varepsilon} |\nabla_{x}v|^2 \, dx +
  \varepsilon \int\limits_{\partial Q_\varepsilon} v^2 \, d\sigma_x\Bigg),
\end{equation}
where $Q_\varepsilon:= \varepsilon\,Q$ is the homothetic transformation with the coefficient $\varepsilon$ of  $Q.$
\end{proposition}
\begin{proof}
Let $r(\mathfrak{s}) = \big( r_1(\mathfrak{s}), \ r_2(\mathfrak{s}), \ r_3(\mathfrak{s}) \big),
\ \mathfrak{s}\in \mathfrak{S}\subset\Bbb{R}^2,$ be a smooth parametrisation of $\partial Q.$
Then $r_\varepsilon := \varepsilon r =
\big( \varepsilon r_1, \ \varepsilon r_2, \ \varepsilon r_3 \big)$
is the  parametrization of $\partial Q_\varepsilon.$ Denote by $\rho(r) := \sqrt{EG-F^2 \,},$ where
$
E = \sum_{i=1}^3 \left(\frac{\partial r_i}{\partial \mathfrak{s}_1}\right)^2, \
G = \sum_{i=1}^3 \left(\frac{\partial r_i}{\partial \mathfrak{s}_2}\right)^2, \
F = \sum_{i=1}^3       \frac{\partial r_i}{\partial \mathfrak{s}_1}
                       \frac{\partial r_i}{\partial \mathfrak{s}_2}.
$
Then $\rho(r_\varepsilon) = \varepsilon^2 \rho(r).$
Using definition of the surface integral, we get
\begin{equation}\label{trace1}
   \int_{\partial Q_\varepsilon} v^2(x) \, d\sigma_x
 = \int_\mathfrak{S} v^2 \big(r_\varepsilon(\mathfrak{s})\big)
   \rho\big(r_\varepsilon(\mathfrak{s})\big) \, d\mathfrak{s}
 = \varepsilon^2 \int_\mathfrak{S} v^2 \big(\varepsilon r(\mathfrak{s})\big)
   \rho\big(r(\mathfrak{s})\big) \, d\mathfrak{s}
 = \varepsilon^2 \int_{\partial Q} v_\varepsilon^2(\xi) \, d\sigma_{\xi}
\end{equation}
for all $v\in H^1(Q_{\varepsilon}),$
where $v_\varepsilon(\xi) := v(\varepsilon\xi), \ \xi=(\xi_1, \xi_2, \xi_3),$
and $x=\varepsilon\xi.$

Taking into account the boundedness of the trace operator, i.e.,
\begin{equation*}
\exists c_0>0: \qquad \| v_\varepsilon \|_{L^2(\partial Q)} \le c_0 \| v_\varepsilon \|_{H^1(Q)},
\end{equation*}
where constant $c_0$ does not depend on $v_\varepsilon,$ and the
equality
$$
\varepsilon^3
   \left(
    \int_{Q} |\nabla_{\xi} v_\varepsilon|^2 \, d\xi
 +  \int_{Q} v_\varepsilon^2 \, d\xi
   \right)
 = \varepsilon^2 \int_{Q_\varepsilon} |\nabla_{x} v|^2 \, dx
 + \int_{Q_\varepsilon} v^2 \, dx,
$$
we obtain the first inequality in (\ref{traceineq0}). By the same arguments we can prove the second one.
\end{proof}

It is easy to prove  the inequality
$$
\int_{\Omega_\varepsilon^{(0)}} v^2 \, dx \le C_4 \varepsilon \Bigg( \int_{\Omega_\varepsilon} |\nabla_{x}v|^2 \, dx +
  \int_{\Upsilon_\varepsilon^{(i)}(\ell_i)} v^2 \, d\overline{x}_i\Bigg)
$$
and then with the help of the first inequality in (\ref{traceineq0}) the following one:
\begin{equation}\label{ineq-3}
\int_{\Gamma_\varepsilon^{(0)}} v^2 \, d\sigma_x
 \leq C_5 \Bigg( \int_{\Omega_\varepsilon} |\nabla_{x}v|^2 \, dx
 +    \int_{\Upsilon_\varepsilon^{(i)}(\ell_i)} v^2 \, d\overline{x}_i\Bigg)
\end{equation}
for all $v\in H^1(\Omega_\varepsilon)$ and $ i\in \{1, 2, 3\}.$

Define a nonlinear operator
$\mathcal{A}_\varepsilon(t): \mathcal{H}_\varepsilon \longrightarrow \mathcal{H}_\varepsilon^*$
through the relation
\begin{multline*}
     \big\langle \mathcal{A}_\varepsilon(t){u}, v \big\rangle_\varepsilon
 =   \int_{\Omega_\varepsilon}
     \nabla u \cdot \nabla v \, dx
 +   \int_{\Omega_\varepsilon} k(u) v \, dx
\\
 +   \varepsilon^{\alpha_0}
     \int_{\Gamma_\varepsilon^{(0)}} \kappa_0(u) v \, d\sigma_x
 +   \sum_{i=1}^3 \varepsilon^{\alpha_i}
     \int_{\Gamma_\varepsilon^{(i)}} \kappa_i(u, x_i, t) v \, d\sigma_x
\qquad \forall u,v\in\mathcal{H}_\varepsilon,
\end{multline*}
and the linear functional $\mathcal{F}_\varepsilon(t)\in\mathcal{H}_\varepsilon^*$
by the formula
$$
     \big\langle \mathcal{F}_\varepsilon(t), v \big\rangle_\varepsilon
 =   \int_{\Omega_\varepsilon} f \,v \, dx
 +   \sum_{i=0}^3 \varepsilon^{\beta_i}
     \int_{\Gamma_\varepsilon^{(i)}} \varphi_\varepsilon^{(i)} \, v \, d\sigma_x
\qquad \forall v\in \mathcal{H}_\varepsilon,
$$
for a.e. $t\in(0,T),$
where $\langle \cdot, \cdot \rangle_\varepsilon$ is the duality pairing of $\mathcal{H}_\varepsilon^*$
and $\mathcal{H}_\varepsilon$.

Then the integral identity (\ref{int-identity}) can be rewritten as follows
\begin{equation}\label{oper-equat}
     \big\langle {\partial_t u_\varepsilon}, v \big\rangle_\varepsilon
 +   \big\langle \mathcal{A}_\varepsilon(t){u_\varepsilon}, v \big\rangle_\varepsilon
 =   \big\langle \mathcal{F}_\varepsilon(t), v \big\rangle_\varepsilon
\qquad \forall v\in \mathcal{H}_\varepsilon,
\end{equation}
for a.e. $t\in(0, T),$ and $u_\varepsilon|_{t=0}=0.$

To prove the well-posedness result, we verify some properties of the operator $\mathcal{A}_\varepsilon$ for a fixed value of $\varepsilon.$
\begin{enumerate}
\item
 With the help of (\ref{kappa_ineq+}) and Cauchy's inequality with
 $\delta>0 \ (ab \leq \delta a^2 + \tfrac{b^2}{4\delta}),$
 we obtain
 $$
      \big\langle \mathcal{A}_\varepsilon(t){v}, v \big\rangle_\varepsilon
 $$
 $$
 \geq \int_{\Omega_\varepsilon} | \nabla v |^2 \,dx
 +    \int_{\Omega_\varepsilon} k(0) \, v \,dx
 +    \varepsilon^{\alpha_0}
      \int_{\Gamma_\varepsilon^{(0)}} \kappa_0(0) \, v \,d\sigma_x
 +    \sum\limits_{i=1}^3 \varepsilon^{\alpha_i}
      \int_{\Gamma_\varepsilon^{(i)}} \kappa_i(0, x_i, t) \, v \,d\sigma_x
 $$
 $$
 \geq \| v \|_\varepsilon^2
 -    \delta \left(
      \int_{\Omega_\varepsilon} v^2 \,dx
 +    \int_{\Gamma_\varepsilon^{(0)}} v^2 \,d\sigma_x
 +    \varepsilon \sum\limits_{i=1}^3
      \int_{\Gamma_\varepsilon^{(i)}} v^2 \,d\sigma_x
             \right)
 $$
\begin{equation}\label{coercc}
 -    \frac{1}{4\delta}
      \left(
       |k(0)|^2 \, |\Omega_\varepsilon|_3
 +     \varepsilon^{2\alpha_0}
       |\kappa_0(0)|^2 \, |\Gamma_\varepsilon^{(0)}|_2
 +     \sum\limits_{i=1}^3 \varepsilon^{2\alpha_i-1}
       \max\limits_{[0, \ell_i]\times[0,T]}
       |\kappa_i(0, x_i, t)|^2 \, |\Gamma_\varepsilon^{(i)}|_2
      \right).
\end{equation}
 Here and in what follows  $|S|_n$ is the $n$-dimensional Lebesgue measure of a set $S.$
 Then using (\ref{equivalent_norm}), (\ref{ineq1}), (\ref{ineq-3}) and recalling the assumption  ${\bf C3}$(a), we can select appropriate $\delta$ such that
 \begin{equation*}
   \big\langle \mathcal{A}_\varepsilon (t) v, v \big\rangle_\varepsilon
 \geq C_6 \| v \|_\varepsilon^2 - C_7 \varepsilon^2
      \bigg( 1+ \varepsilon^{2\alpha_0}+ \sum_{i=1}^3\varepsilon^{2\alpha_i-2}\bigg)
\quad \forall v\in\mathcal{H}_\varepsilon.
\end{equation*}
This inequality means that the operator $\mathcal{A}_\varepsilon$ is coercive for a.e. $t\in(0, T).$
 \item
 Let us show that it is strongly monotone for a.e. $t\in(0, T).$
 Taking into account (\ref{kappa_ineq}), we get
 $$
      \big\langle
       \mathcal{A}_\varepsilon(t){u_1} - \mathcal{A}_\varepsilon(t){u_2}, \ u_1 - u_2
      \big\rangle_\varepsilon
 \geq \| u_1 - u_2 \|_\varepsilon^2 \qquad \forall u_1, u_2\in\mathcal{H}_\varepsilon.
 $$
\item
 The operator $\mathcal{A}_\varepsilon$ is hemicontinuous for a.e. $t\in(0, T).$ Indeed, the real valued function
 $$
 [0,1] \ni \tau \to
 \big\langle \mathcal{A}_\varepsilon[u_1 + \tau v], u_2 \big\rangle_\varepsilon
 $$
 is continuous on $[0,1]$ for all fixed $u_1, u_2, v \in \mathcal{H}_\varepsilon$ due to the continuity
 of the functions $k, \ \{\kappa_i\}_{i=0}^3$ and Lebesque's dominated convergence theorem.
\item
 Let us prove that operator $\mathcal{A}_\varepsilon$ is bounded. Using Cauchy-Bunyakovsky integral
 inequality, (\ref{equivalent_norm}) and (\ref{kappa_ineq+}), we deduce the following inequality:
 $$
      \big\langle \mathcal{A}_\varepsilon{u}, v \big\rangle_\varepsilon
 \leq \int_{\Omega_\varepsilon} \nabla u \cdot \nabla v \,dx
 +    \int_{\Omega_\varepsilon}
      \big( k_+ \, |u| + |k(0)| \big) |v| \,dx
 $$
 $$
 +    \varepsilon^{\alpha_0}
      \int_{\Gamma_\varepsilon^{(0)}}
      \big( k_+ \, |u| + |\kappa_0(0)| \big) |v| \,d\sigma_x
 +    \sum\limits_{i=1}^3 \varepsilon^{\alpha_i}
      \int_{\Gamma_\varepsilon^{(i)}}
      \big( k_+ \, |u| + |\kappa_i(0, x_i, t)| \big) |v| \,d\sigma_x
 $$
 $$
 \leq \| u \|_\varepsilon \| v \|_\varepsilon
 +    k_+ \, \| u \|_{L^2(\Omega_\varepsilon)}
      \| v \|_{L^2(\Omega_\varepsilon)}
 +    k_+ \sum\limits_{i=0}^3 \varepsilon^{\alpha_i}
      \| u \|_{L^2(\Gamma_\varepsilon^{(i)})}
             \| v \|_{L^2(\Gamma_\varepsilon^{(i)})}
 +    |k(0)| \sqrt{|\Omega_\varepsilon|_3} \
      \| v \|_{L^2(\Omega_\varepsilon)}
 $$
 \begin{equation}\label{bound}
  +    \varepsilon^{\alpha_0}
      |\kappa_0(0)| \sqrt{\big|\Gamma_\varepsilon^{(0)}\big|_2 \, } \,
      \| v \|_{L^2(\Gamma_\varepsilon^{(0)})}
 +    \sum\limits_{i=1}^3 \varepsilon^{\alpha_i}
      \max\limits_{[0, \ell_i]\times[0,T]}
      |\kappa_i(0, x_i, t)| \sqrt{\big|\Gamma_\varepsilon^{(i)}\big|_2 \, } \,
      \| v \|_{L^2(\Gamma_\varepsilon^{(i)})}.
 \end{equation}
Now with the help of (\ref{ineq1}) and (\ref{ineq-3}), we obtain
 $$
      \big\langle \mathcal{A}_\varepsilon{u}, v \big\rangle_\varepsilon
 \leq C_8 \bigg( 1+ \varepsilon^{\alpha_0}+ \sum_{i=1}^3\varepsilon^{\alpha_i -1}\bigg)
      \Big( \varepsilon + \| u \|_\varepsilon \Big) \| v \|_\varepsilon
 \qquad \forall \, u,v\in\mathcal{H}_\varepsilon \quad \text{and a.e.} \ t\in(0, T).
 $$
\end{enumerate}

Thus, the existence and uniqueness of the weak solution for every fixed value $\varepsilon$
follow directly from Corollary 4.1
(see \cite[Chapter 3]{Showalter-book}).

\subsection{A priori estimates}\label{a-priori_es}
Taking into account (\ref{coercc}),  (\ref{ineq1}) and (\ref{ineq-3}),   we derive from (\ref{oper-equat})  that
\begin{gather*}
\frac12\int_{\Omega_\varepsilon} u_\varepsilon^2(x, \tau) \, dx
 + c_3 \int_0^\tau\|u\|_\varepsilon \, dt - c_4 \varepsilon^2 \tau \bigg(
    |k(0)|^2 + \varepsilon^{2\alpha_0}|\kappa_0(0)|^2
 +  \sum_{i=1}^3\varepsilon^{2\alpha_i-2}
    \max\limits_{[0, \ell_i]\times[0,T]} |\kappa_i(0, x_i, t)|^2
   \bigg)
\\
\le \delta \int_0^\tau\|u\|^2_\varepsilon \, dt
 + \frac{c_5}{4 \delta} \int_0^\tau\|f\|_{L^2(\Omega_\varepsilon)}^2 dt
 + \frac{\varepsilon^{2 \beta_0} c_6}{4 \delta}
   \int_0^\tau\|\varphi_\varepsilon^{(0)}\|_{L^2(\Gamma_\varepsilon^{(0)})}^2 dt
 + \sum_{i=1}^3  \frac{\varepsilon^{2 \beta_i -1} c_7}{4 \delta}
   \int_0^\tau\|\varphi_\varepsilon^{(i)}\|_{L^2(\Gamma_\varepsilon^{(i)})}^2 dt
\end{gather*}
for any $\tau \in (0, T].$ Selecting appropriate $\delta >0$ and taking the conditions $\bf{C1}$ -- $\bf{C3}(a)$ into account, we obtain the uniform estimate
\begin{gather}
\max_{t\in [0,T]} \|u_{\varepsilon}(\cdot,t)\|_{L^2(\Omega_\varepsilon)} + \|u_{\varepsilon}\|_{L^2(0,T; \mathcal{H}_\varepsilon)} \notag
\\
 \leq C_0
      \bigg(
       \sqrt{T} \varepsilon
       \Big(
       |k(0)| + \varepsilon^{\alpha_0}|\kappa_0(0)|
 +     \sum_{i=1}^3\varepsilon^{\alpha_i-1}
       \max\limits_{[0, \ell_i]\times[0,T]} |\kappa_i(0, x_i, t)|
       \Big) \notag
\\
 +     \, \|f\|_{L^2(\Omega_\varepsilon \times (0,T))}
 +     \varepsilon^{\beta_0}
       \|\varphi_\varepsilon^{(0)}\|_{L^2(\Gamma_\varepsilon^{(0)}\times (0,T))}
 +     \sum\limits_{i=1}^{3} \varepsilon^{\beta_i - \frac12}
       \|\varphi_\varepsilon^{(i)}\|_{L^2(\Gamma_\varepsilon^{(i)}\times (0,T))}
      \bigg)
\le   C_1 \varepsilon   \label{apr_ocin}
\end{gather}
for all values of the parameters $\{\alpha_i\}_{i=0}^3$ and $\beta_0 \ge 0,  \ \beta_i \ge 1, \ i\in \{1, 2, 3\}.$

Now let us consider the case   ${\bf C3}({\rm a})$ $(\alpha_0 < 0).$
From the integral identity (\ref{int-identity}) and inequalities (\ref{kappa_ineq+}),  (\ref{kappa_ineq++}), (\ref{ineq1}), (\ref{equivalent_norm}), (\ref{ineq-3}) and (\ref{apr_ocin})  it follows
\begin{multline*}
\varepsilon^{\alpha_0} \int_{\Gamma_\varepsilon^{(0)}\times(0,T)} u_\varepsilon^2\, d\sigma_x dt
\le
C_1
    \bigg(
     \sqrt{T} \varepsilon
     \Big(
      |k(0)|
 +    \sum_{i=1}^3\varepsilon^{\alpha_i-1}
      \max\limits_{[0, \ell_i]\times[0,T]} |\kappa_i(0, x_i, t)|
     \Big)
\\
\|f\|_{L^2(\Omega_\varepsilon \times (0,T))}
+  \varepsilon^{\beta_0}
     \|\varphi_\varepsilon^{(0)}\|_{L^2(\Gamma_\varepsilon^{(0)}\times (0,T))} +
\sum\limits_{i=1}^{3} \varepsilon^{\beta_i - \frac12} \|\varphi_\varepsilon^{(i)}\|_{L^2(\Gamma_\varepsilon^{(i)}\times (0,T)) }\Big)
\|u_\varepsilon\|_{L^2(0,T; \mathcal{H}_\varepsilon)}
\le C_2 \varepsilon^2.
\end{multline*}
Now with the help of (\ref{traceineq0}) we get
\begin{equation}\label{est-alpha_1}
 \int_{\Omega_\varepsilon^{(0)}\times (0,T)} u_\varepsilon^2\, dx dt \le
 C_3 \Bigg( \varepsilon^2 \int_{\Omega_\varepsilon^{(0)}\times(0,T)} |\nabla_{x}u_\varepsilon|^2 \, dx dt+
  \varepsilon^{1 -\alpha_0} \varepsilon^{\alpha_0} \int_{\Gamma_\varepsilon^{(0)}\times(0,T)} u_\varepsilon^2 \, d\sigma_x dt\Bigg) \le C_4 \varepsilon^\vartheta,
\end{equation}
where $\vartheta:= \min \{4, \ 3 - \alpha_0\}.$ This means that
\begin{equation}\label{est-alpha_0}
\frac{1}{\varepsilon^3} \int_{\Omega_\varepsilon^{(0)}\times(0,T)} u_\varepsilon^2\, dx dt \le C_5\, \varepsilon^{\min\{1,  - \alpha_0\}} \longrightarrow 0\quad \text{as} \ \ \varepsilon \to 0.
\end{equation}

\section{Formal asymptotic approximation. The case $\alpha_0 \ge 0,$ $\alpha_i \ge 1,$ $ i\in\{1, 2, 3\}.$}\label{Formal asymptotics}

In this section we assume that  the functions
$f, \ k, \ \{\varphi_{\varepsilon}^{(i)}, \ \kappa_i\}_{i=0}^3$ are smooth enough.
Following the approach of \cite{Mel_Klev_AA-2016},
we propose ansatzes of the asymptotic approximation  for the solution to the problem~(\ref{probl})
in the following form:
\begin{enumerate}
  \item
 the regular parts of the approximation
  \begin{equation}\label{regul}
    \omega_0^{(i)} (x_i, t)
 +  \varepsilon \omega_1^{(i)} (x_i, t)
 +  \varepsilon^{2} u_2^{(i)} \left( \dfrac{\overline{x}_i}{\varepsilon}, x_i, t \right)
 +  \varepsilon^{3} u_3^{(i)} \left( \dfrac{\overline{x}_i}{\varepsilon}, x_i, t \right)
\end{equation}
is located inside of each thin cylinder $\Omega^{(i)}_\varepsilon$ 
and their terms  depend both on the corresponding longitudinal variable $x_i$ and so-called
``fast variables'' $\dfrac{\overline{x}_i}{\varepsilon} \ (i=1,2,3);$
  \item
and the inner part of the approximation
  \begin{equation}\label{junc}
    N_0\left(\frac{x}{\varepsilon}, t \right)
 +  \varepsilon N_1\left(\frac{x}{\varepsilon}, t \right)
 +  \varepsilon^2 N_2\left(\frac{x}{\varepsilon}, t \right)
\end{equation}
is located in a neighborhood of the node $\Omega^{(0)}_\varepsilon$.
\end{enumerate}

\subsection{Regular parts}\label{regul_asymp}
Substituting the representation (\ref{regul}) for each fixed index $i \in \{1, 2, 3\}$
into the differential equation  of the problem~(\ref{probl}),
using  Taylor's formula for the function $f$
at the point $\overline{x}_i=(0,0)$ for the function $k$ at  $\omega_0^{(i)},$ and collecting coefficients at $\varepsilon^0$, we obtain
\begin{equation}\label{eq-1}
 -   \Delta_{\overline{\xi}_i}{u}_{2}^{(i)} (\overline{\xi}_i, x_i, t)
 =-  \frac{\partial\omega_0^{(i)}}{\partial{t}} (x_i, t)
 +   \frac{\partial^2\omega_0^{(i)}}{\partial{x_i}^2} (x_i, t)
 -   k\Big(\omega_0^{(i)}(x_i, t)\Big)
 +   f_0^{(i)}(x_i, t),
\end{equation}
where $\overline{\xi}_i = \dfrac{\overline{x}_i}{\varepsilon}$ and
$f_0^{(i)}(x_i, t):=  f(x, t)|_{\overline{x}_i=(0,0)}.$

It is easy to calculate the outer unit normal to $\Gamma^{(i)}_\varepsilon:$
$$
{\boldsymbol{\nu}}^{(i)} (x_i, \ \overline{\xi}_i)
 =  \dfrac{1}{\sqrt{1 + \varepsilon^2 |h_i^\prime (x_i)|^2 \,}}
    \big( -\varepsilon h_i^\prime (x_i), \ \overline{\nu}_i (\overline{\xi}_i) \big)
 =
\left\{\begin{array}{lr}
    \dfrac{ \big(
   -\varepsilon h_1^\prime (x_1), \
    {\nu}_2^{(1)} (\overline{\xi}_1), \
    {\nu}_3^{(1)} (\overline{\xi}_1)
    \big) }{\sqrt{1 + \varepsilon^2 |h_1^\prime (x_1)|^2 \,}},
    \ & i=1,
\\
    \dfrac{ \big(
    {\nu}_1^{(2)} (\overline{\xi}_2), \
   -\varepsilon h_2^\prime (x_2), \
    {\nu}_3^{(2)} (\overline{\xi}_2)
    \big) }{\sqrt{1 + \varepsilon^2 |h_2^\prime (x_2)|^2 \,}},
    \ & i=2,
\\
    \dfrac{ \big(
    {\nu}_1^{(3)} (\overline{\xi}_3), \
    {\nu}_2^{(3)} (\overline{\xi}_3), \
   -\varepsilon h_3^\prime (x_3)
    \big) }{\sqrt{1 + \varepsilon^2 |h_3^\prime (x_3)|^2 \,}},
    \ & i=3,
\end{array}\right.
$$
where $\overline{\nu}_i (\frac{\overline{x}_i}{\varepsilon})$ is the outward normal for the disk
${\Upsilon_{\varepsilon}^{(i)} (x_i)}:=\{ \overline{\xi}_i\in\Bbb{R}^2 \ : \ |\overline{\xi}_i|< h_i(x_i) \}.$

Taking the view of the outer unit normal into account and putting the sum (\ref{regul})
into  the third  relation of the problem (\ref{probl}), we get
with the help of Taylor's formula for the function $\kappa_i$ the following relation:
\begin{equation}\label{eq-2}
    \varepsilon
    \partial_{\overline{\nu}_i (\overline{\xi}_i)}{u}_{2}^{(i)} (\overline{\xi}_i, x_i, t)
 =  \varepsilon \,h_i^\prime(x_i) \, \frac{\partial\omega_{0}^{(i)}}{\partial{x_i}} (x_i, t)
 -  \varepsilon^{\alpha_i} \kappa_i\Big(\omega_0^{(i)}(x_i, t), x_i, t \Big)
 +  \varepsilon^{\beta_i} \varphi^{(i)}(\overline{\xi}_i, x_i, t).
\end{equation}

Relations (\ref{eq-1}) and (\ref{eq-2}) form the linear inhomogeneous Neumann  boundary-value problem
\begin{equation}\label{regul_probl_2}
\left\{\begin{array}{rcl}
 -  \Delta_{\overline{\xi}_i}{u}_{2}^{(i)} (\overline{\xi}_i, x_i, t)
 & = &
 -  \dfrac{\partial\omega_0^{(i)}}{\partial{t}} (x_i, t)
 +  \dfrac{\partial^2\omega_0^{(i)}}{\partial{x_i}^2} (x_i, t)
 -  k\Big(\omega_0^{(i)}(x_i, t)\Big)
 +  f_0^{(i)} (x_i, t),
 \qquad \overline{\xi}_i\in\Upsilon_i (x_i),
\\[2mm]
    \partial_{\boldsymbol{\nu}_{\overline{\xi}_i}}{u}_{2}^{(i)}(\overline{\xi}_i, x_i, t)
 & = &
    h_i^\prime (x_i)\dfrac{\partial\omega_0^{(i)}}{\partial{x_i}} (x_i, t)
 -  \delta_{\alpha_i, 1} \, \kappa_i\Big(\omega_0^{(i)}(x_i, t), x_i, t \Big)
 +  \delta_{\beta_i, 1} \, \varphi^{(i)} (\overline{\xi}_i, x_i, t),
\\
 &   & \hspace{9.31cm} \overline{\xi}_i\in\partial\Upsilon_i(x_i),
\\
    \langle u_2^{(i)} (\, \cdot \, , x_i, t) \rangle_{\Upsilon_i (x_i)}
 & = &
    0,
\end{array}\right.
\end{equation}
to define $u_2^{(i)}.$
Here $\langle u(\, \cdot \, , x_i, t) \rangle_{\Upsilon_i(x_i)} :=  \int_{\Upsilon_i(x_i)}u (\overline{\xi}_i, x_i, t)d{\overline{\xi}_i},$ \
the variables $(x_i, t)$ are regarded as parameters from $I_\varepsilon^{(i)} \times (0, T),$
where $I_\varepsilon^{(i)} :=\{x: \ x_i\in  (\varepsilon \ell_0, \, \ell_i), \ \overline{x_i}=(0,0)\}.$
We add the third relation in (\ref{regul_probl_2}) for the uniqueness of a solution.

Writing down the necessary and sufficient conditions for the solvability
of the problem (\ref{regul_probl_2}), we derive the differential equation
\begin{multline}\label{omega_probl_2}
    \pi h_i^2(x_i) \dfrac{\partial\omega_0^{(i)}}{\partial{t}} (x_i, t)
 -  \pi \dfrac{\partial}{\partial{x_i}}
    \left(
     h_i^2(x_i)\frac{\partial\omega_0^{(i)}}{\partial{x_i}}(x_i, t)
    \right)
 +  \pi h_i^2(x_i) k\Big(\omega_0^{(i)}(x_i, t)\Big)
 +  2\pi \, \delta_{\alpha_i, 1} \,  h_i(x_i) \kappa_i\Big(\omega_0^{(i)}(x_i, t), x_i, t \Big)
\\
 =  \pi h_i^2(x_i) f_0^{(i)}(x_i, t)
 +  \delta_{\beta_i, 1} \int\limits_{\partial\Upsilon_i(x_i)}
    \varphi^{(i)}(\overline{\xi}_i, x_i, t) \, dl_{\overline{\xi}_i},
\quad (x_i, t)\in I_\varepsilon^{(i)} \times (0, T),
\end{multline}
to define $\omega_0^{(i)}$ \ $(i\in \{1, 2, 3\}).$

Let $\omega_0^{(i)}$ be a solution of the differential equation (\ref{omega_probl_2})
(its existence will be proved in the subsection~\ref{Lim-prob}).
Thus, there exists a unique solution to the problem (\ref{regul_probl_2}) for each $i\in \{1, 2, 3\}.$

For determination of the coefficients $u_3^{(i)}, \  i=1,2,3,$ we similarly obtain the following problems:
\begin{equation}\label{regul_probl_3}
\left\{\begin{array}{rcl}
- \Delta_{\overline{\xi}_i}{u}_{3}^{(i)} (\overline{\xi}_i, x_i, t)
 & = &
 -  \dfrac{\partial\omega_1^{(i)}}{\partial{t}} (x_i, t)
 +  \dfrac{\partial^2\omega_1^{(i)}}{\partial{x_i}^2} (x_i, t)
 -  k^\prime\Big(\omega_0^{(i)}(x_i, t)\Big) \, \omega_1^{(i)}(x_i, t)
 +  f_1^{(i)} (\overline{\xi}_i, x_i, t),
\\[2mm]
 &   &
 \hspace{9cm} \overline{\xi}_i\in\Upsilon_i (x_i),
\\
  \partial_{\boldsymbol{\nu}_{\overline{\xi}_i}}{u}_{3}^{(i)}(\overline{\xi}_i, x_i, t)
 & = &
   h_i^\prime (x_i)\dfrac{d\omega_1^{(i)}}{d{x_i}} (x_i, t)
 - \delta_{\alpha_i, 1} \,
   \partial_s \kappa_i\Big(\omega_0^{(i)}(x_i, t), x_i, t \Big) \, \omega_1^{(i)}(x_i, t)
\\[2mm]
 &   &
 -  \, \delta_{\alpha_i, 2} \, \kappa_i\Big(\omega_0^{(i)}(x_i, t), x_i, t \Big)
 +  \delta_{\beta_i, 2} \, \varphi^{(i)} (\overline{\xi}_i, x_i, t),
    \hspace{1.58cm} \overline{\xi}_i\in\partial\Upsilon_i(x_i),
\\
   \langle u_3^{(i)} (\, \cdot \, , x_i, t) \rangle_{\Upsilon_i (x_i)}
 & = &
   0,
\end{array}\right.
\end{equation}
for each $i\in \{1, 2, 3\}.$ Here
\begin{equation*}
f_1^{(i)}(\overline{\xi}_i, x_i, t) =
    \sum\limits_{j=1}^3 (1-\delta_{ij}) \, \xi_j \, \frac{\partial }{\partial{x_j}}
    f(x, t)|_{\overline{x}_i=(0,0)}.
\end{equation*}

Repeating the previous reasoning, we find that the coefficients $\{\omega_1^{(i)}\}_{i=1}^3$
have to be solutions to the respective linear ordinary differential equation
\begin{multline}\label{omega_probl_3}
    \pi h_i^2(x_i) \dfrac{\partial\omega_1^{(i)}}{\partial{t}} (x_i, t)
 -  \pi \dfrac{\partial}{\partial{x_i}}
    \left(h_i^2(x_i)\frac{\partial\omega_1^{(i)}}{\partial{x_i}}(x_i, t)\right)
\\
 +  \pi h_i^2(x_i) k^\prime\Big(\omega_0^{(i)}(x_i, t)\Big) \,
    \omega_1^{(i)}(x_i, t)
 +  2\pi \, \delta_{\alpha_i, 1} \, h_i(x_i)
    \partial_s \kappa_i\Big(\omega_0^{(i)}(x_i, t), x_i, t\Big) \,
    \omega_1^{(i)}(x_i, t)
\\
 =  \int\limits_{\Upsilon_i(x_i)}f_1^{(i)}(\overline{\xi}_i, x_i, t) \, d{\overline{\xi}_i}
 -  2\pi \, \delta_{\alpha_i, 2} \,  h_i(x_i) \kappa_i\Big(\omega_0^{(i)}(x_i, t), x_i, t \Big)
\\
 +  \delta_{\beta_i, 2} \int\limits_{\partial\Upsilon_i(x_i)}
    \varphi^{(i)}(\overline{\xi}_i, x_i, t) \, dl_{\overline{\xi}_i},
\quad (x_i, t)\in I_\varepsilon^{(i)} \times (0, T) \quad \big(i\in \{1, 2, 3\}\big).
\end{multline}


 \subsection{Inner part}\label{inner_asymp}

To obtain conditions for the functions $\{\omega_n^{(i)}\}_{i=1}^3, \ n\in \{0, 1\}$
at the point $(0,0,0),$ we introduce the inner part of the asymptotic approximation~(\ref{junc})
in a neighborhood of the node $\Omega^{(0)}_\varepsilon$.
If we pass to the ``fast variables'' $\xi=\frac{x}{\varepsilon}$ and tend  $\varepsilon$ to $0,$
the domain $\Omega_\varepsilon$ is transformed into the unbounded domain $\Xi$
that  is the union of the domain~$\Xi^{(0)}$ and three semibounded cylinders
$$
\Xi^{(i)}
 =  \{ \xi=(\xi_1,\xi_2,\xi_3)\in\Bbb R^3 \ :
    \quad  \ell_0<\xi_i<+\infty,
    \quad |\overline{\xi}_i|<h_i(0) \},
\qquad i=1,2,3,
$$
i.e., $\Xi$ is the interior of $\bigcup_{i=0}^3\overline{\Xi^{(i)}}$ (see Fig.~\ref{f4}).

\begin{figure}[htbp]
\begin{center}
\includegraphics[width=6cm]{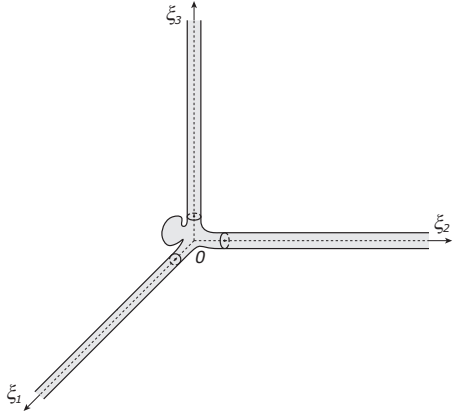}
\end{center}
\caption{The domain $\Xi$}\label{f4}
\end{figure}

Let us introduce the following notation for parts of the boundary of the domain $\Xi$:
$$
\Gamma_i = \{ \xi\in\Bbb R^3 \ : \quad \ell_0<\xi_i<+\infty, \quad |\overline{\xi}_i|=h_i(0)\}, \quad i=1,2,3,\quad \text{and} \quad
\Gamma_0 = \partial\Xi \backslash \Big(\bigcup_{i=1}^3 \Gamma_i \Big).
$$

Substituting (\ref{junc}) into the problem (\ref{probl}) and equating coefficients at the same powers
of $\varepsilon$, we derive the following relations for $N_n, \ (n\in \{0, 1, 2\}):$
\begin{equation}\label{junc_probl_n}
 \left\{\begin{array}{rcll}
  -\Delta_{\xi}{N_n}(\xi, t) & = &
   F_n(\xi, t), &
   \quad \xi\in\Xi,
\\[2mm]
   \partial_{{\boldsymbol \nu}_\xi}{N_n}(\xi, t) & = &
   B_{n}^{(0)}(\xi, t), &
   \quad \xi\in\Gamma_0,
\\[2mm]
   \partial_{{\boldsymbol \nu}_{\overline{\xi}_i}}{N_n}(\xi, t) & = &
   B_{n}^{(i)}(\xi, t), &
   \quad \xi\in\Gamma_i, \quad i=1,2,3,
\\[2mm]
   N_n(\xi, t) & \sim &
   \omega^{(i)}_{n}(0, t) + \Psi^{(i)}_{n}(\xi, t), &
   \quad \xi_i \to +\infty, \ \ \overline{\xi}_i \in \Upsilon_i(0), \quad i=1,2,3.
 \end{array}\right.
\end{equation}
Here
$$
F_0\equiv F_1\equiv 0, \quad
F_2(\xi, t) = - \, \partial_t N_0 - k(N_0) + f(0, t),  \quad
\xi\in\Xi,
$$
$$
B_{0}^{(0)} \equiv 0, \qquad
B_{1}^{(0)}
 = - \, \delta_{\alpha_0, 0} \, \kappa_0(N_0)
 +   \delta_{\beta_0, 0} \, \varphi^{(0)} (\xi, t),
$$
$$
B_{2}^{(0)}(\xi, t)
 = - \, \delta_{\alpha_0, 0} \, \kappa_0^{\prime}(N_0)N_1
   - \delta_{\alpha_0, 1} \, \kappa_0(N_0)
 +   \delta_{\beta_0, 1} \, \varphi^{(0)} (\xi, t),
\quad \xi\in\Gamma_0,
$$
$$
B_{0}^{(i)} \equiv B_{1}^{(i)} \equiv 0, \qquad
B_{2}^{(i)}(\xi, t)
 =-  \, \delta_{\alpha_i, 1} \, \kappa_i(N_0, 0, t)
 +   \delta_{\beta_i, 1} \, \varphi^{(i)} (\overline{\xi}_i, 0, t),
\quad \xi\in\Gamma_i, \ \ i=1,2,3.
$$
The variable $t$ is regarded as parameter from $(0, T).$
The right hand sides in the differential equation and boundary conditions on $\{\Gamma_i\}$
of the problem~ (\ref{junc_probl_n}) are obtained with the help of the Taylor's formula
for the functions $f, \ k$ and $\varphi^{(i)}, \ \kappa_0, \ \kappa_i$ at the points $x=0, \ s=N_0$ and $x_i=0, \ i=1,2,3,$ respectively.

The fourth condition in (\ref{junc_probl_n}) appears by matching the regular and inner asymptotics in a neighborhood of the node, namely the asymptotics of the terms $\{N_n\}$ as $\xi_i \to +\infty$ have to coincide with the corresponding asymptotics of  the terms
$\{\omega_{n}^{(i)}\}$ as $x_i =\varepsilon \xi_i \to +0, \ i=1,2,3,$ respectively.
Expanding formally each term of the regular asymptotics in the Taylor series at the points $x_i=0$
and collecting the coefficients of the same powers of $\varepsilon,$  we get
\begin{equation}\label{Psi_k}
\begin{array}{c}
\Psi_{0}^{(i)} \equiv 0, \qquad
\Psi_{1}^{(i)}(\xi, t)
 =   \xi_i\dfrac{\partial\omega_{0}^{(i)}}{\partial x_i}(0, t),
\quad  i=1,2,3,
\\[3mm]
\Psi_{2}^{(i)}(\xi, t)
 =   \dfrac{\xi_i^2}{2}\dfrac{\partial^2\omega_{0}^{(i)}}{\partial x_i^2}(0, t)
 +   \xi_i\dfrac{\partial\omega_{1}^{(i)}}{\partial x_i}(0, t)
 +   u_{2}^{(i)} (\overline{\xi}_i, 0, t),
\quad  i=1,2,3.
\end{array}
\end{equation}

A solution of the problem (\ref{junc_probl_n}) at $n=1,2$ is sought in the form
\begin{equation}\label{new-solution}
N_n(\xi, t) = \sum\limits_{i=1}^3 \Psi_{n}^{(i)}(\xi, t)\chi_i(\xi_i) + \widetilde{N}_n(\xi, t),
\end{equation}
where $ \chi_i \in C^{\infty}(\Bbb{R}_+),\ 0\leq \chi_i \leq1$ and
$$
\chi_i(\xi_i) =
\left\{\begin{array}{ll}
 0, & \text{if} \quad \xi_i \leq 1+\ell_0,
\\[2mm]
 1, & \text{if} \quad \xi_i \geq 2+\ell_0,
\end{array}\right. \qquad i=1,2,3.
$$
Then $\widetilde{N}_n$ has to be a  solution of the problem
\begin{equation}\label{junc_probl_general}
 \left\{\begin{array}{rcll}
  -\Delta_{\xi}{\widetilde{N}_n}(\xi, t) & = &
   \widetilde{F}_n(\xi, t),                  &
   \quad \xi\in\Xi,
\\[2mm]
   \partial_{\boldsymbol{\nu}_\xi}{\widetilde{N}_n}(\xi, t) & = &
   \widetilde{B}_{n}^{(0)}(\xi, t),                                           &
   \quad \xi\in\Gamma_0,
\\[2mm]
   \partial_{\boldsymbol{\nu}_{\overline{\xi}_i}}{\widetilde{N}_n}(\xi, t) & = &
   \widetilde{B}_{n}^{(i)}(\xi, t),                               &
   \quad \xi\in\Gamma_i, \quad i=1,2,3,
 \end{array}\right.
\end{equation}
where
$$
\widetilde{F}_1(\xi, t)
 =  \sum\limits_{i=1}^3
    \Big(
    \xi_i\dfrac{\partial\omega_0^{(i)}}{\partial x_i}(0, t) \chi_i^{\prime\prime}(\xi_i)
 + 2\dfrac{\partial\omega_0^{(i)}}{\partial x_i}(0, t) \chi_i^{\prime}(\xi_i)
    \Big),
$$
\begin{multline*}
\widetilde{F}_2(\xi, t)
 =  \sum\limits_{i=1}^3
    \Bigg[
    \bigg(
     \dfrac{\xi_i^2}{2}\dfrac{d^2\omega_{0}^{(i)}}{dx_i^2}(0, t)
 +   \xi_i\dfrac{\partial\omega_{1}^{(i)}}{\partial x_i}(0, t)
 +   u_{2}^{(i)} (\overline{\xi}_i, 0, t)
    \bigg)
    \chi_i^{\prime\prime}(\xi_i)
\\
 + 2\bigg(
     \xi_i\dfrac{\partial^2\omega_{0}^{(i)}}{\partial x_i^2}(0, t)
 +   \dfrac{\partial\omega_{1}^{(i)}}{\partial x_i}(0, t)
    \bigg)
    \chi_i^{\prime}(\xi_i)
    \Bigg]
\\
 -  \, \partial_t N_0
 -  k\big(N_0\big)
 +  \sum_{i=1}^3
    \left(
    \, \partial_t \omega_0^{(i)}(0, t)
 +  k \big(\omega_0^{(i)}(0, t)\big)
    \right) \chi_i(\xi_i)
 +  \bigg( 1 - \sum\limits_{i=1}^3 \chi_i(\xi_i) \bigg) f(0, t),
\end{multline*}
and
\begin{multline*}
\widetilde{B}_{1}^{(0)}
 =-   \delta_{\alpha_0, 0} \, \kappa_0(N_0)
 +   \delta_{\beta_0, 0} \, \varphi^{(0)} (\xi, t), \qquad
\widetilde{B}_{2}^{(0)}(\xi, t)
 = -  \delta_{\alpha_0, 0} \, \kappa_0^{\prime}(N_0)N_1
   - \delta_{\alpha_0, 1} \, \kappa_0(N_0)
 +   \delta_{\beta_0, 1} \, \varphi^{(0)} (\xi, t),
\\
\widetilde{B}_{1}^{(i)} \equiv 0, \qquad
\widetilde{B}_{2}^{(i)} (\xi, t)
 =- \, \delta_{\alpha_i, 1}
    \left(
     \kappa_i \big( N_0, 0, t \big)
 -   \kappa_i \big(\omega_0^{(i)}(0, t), 0, t \big) \chi_i(\xi_i)
    \right)
 +  \delta_{\beta_i, 1} \, \varphi^{(i)} (\overline{\xi}_i, 0, t)
    \big( 1-\chi_i(\xi_i) \big),
\end{multline*}
for $i\in\{1,2,3\}.$
In addition, we demand that $\widetilde{N}_n$ satisfies the following stabilization conditions:
\begin{equation}\label{junc_probl_general+cond}
   \widetilde{N}_n(\xi, t)  \rightarrow  \omega^{(i)}_{n}(0, t)
   \quad \text{as} \quad \xi_i \to +\infty, \ \ \overline{\xi}_i \in \Upsilon_i(0), \quad i=1,2,3.
\end{equation}

The existence of a solution to the  problem (\ref{junc_probl_general})
in the corresponding energetic space can be obtained from general results
about the asymptotic behavior of solutions to elliptic problems in domains
with different exits to infinity (see e.g. \cite{Ko-Ol,Naz96}).
We will use approach proposed in \cite{Naz96,ZAA99}.

Let $C^{\infty}_{0,\xi}(\overline{\Xi})$ be a space of functions infinitely differentiable in $\overline{\Xi}$ and finite with
respect to  $\xi$, i.e.,
$$
\forall \,v\in C^{\infty}_{0,\xi}(\overline{\Xi}) \quad \exists \,R>0 \quad \forall \, \xi\in\overline{\Xi} \quad \xi_i \geq R, \ \ i=1,2,3 \, : \quad v(\xi)=0.
$$
We now define a  space  $\mathcal{H} := \overline{\left( C^{\infty}_{0,\xi}(\overline{\Xi}), \ \| \cdot \|_\mathcal{H} \right)}$, where
$$
\| v \|_\mathcal{H}
 =  \sqrt{\int_\Xi|\nabla v(\xi)|^2 \, d\xi + \int_\Xi |v(\xi)|^2 |\rho(\xi)|^2 \, d\xi \, } ,
$$
and the weight function  $ \rho \in C^{\infty}(\Bbb{R}^3),\ 0\leq \rho \leq1$ and
$$
\rho (\xi) =
\left\{\begin{array}{ll}
1,            & \mbox{if} \quad                        \xi \in \Xi^{(0)}, \\
|\xi_i|^{-1}, & \mbox{if} \quad \xi_i \geq \ell_0+1, \ \xi \in \Xi^{(i)}, \quad i=1,2,3.
\end{array}\right.
$$

\begin{definition}
A function $\widetilde{N}_n$ from the space $\mathcal{H}$ is called a weak solution of the problem
(\ref{junc_probl_general}) if the identity
\begin{equation*}
    \int_{\Xi} \nabla \widetilde{N}_n \cdot \nabla v \, d\xi =  \int_{\Xi} \widetilde{F}_n \, v \, d\xi
 +  \sum_{i=0}^3 \ \int\limits_{\Gamma_i}  \widetilde{B}^{(i)}_n \, v \, d\sigma_\xi
\end{equation*}
holds for all $v\in\mathcal{H}$.
\end{definition}

Similarly as in \cite{ZAA99}, we prove the following proposition.
\begin{proposition}\label{tverd1}
   Let $\rho^{-1} \widetilde{F}_n (\cdot\,, t) \in L^2(\Xi), \
                  \widetilde{B}^{(0)}_{n} (\cdot\,, t) \in L^2(\Gamma_0), \
        \rho^{-1} \widetilde{B}^{(i)}_{n} (\cdot\,, t) \in L^2(\Gamma_i), \quad i=1,2,3,$ for a.e. $t\in (0, T).$
Then there exist a weak solution of problem (\ref{junc_probl_general}) if and only if
\begin{equation}\label{solvability}
    \int_{\Xi} \widetilde{F}_n \, d\xi
 +  \sum_{i=0}^3 \, \int_{\Gamma_i}
    \widetilde{B}^{(i)}_n \, d\sigma_\xi
 =  0.
\end{equation}
This solution is defined up to an additive constant.
The additive constant  can be chosen to guarantee the existence and uniqueness of a weak solution
of problem (\ref{junc_probl_general}) with the following differentiable asymptotics:
\begin{equation}\label{inner_asympt_general}
\widehat{N}_n(\xi,t)=\left\{
\begin{array}{rl}
    {\mathcal O}(\exp( - \gamma_1 \xi_1)) & \mbox{as} \ \ \xi_1\to+\infty,
\\[2mm]
    \boldsymbol{\delta_n^{(2)}}(t)
 +  {\mathcal O}(\exp( - \gamma_2 \xi_2)) & \mbox{as} \ \ \xi_2\to+\infty,
\\[2mm]
    \boldsymbol{\delta_n^{(3)}}(t)
 +  {\mathcal O}(\exp( - \gamma_3 \xi_3)) & \mbox{as} \ \ \xi_3\to+\infty,
\end{array}
\right.
\end{equation}
where $\gamma_i, \ i=1,2,3 $ are positive constants.
\end{proposition}

The values $\boldsymbol{\delta_n^{(2)}}$ and $\boldsymbol{\delta_n^{(3)}}$
in (\ref{inner_asympt_general}) are defined as follows:
\begin{equation}\label{const_d_0}
\boldsymbol{\delta_n^{(i)}} (t)
 =  \int_{\Xi} \mathfrak{N}_i(\xi) \, \widetilde{F}_n(\xi, t) \, d\xi
 +  \sum\limits_{j=0}^3 \,
    \int_{\Gamma_j} \mathfrak{N}_i(\xi) \, \widetilde{B}_n^{(j)}(\xi, t) \, d\sigma_\xi ,
\quad i=2,3, \ \  n\in\{0, 1, 2\},
\end{equation}
where $\mathfrak{N}_2$ and $\mathfrak{N}_3$ are special solutions to
the corresponding homogeneous problem
\begin{equation}\label{hom_probl}
  -\Delta_{\xi}\mathfrak{N} = 0 \ \ \text{in} \ \ \Xi, \qquad
  \partial_\nu \mathfrak{N} = 0 \ \ \text{on} \ \ \partial \Xi,
\end{equation}
for the problem (\ref{junc_probl_general}).

\begin{proposition}\label{tverd2}
The  problem (\ref{hom_probl}) has two linearly independent solutions
$\mathfrak{N}_2$ and $\mathfrak{N}_3$ that do not belong to the space
${\mathcal H}$ and they have the following differentiable asymptotics:
\begin{equation}\label{inner_asympt_hom_solution_1}
\mathfrak{N}_2(\xi) = \left\{
\begin{array}{rl}
    - \dfrac{\xi_1}{\pi h_1^2(0)}
 +  {\mathcal O}(\exp( - \gamma_1 \xi_1)) & \mbox{as} \ \ \xi_1\to+\infty,
 \\[3mm]
   \dfrac{\xi_2}{\pi h_2^2(0)} +   C_2^{(2)}
 +  {\mathcal O}(\exp( - \gamma_2 \xi_2)) & \mbox{as} \ \ \xi_2\to+\infty,
 \\[3mm]
    C_2^{(3)} + {\mathcal O}(\exp( - \gamma_3 \xi_3)) & \mbox{as} \ \ \xi_3\to+\infty,
\end{array}
\right.
\end{equation}
\begin{equation}\label{inner_asympt_hom_solution_2}
\mathfrak{N}_3(\xi) = \left\{
\begin{array}{rl}
     - \dfrac{\xi_1}{\pi h_1^2(0)}
 +     {\mathcal O}(\exp( - \gamma_1 \xi_1)) & \mbox{as} \ \ \xi_1\to+\infty,
 \\[3mm]
   C_3^{(2)} +  {\mathcal O}(\exp( - \gamma_2 \xi_2)) & \mbox{as} \ \ \xi_2\to+\infty,
 \\
    \dfrac{\xi_3}{\pi h_3^2(0)} + C_3^{(3)}
 +   {\mathcal O}(\exp( - \gamma_3 \xi_3)) & \mbox{as} \ \ \xi_3\to+\infty,
\end{array}
\right.
\end{equation}

Any other solution to the homogeneous problem, which has polynomial growth at infinity,
can be presented as a linear combination
$\mathfrak{c}_1 + \mathfrak{c}_2 \mathfrak{N}_2 + \mathfrak{c}_3 \mathfrak{N}_3.$
\end{proposition}
\begin{proof}
The solution $\mathfrak{N}_2$ is sought in the form of a sum
$$
\mathfrak{N}_2 (\xi) = - \dfrac{\xi_1}{\pi h_1^2(0)} \, \chi_1(\xi_1) + \dfrac{\xi_2}{\pi h_2^2(0)} \, \chi_2(\xi_2)  +  \widetilde{\mathfrak{N}}_2(\xi),
$$
where $\widetilde{\mathfrak{N}}_2 \in {\mathcal H}$ and $\widetilde{\mathfrak{N}}_2$
is the solution to the problem (\ref{junc_probl_general}) with right-hand sides
$$
\widetilde{F}_2^*(\xi) =
\left\{
\begin{array}{rl}
    \dfrac{1}{\pi h_1^2(0)}\left(\big(\xi_1 \, \chi_1^{\prime}(\xi_1)\big)' + \chi_1^{\prime}(\xi_1)\right),
  & \ \ \xi \in  \Xi^{(1)},
 \\[3mm]
  -\dfrac{1}{\pi h_2^2(0)}\left(\big(\xi_2 \, \chi_2^{\prime}(\xi_2)\big)' + \chi_2^{\prime}(\xi_2)\right),
  &  \ \ \xi \in  \Xi^{(2)},
 \\[2mm]
  0 \, ,
  & \ \ \xi \in \Xi^{(0)} \cup \Xi^{(3)}.
\end{array}
\right.
$$
It is easy to verify that the solvability condition (\ref{solvability}) is satisfied.
Thus, by virtue of Proposition \ref{tverd1} there exist a unique solution $\widetilde{\mathfrak{N}}_2 \in {\mathcal H}$  that has the asymptotics
$$
\widetilde{\mathfrak{N}}_2(\xi)
 =  (1- \delta_{j 1}) C_2^{(j)}
 +  {\mathcal O}(\exp( - \gamma_j \xi_j))
\quad \mbox{as} \ \ \xi_j\to+\infty, \qquad j=1,2,3.
$$

Similar we can prove the existence of the solution $\mathfrak{N}_3$ with the asymptotics (\ref{inner_asympt_hom_solution_2}).

Obviously, that  $\mathfrak{N}_2$ and $\mathfrak{N}_3$ are linearly independent and
any other solution to the homogeneous problem, which has polynomial growth at infinity, can be presented as
$\mathfrak{c}_1 + \mathfrak{c}_2 \mathfrak{N}_2 + \mathfrak{c}_3 \mathfrak{N}_3.$
\end{proof}

\begin{remark}\label{remark_constant}
To obtain formulas (\ref{const_d_0})  it is necessary to substitute
the functions $\widehat{N}_n, \mathfrak{N}_2$ and $\widehat{N}_n, \mathfrak{N}_3$
in  the second Green-Ostrogradsky formula
$$
   \int_{\Xi_R} \big(\widehat{N} \, \Delta_\xi \mathfrak{N} - \mathfrak{N} \, \Delta_\xi \widehat{N}\big)\, d\xi
 = \int_{\partial \Xi_R}
   \big(\widehat{N} \,\partial_{\nu_\xi} \mathfrak{N} - \mathfrak{N}\, \partial_{\nu_\xi}\widehat{N}\big)\, d\sigma_\xi
$$
respectively, and then pass to the limit as $R \to +\infty.$ Here $\Xi_R = \Xi \cap \{\xi : \ |\xi_i| < R, \ i=1, 2, 3\}.$
\end{remark}

\subsubsection{Limit problem}\label{Lim-prob}
The problem (\ref{junc_probl_n})
at $n=0$ is as follows:
\begin{equation*}
 \left\{\begin{array}{rcll}
  -\Delta_{\xi}{{N}_0}(\xi, t) & = &
   0,                  &
   \quad \xi\in\Xi,
\\[2mm]
   \partial_{\boldsymbol{\nu}_\xi}{{N}_0}(\xi, t) & = &
   0,                                           &
   \quad \xi\in\Gamma_0,
\\[2mm]
   \partial_{\boldsymbol{\nu}_{\overline{\xi}_i}}{{N}_0}(\xi, t) & = &
   0,                               &
   \quad \xi\in\Gamma_i, \quad i=1,2,3,
\\[2mm]
   {N}_0(\xi, t)                                                   & \longrightarrow&
   \omega^{(i)}_{0}(0, t),                                                          &
   \quad \xi_i \to +\infty, \ \ \overline{\xi}_i \in \Upsilon_i(0), \quad i=1,2,3.
 \end{array}\right.
\end{equation*}
It is ease to verify that $\boldsymbol{\delta_0^{(2)}}=\boldsymbol{\delta_0^{(3)}}=0$ and $\widehat{N}_0 \equiv 0.$
Thus, this problem has a solution in $\mathcal{H}$ if and only if
\begin{equation}\label{trans0}
\omega_0^{(1)} (0, t) = \omega_0^{(2)} (0, t) = \omega_0^{(3)} (0, t);
\end{equation}
in this case  \ $N_0 = \widetilde{N}_0 = \omega_0^{(1)} (0, t).$

In the problem (\ref{junc_probl_general}) at $n=1$ the solvability condition (\ref{solvability})
reads as follows:
\begin{equation}\label{transmisiont0}
  \pi h_1^2 (0) \frac{\partial\omega_{0}^{(1)}}{\partial x_1} (0, t)
 +  \pi h_2^2 (0) \frac{\partial\omega_{0}^{(2)}}{\partial x_2} (0, t)
 + \pi h_3^2 (0) \frac{\partial\omega_{0}^{(3)}}{\partial x_3} (0, t)
 -  \delta_{\alpha_0, 0} \big|\Gamma_0\big|_2 \kappa_0 \big(\omega_0^{(1)}(0, t)\big)
 = - \boldsymbol{d_0^*}(t),
\end{equation}
where
$$
\boldsymbol{d_0^*}(t)
 =   \delta_{\beta_0, 0} \int_{\Gamma_{0}}
     \varphi^{(0)}(\xi, t) \, d\sigma_{\xi}.
$$

Substituting (\ref{regul}) into the forth condition in (\ref{probl})
and neglecting terms of order of ${\mathcal O}(\varepsilon),$
we arrive to the following boundary conditions:
\begin{equation}\label{bv_left}
\omega_{0}^{(i)}(\ell_i, t)=0, \quad i=1,2,3.
\end{equation}

Thus, taking into account (\ref{omega_probl_2}), (\ref{trans0}), (\ref{transmisiont0})
and (\ref{bv_left}), we obtain for $\{\omega_0^{(i)}\}_{i=1}^3$ the following semi-linear problem:
\begin{equation}\label{main}
 \left\{\begin{array}{l}
    \pi h_i^2(x_i) \dfrac{\partial\omega_0^{(i)}}{\partial{t}} (x_i, t)
  - \pi \dfrac{\partial}{\partial{x_i}}
    \left(h_i^2(x_i)\dfrac{\partial\omega_0^{(i)}}{\partial{x_i}}(x_i, t)\right)
  + \pi h_i^2(x_i) \, k\Big(\omega_0^{(i)}(x_i, t)\Big)
 \\[5mm]
 \begin{array}{rclrl}
  + \ 2\pi \delta_{\alpha_i, 1} h_i(x_i)
    \kappa_i\Big(\omega_0^{(i)}(x_i, t), x_i, t \Big)
    & = &
    \widehat{F}_0^{(i)}(x_i, t), \ & (x_i, t) \in I_i\times(0, T), & i=1,2,3,
 \\[3mm]
    \omega_0^{(i)}(\ell_i, t)
    & = &
    0, & t\in(0, T), & i=1,2,3,
 \\[3mm]
    \omega_0^{(1)} (0, t) \hspace{0.26cm} = \hspace{0.26cm} \omega_0^{(2)} (0, t)
    & = &
    \omega_0^{(3)} (0, t), & t\in(0, T), &
 \\[3mm]
   \sum\limits_{i=1}^3 \pi h_i^2 (0) \dfrac{\partial\omega_{0}^{(i)}}{\partial{x_i}} (0, t)
  - \delta_{\alpha_0, 0} \big|\Gamma_0\big|_2 \kappa_0 \big(\omega_0^{(1)}(0, t)\big) & = &
   - \boldsymbol{d_0^*}(t), & t\in(0, T), &
 \\[3mm]
    \omega_{0}^{(i)} (x_i, 0)
    & = &
    0, & x_i\in I_i, & i=1,2,3,
 \end{array}
 \end{array}\right.
\end{equation}
where $I_i:=\{x: \ x_i\in (0, \ell_i), \ \overline{x}_i=(0,0)\}$ and
\begin{equation}\label{right-hand-side}
\widehat{F}_0^{(i)}(x_i, t)
 := \pi h_i^2(x_i) \, f(x, t) \, \big|_{\overline{x}_i=(0,0)}
 +  \delta_{\beta_i, 1} \int\limits_{\partial\Upsilon_i(x_i)}
    \varphi^{(i)} (\overline{\xi}_i, x_i, t) \, dl_{\overline{\xi}_i},
\quad \ x\in I_i.
\end{equation}
The problem (\ref{main}) is  called {\it the limit problem} for problem (\ref{probl}).

For functions
$$
\widetilde{\phi}(x)=\left\{
                     \begin{array}{ll}
                      \phi^{(1)}(x_1), & \hbox{if} \ \ x_1 \in I_1, \\
                      \phi^{(2)}(x_2), & \hbox{if} \ \ x_2 \in I_2, \\
                      \phi^{(3)}(x_3), & \hbox{if} \ \ x_3 \in I_3,
                     \end{array}
                    \right.
$$
defined on the graph $\mathcal{I} := \overline{I_1}\cup \overline{I_2} \cup \overline{I_3} ,$
we introduce the Sobolev space
$$
{\mathcal H}_0 := \left\{ \widetilde{\phi}: \
\phi^{(i)} \in H^1(I_i), \ \ \phi^{(i)}(\ell_i) = 0, \ \ i=1, 2, 3, \ \ \hbox{and} \ \
\phi^{(1)}(0) = \phi^{(3)}(0) = \phi^{(3)}(0)
\right\}
$$
with the scalar product
$$
( \widetilde{\phi}, \widetilde{\psi})_0
 := \sum_{i=1}^3 \pi \int_0^{\ell_i} h_i^2(x_i)\,
    \frac{d\phi^{(i)}}{d{x_i}} \, \frac{d\psi^{(i)}}{d{x_i}}\, dx_i,
\qquad \widetilde{\phi},\, \widetilde{\psi} \in {\mathcal H}_0.
$$

\begin{definition}
A function $\widetilde{\omega}\in L^2(0, T; {\mathcal H}_0),$
with $\widetilde{\omega}^\prime\in L^2(0, T; {\mathcal H}_0^*),$
is called a weak solution to the problem (\ref{main}) if it satisfies the integral identity
\begin{multline}\label{int-identity_main}
   \pi \sum_{i=1}^3
   \int_0^{\ell_i} h_i^2(x_i) \partial_t \omega^{(i)}(x_i, t) \psi^{(i)}(x_i) \,dx_i
 + ( \widetilde{\omega}, \widetilde{\psi} )_0
 + \delta_{\alpha_0, 0} \big|\Gamma_0\big|_2 \,
   \kappa_0 \big(\omega_0^{(1)}(0, t)\big) \, \psi^{(1)}(0)
\\
 + \sum_{i=1}^3
   \Big(
    \pi \int_0^{\ell_i} h_i^2(x_i) \, k(\omega^{(i)}(x_i, t)) \psi^{(i)}(x_i) \,dx_i
 +  2\pi \delta_{\alpha_i, 1}
    \int_0^{\ell_i} h_i(x_i) \, \kappa_i(\omega^{(i)}(x_i, t), x_i, t) \psi^{(i)}(x_i) \,dx_i
   \Big)
\\
 = \boldsymbol{d_0^*}(t) \, \psi^{(1)}(0)
 + \sum_{i=1}^3
   \int_0^{\ell_i} \widehat{F}_0^{(i)}(x_i, t) \, \psi^{(i)}(x_i)\, dx_i
\end{multline}
for any function $\widetilde{\psi}\in \mathcal{H}_0$ and a.e. $t\in(0, T),$
and $\widetilde{\omega}|_{t=0}=0.$
\end{definition}

Similarly as was done in Section~\ref{exis-uniq},
the integral identity (\ref{int-identity_main}) can be rewritten as follows
\begin{equation*}
     \big\langle \partial_t \widetilde{\omega}, \widetilde{\psi} \big\rangle_0
 +   \big\langle \mathcal{A}_0 (t) \, \widetilde{\omega}, \widetilde{\psi} \big\rangle_0
 =   \big\langle \mathcal{F}_0 (t), \widetilde{\psi} \big\rangle_0,
\end{equation*}
for all $\widetilde{\psi} \in {\mathcal H}_0$ and
a.e. $t\in(0, T),$ and $\widetilde{\omega}|_{t=0}=0.$
Here the nonlinear operator
$\mathcal{A}_0(t) : \mathcal{H}_0  \mapsto \mathcal{H}_0^*$
is defined through the relation
\begin{multline*}
   \big\langle \mathcal{A}_0 (t) \phi^{(i)}, \psi^{(i)} \big\rangle_0
 = ( \widetilde{\phi}, \widetilde{\psi} )_0
 + \delta_{\alpha_0, 0} \big|\Gamma_0\big|_2 \,
   \kappa_0 \big(\omega_0^{(1)}(0, t)\big) \, \psi^{(1)}(0)
\\
 + \sum_{i=1}^3 \left(
    \pi \int_0^{\ell_i} h_i^2 \, k(\phi^{(i)}) \psi^{(i)} \,dx_i
 +  2\pi \, \delta_{\alpha_i, 1} \int_0^{\ell_i}
    h_i \, \kappa_i(\phi^{(i)}, x_i, t) \psi^{(i)} \,dx_i
   \right)
\end{multline*}
for all $\widetilde{\phi}, \widetilde{\psi} \in {\mathcal H}_0,$
and the linear functional $\mathcal{F}_0(t)\in\mathcal{H}_0^*$ is defined by
$$
   \big\langle \mathcal{F}_0 (t), \widetilde{\psi} \big\rangle_0
 = \boldsymbol{d_0^*}(t) \, \psi^{(1)}(0)
 + \sum_{i=1}^3
   \int_0^{\ell_i} \widehat{F}_0^{(i)} \, \psi^{(i)}\, dx_i
\qquad \forall \, \widetilde{\psi} \in {\mathcal H}_0,
$$
where $\langle \cdot, \cdot \rangle_0$ is the duality pairing of the dual space
$\mathcal{H}_0^*$ and $\mathcal{H}_0$.

Using (\ref{kappa_ineq}) and (\ref{kappa_ineq+}), we can prove that the operator $\mathcal{A}_0$
is bounded, strongly monotone, hemicontinuous and coercive.
As a result, the existence and uniqueness of the weak solution to the problem (\ref{main})
follow directly from Corollary 4.1 (see \cite[Chapter 3]{Showalter-book}).

\subsubsection{Problem for $\{\widetilde{\omega}_1\}$}

Let us verify the solvability condition  (\ref{solvability})
for the problem (\ref{junc_probl_general}) at $n=2$. Knowing that $N_0 \equiv \omega_0^{(1)} (0, t)$
and taking into account  the third relation in problem (\ref{regul_probl_2}),
the equality  (\ref{solvability}) can be re-written as follows:
\begin{multline*}
    \sum\limits_{i=1}^3
    \Bigg[ \,
    \pi h_i^2 (0) \, \int\limits_{\ell_0+1}^{\ell_0+2}
    \bigg(
     \xi_i\dfrac{\partial^2\omega_{0}^{(i)}}{\partial{x_i}^2}(0, t)
 +   \dfrac{\partial\omega_{1}^{(i)}}{\partial{x_i}}(0, t)
    \bigg)
    \chi_i^{\prime}(\xi_i) \, d\xi_i
\\
 -  \int\limits_{\ell_0}^{\ell_0+2} (1-\chi_i(\xi_i))
    \int\limits_{\Upsilon_i(0)}
    \left(
     \, \partial_t \omega_0^{(1)}(0, t) + k \big(\omega_0^{(1)}(0, t)\big) - f(0, t)
    \right) \, d\overline{\xi}_i \, d\xi_i
\\
 -  \int\limits_{\ell_0}^{\ell_0+2} (1-\chi_i(\xi_i))
    \int\limits_{\partial\Upsilon_i (0)}
    \left(
     \delta_{\alpha_i, 1} \, \kappa_i \big(\omega_0^{(1)}(0, t), 0, t \big)
 -   \delta_{\beta_i, 1} \, \varphi^{(i)} (\overline{\xi}_i, 0, t)
    \right)
    \, dl_{\overline{\xi}_i}  \, d\xi_i
    \Bigg]
\\
 -  \delta_{\alpha_0, 0} \int\limits_{\Gamma_{0}}
    \kappa_0^{\prime}\big(\omega_0^{(1)}(0, t)\big)N_1(\xi, t) \, d\sigma_{\xi}
 -  \delta_{\alpha_0, 1} \int\limits_{\Gamma_{0}}
    \kappa_0 \big(\omega_0^{(1)}(0, t)\big) \, d\sigma_{\xi}
 +  \delta_{\beta_0, 1} \int\limits_{\Gamma_{0}}
    \varphi^{(0)}(\xi, t) \, d\sigma_{\xi}
\\
 -  \int\limits_{\Xi^{(0)}}
    \left(
     \, \partial_t \omega_0^{(1)}(0, t) + k \big(\omega_0^{(1)}(0, t)\big) - f(0, t)
    \right) \, d\xi
 =  \ 0.
\end{multline*}
Whence, integrating by parts in the first three integrals with regard to (\ref{omega_probl_2}),
we obtain the following relations for $\{\omega_1^{(i)}\}_{i=1}^3:$
\begin{equation}\label{transmisiont1}
    \sum\limits_{i=1}^3 \pi h_i^2 (0) \frac{\partial\omega_{1}^{(i)}}{\partial{x_i}} (0, t)
 =  \boldsymbol{d_{1}^*}(t),
\end{equation}
where
\begin{multline}\label{const_d_*}
\boldsymbol{d_1^*}(t)
 =  - \ell_0 \sum\limits_{i=1}^3
    \Bigg(
     \pi h_i^2(0)
     \left(
      \, \partial_t \omega_0^{(i)}(0, t) + k \big(\omega_0^{(1)}(0, t)\big) - f(0, t)
     \right)
\\
 +   2 \pi \, \delta_{\alpha_i, 1} \, h_i(0)\, \kappa_i \big(\omega_0^{(1)}(0, t), 0, t \big)
 -   \delta_{\beta_i, 1} \int\limits_{\partial\Upsilon_i(0)}
     \varphi^{(i)} (\overline{\xi}_i, 0, t)
     \, dl_{\overline{\xi}_i}
    \Bigg)
 + \delta_{\alpha_0, 0} \, \kappa_0^{\prime}\big(\omega_0^{(1)}(0, t)\big)
    \int\limits_{\Gamma_{0}} N_1(\xi, t) \, d\sigma_{\xi}
\\
 +  \delta_{\alpha_0, 1} \big|\Gamma_0\big|_2 \kappa_0 \big(\omega_0^{(1)}(0, t)\big)
 -  \delta_{\beta_0, 1} \int\limits_{\Gamma_{0}}
    \varphi^{(0)}(\xi, t) \, d\sigma_{\xi}
 +  \big|\Xi^{(0)}\big|_3
    \left(
     \, \partial_t \omega_0^{(1)}(0, t) + k \big(\omega_0^{(1)}(0, t)\big) - f(0, t)
    \right).
\end{multline}

Hence, if the functions $\{\omega_{1}^{(i)}\}_{i=1}^3$ satisfy (\ref{transmisiont1}),
then there exist a weak solution $\widetilde{N}_2$ of the problem~(\ref{junc_probl_general}).
According to Proposition  \ref{tverd1}, it can be chosen in a unique way to guarantee the
asymptotics~(\ref{inner_asympt_general}).

It remains to satisfy the stabilization conditions (\ref{junc_probl_general+cond}) at $n=1$.
For this, we represent a weak solution of the problem (\ref{junc_probl_general}) in the following form:
$$
\widetilde{N}_1 (\xi, t) = \omega_1^{(1)}(0, t) + \widehat{N}_1 (\xi, t), \quad \xi \in \Xi^{(1)}.
$$
Taking into account the asymptotics (\ref{inner_asympt_general}), we have to put
\begin{equation}\label{trans1}
    \omega_1^{(1)}(0, t)
 =  \omega_1^{(2)}(0, t) - \boldsymbol{\delta_1^{(2)}}(t)
 =  \omega_1^{(3)}(0, t) - \boldsymbol{\delta_1^{(3)}}(t).
\end{equation}
As a result, we get the solution of the problem (\ref{junc_probl_n}) with the following asymptotics:
\begin{equation}\label{inner_asympt}
{N}_{1}(\xi, t)
 =  \omega_{1}^{(i)} (0, t) + \Psi_1^{(i)} (\xi, t)
 +  {\mathcal O}(\exp(-\gamma_i\xi_i))
\quad \mbox{as} \ \ \xi_i\to+\infty, \quad i=1,2,3.
\end{equation}

Let us denote by
$$
G_1(\xi, t)
 := \omega_{1}^{(i)}(0, t) + \Psi_1^{(i)}(\xi, t), \quad (\xi, t)\in\Xi^{(i)}\times(0, T),
\quad i=1,2,3.
$$
\begin{remark}\label{rem_exp-decrease}
Due to (\ref{inner_asympt}),
the function ${N}_{1} - G_1$ are exponentially decrease as $\xi_i \to +\infty,$ $i=1,2,3.$
\end{remark}

Relations (\ref{trans1}) and (\ref{transmisiont1})
are the first and second transmission conditions for $\{\omega_1^{(i)}\}_{i=1}^3$ at $x=0.$
Thus, the second term of the regular asymptotics $\widetilde{\omega}_1$ is determined from the linear problem
\begin{equation}\label{omega_probl*}
 \left\{\begin{array}{l}
    \pi h_i^2(x_i) \dfrac{\partial\omega_1^{(i)}}{\partial{t}} (x_i, t)
  - \pi \dfrac{\partial}{\partial{x_i}}
    \left(h_i^2(x_i)\dfrac{\partial\omega_1^{(i)}}{\partial{x_i}}(x_i, t)\right)
  + \pi h_i^2(x_i) k^\prime\Big(\omega_0^{(i)}(x_i, t)\Big)
    \omega_1^{(i)}(x_i, t)
 \\[5mm]
  + \ 2\pi \, \delta_{\alpha_i, 1} \, h_i(x_i)
    \partial_s \kappa_i\Big(\omega_0^{(i)}(x_i, t), x_i, t \Big) \omega_1^{(i)}(x_i, t)
    \hspace{0.26cm}
  = \hspace{0.26cm} \widehat{F}_1^{(i)}(x_i, t),
    \hspace{0.31cm} (x_i, t)\in I_i \times (0, T), \hspace{0.30cm} i=1,2,3,
 \\[4mm]
 \begin{array}{rclrl}
    \omega_1^{(i)}(\ell_i, t) & = &
    0, & t\in(0, T), & i=1,2,3,
 \\[3mm]
    \hspace{1.92cm}
    \omega_1^{(1)} (0, t) \hspace{0.26cm} = \hspace{0.26cm}
    \omega_1^{(2)} (0, t) - \boldsymbol{\delta_1^{(2)}}(t) & = &
    \omega_1^{(3)} (0, t) - \boldsymbol{\delta_1^{(3)}}(t), \hspace{0.17cm} & t\in(0, T), &
 \\[3mm]
    \sum\limits_{i=1}^3 \pi h_i^2 (0) \dfrac{\partial\omega_{1}^{(i)}}{\partial{x_i}} (0, t) & = &
    \boldsymbol{d_1^*}(t), & t\in(0, T), &
 \\[3mm]
    \omega_{1}^{(i)} (x_i, 0)
    & = &
    0, & x_i\in I_i, & i=1,2,3,
 \end{array}
 \end{array}\right.
\end{equation}
where
\begin{multline}\label{omegaF1}
\widehat{F}_1^{(i)}(x_i, t)
 =  \int\limits_{\Upsilon_i(x_i)} f_1^{(i)} (\overline{\xi}_i, x_i, t) \, d{\overline{\xi}_i}
  - 2\pi \, \delta_{\alpha_i, 2} \,  h_i(x_i) \kappa_i\Big(\omega_0^{(i)}(x_i, t), x_i, t \Big)
\\
  + \delta_{\beta_i, 2} \int\limits_{\partial\Upsilon_i(x_i)}
    \varphi^{(i)}(\overline{\xi}_i, x_i, t) \, dl_{\overline{\xi}_i},
\quad (x_i, t)\in I_i \times (0, T), \quad i=1,2,3.
\end{multline}
The values $\boldsymbol{\delta_1^{(2)}}$ and $\boldsymbol{\delta_1^{(3)}}$
are uniquely determined (see Remark~\ref{remark_constant}) by formula
\begin{multline}\label{delta_1}
\boldsymbol{\delta_1^{(i)}}(t)
 =  \int_{\Xi} \mathfrak{N}_i (\xi) \,
    \sum\limits_{j=1}^3
    \bigg(
    \xi_j\dfrac{\partial\omega_0^{(j)}}{\partial{x_j}}(0, t) \chi_j^{\prime\prime}(\xi_j)
 + 2\dfrac{\partial\omega_0^{(j)}}{\partial{x_j}}(0, t) \chi_j^{\prime}(\xi_j)
    \bigg) \, d\xi,
\\
 +  \int_{\Gamma_0} \mathfrak{N}_i(\xi)
    \Big(
 -   \delta_{\alpha_0, 0} \, \kappa_0\big(\omega_0^{(1)}(0, t)\big)
 +   \delta_{\beta_0, 0} \, \varphi^{(0)} (\xi, t)
    \Big) \, d\sigma_\xi,
\quad i=2,3.
\end{multline}

With the help of the substitutions
$
\phi_1^{(1)}(x_1, t) = \omega_1^{(1)}(x_1, t),
$
$
\phi_1^{(2)}(x_2, t) = \omega_1^{(2)}(x_2, t) - \boldsymbol{\delta_1^{(2)}} (t) (\ell_2 - x_2),
$
$
\phi_1^{(3)}(x_3, t) = \omega_1^{(3)}(x_3, t) - \boldsymbol{\delta_1^{(3)}} (t) (\ell_3 - x_3),
$
we reduce the problem (\ref{omega_probl*}) to the respective linear parabolic problem in the space
$L^2\big(0, T; \mathcal{H}_0\big).$ Thus the existence and uniqueness of the solution to the problem (\ref{omega_probl*}) follow from
the classical theory of linear parabolic problems.


\section{Justification}\label{justification}

With the help of  $\widetilde{\omega}_0, \ \widetilde{\omega}_1, \ N_1$
and smooth cut-off functions defined by formulas
\begin{equation}\label{cut-off-functions}
\chi_{\ell_0}^{(i)} (x_i) =
\left\{\begin{array}{ll}
1, & \text{if} \ \ x_i \ge 3 \, \ell_0,
\\
0, & \text{if} \ \ x_i \le 2 \, \ell_0,
\end{array}\right.
\quad i=1, 2, 3,
\end{equation}
we construct the following asymptotic approximation:
\begin{multline}\label{asymp_expansion}
{U}_\varepsilon^{(1)} (x, t)
 = \sum\limits_{i=1}^3 \chi_{\ell_0}^{(i)} \left(\frac{x_i}{\varepsilon^\mathfrak{a}}\right)
    \left( \omega_0^{(i)} (x_i, t) + \varepsilon \, \omega_1^{(i)} (x_i, t) \right)
\\
 +  \left(1 - \sum\limits_{i=1}^3 \chi_{\ell_0}^{(i)}
     \left(\frac{x_i}{\varepsilon^\mathfrak{a}}\right)
    \right)
    \bigg(
    \omega_0^{(1)}(0, t)
 +  \varepsilon N_1 \left( \frac{x}{\varepsilon}, t \right)
    \bigg),
\quad x\in\Omega_\varepsilon\times(0, T),
\end{multline}
where $\mathfrak{a}$ is a fixed number from the interval $\big(\frac23, 1 \big).$

\begin{theorem}\label{mainTheorem}
Let assumptions made in the statement of the problem $(\ref{probl})$ are satisfied.
Then the sum $(\ref{asymp_expansion})$ is the asymptotic approximation for the solution
$u_\varepsilon$ to the boundary-value problem~$(\ref{probl}),$ i.e., \quad
$
      \exists \, {C}_0 >0 \ \ \exists \, \varepsilon_0>0 \ \
      \forall\, \varepsilon\in(0, \varepsilon_0):
$
\begin{equation}\label{t0}
      \max\limits_{t\in[0,T]}
      \left\|
       {U}_\varepsilon^{(1)} (\cdot, t) - u_\varepsilon (\cdot, t)
      \right\|_{L^2(\Omega_\varepsilon)}
  +   \left\|
       {U}_\varepsilon^{(1)} - u_\varepsilon
      \right\|_{L^2(0, T; {H}^1(\Omega_\varepsilon))}
 \leq {C}_0 \, \mu(\varepsilon),
\end{equation}
where $\mu(\varepsilon) = o(\varepsilon)$ as $\varepsilon \to 0$ and
\begin{equation}\label{mu}
\mu(\varepsilon)
 = \bigg(
    \varepsilon^{1+\frac{\mathfrak{a}}2}
 +  \sum_{i=1}^3
     \Big(
      (1-\delta_{\alpha_{i}, 1}) \varepsilon^{\alpha_i}
 +    (1-\delta_{\beta_{i}, 1}) \varepsilon^{\beta_i}
     \Big)
 +   (1-\delta_{\alpha_{0}, 0}) \varepsilon^{\alpha_0+1}
 +   (1-\delta_{\beta_{0}, 0})\varepsilon^{\beta_0+1}
  \bigg).
\end{equation}
\end{theorem}

\begin{proof}
Substituting ${U}_\varepsilon^{(0)}$
in the equations and the boundary conditions of problem~(\ref{probl}), we find
\begin{equation}\label{nevyazka}
\left\{\begin{array}{rclll}
   \partial_t {U}_\varepsilon^{(1)}
 - \Delta_x {U}_\varepsilon^{(1)}
 + k\Big({U}_\varepsilon^{(1)}\Big)
 - f
 & = & \widehat{R}_\varepsilon               &
   \mbox{in} \ \Omega_\varepsilon\times(0, T), &
\\[2mm]
   \partial_\nu {U}_\varepsilon^{(1)}
 + \varepsilon^{\alpha_0} \kappa_0\Big( {U}_\varepsilon^{(1)}\Big)
 - \varepsilon^{\beta_0} \varphi_\varepsilon^{(0)}
 & = & \breve{R}_{\varepsilon}^{(0)} &
   \mbox{on} \ \Gamma_\varepsilon^{(0)}\times(0, T), &
\\[2mm]
   \partial_\nu {U}_\varepsilon^{(1)}
 + \varepsilon^{\alpha_i} \kappa_i\Big( {U}_\varepsilon^{(1)}, x_i, t \Big)
 - \varepsilon^{\beta_i} \varphi_\varepsilon^{(i)}
 & = & \breve{R}_{\varepsilon}^{(i)} &
   \mbox{on} \ \Gamma_\varepsilon^{(i)}\times(0, T), &                      i=1,2,3,
\\[2mm]
   {U}_\varepsilon^{(1)}               & = & 0 &
   \mbox{on} \ \Upsilon_\varepsilon^{(i)}(\ell_i)\times(0, T), &                 i=1,2,3,
\end{array}\right.
\end{equation}
where
\begin{multline*}
\widehat{R}_{\varepsilon}(x, t)
 =- \sum\limits_{i=1}^3
    \Bigg(
     2 \varepsilon^{-\mathfrak{a}}
     \frac{d \chi_{\ell_0}^{(i)}}{d\zeta_i} (\zeta_i)
     \bigg|_{\zeta_i=\frac{x_i}{\varepsilon^\mathfrak{a}}}
     \bigg(
      \dfrac{\partial\omega_{0}^{(i)}}{\partial x_i} (x_i, t)
 -    \dfrac{\partial\omega_{0}^{(i)}}{\partial x_i} (0, t)
 +    \varepsilon \dfrac{\partial\omega_{1}^{(i)}}{\partial x_i} (x_i, t)
\\
 -    \Big(
       \frac{\partial{N}_{1}}{\partial{\xi_i}}(\xi, t)
 -     \frac{\partial{G}_{1}}{\partial{\xi_i}}(\xi, t)
      \Big)\Big|_{\xi=\frac{x}{\varepsilon}}
     \bigg)
\\
 +   \varepsilon^{-2\mathfrak{a}}
     \frac{d^2 \chi_{\ell_0}^{(i)}}{d\zeta_i^2} (\zeta_i)
     \bigg|_{\zeta_i=\frac{x_i}{\varepsilon^\mathfrak{a}}}
     \bigg(
      \omega_{0}^{(i)}(x_i, t)
 -    \omega_{0}^{(i)}(0, t)
 -    x_i \frac{\partial \omega_{0}^{(i)}}{\partial x_i}(0, t)
 +    \varepsilon \omega_{1}^{(i)}(x_i, t)
 -    \varepsilon \omega_{1}^{(i)}(0, t)
\\
 -    \varepsilon {N}_{1}\Big(\frac{x}{\varepsilon}, t\Big)
 +    \varepsilon G_{1}\Big(\frac{x}{\varepsilon}, t\Big)
     \bigg)
\\
 +   \chi_\ell^{(i)} \left(\frac{x_i}{\varepsilon^\mathfrak{a}}\right)
     \bigg(
 -    \dfrac{\partial\omega_{0}^{(i)}}{\partial t} (x_i, t)
 +    \dfrac{\partial^2\omega_{0}^{(i)}}{\partial x_i^2} (x_i, t)
 -    \varepsilon \dfrac{\partial\omega_{1}^{(i)}}{\partial t} (x_i, t)
 +    \varepsilon \dfrac{\partial^2\omega_{1}^{(i)}}{\partial x_i^2} (x_i, t)
     \bigg)
    \Bigg)
\\
 +  \left(1 - \sum\limits_{i=1}^3 \chi_{\ell_0}^{(i)}
     \left(\frac{x_i}{\varepsilon^\mathfrak{a}}\right)
    \right)
    \left(
    \dfrac{\partial \omega_{0}^{(i)}}{\partial t} (0, t)
 +  \varepsilon \dfrac{\partial N_1}{\partial t} \left( \frac{x}{\varepsilon}, t \right)
    \right)
 +  k\Big({U}_\varepsilon^{(1)}(x, t)\Big)
 -  f(x, t),
\end{multline*}
and
\begin{multline*}
\breve{R}_{\varepsilon}^{(0)}(x, t)
 =  \varepsilon^{\alpha_0} \kappa_0\Big({U}_\varepsilon^{(1)}(x, t)\Big)
 -  \delta_{\alpha_0, 0} \, \kappa_0\big(\omega_0^{(1)}(0, t)\big)
 -  \varepsilon^{\beta_0} \varphi_\varepsilon^{(0)}(x, t)
 +  \delta_{\beta_0, 0} \, \varphi_\varepsilon^{(0)} (x, t),
\\
\breve{R}_{\varepsilon}^{(i)}(x, t)
 =  - \, \frac{\varepsilon h_i^\prime (x_i)}{\sqrt{ 1 + \varepsilon^2 |h_i^\prime(x_i)|^2 \, }}
    \chi_{\ell_0}^{(i)} \left(\frac{x_i}{\varepsilon^\mathfrak{a}}\right)
    \left(
     \frac{\partial\omega_0^{(i)}}{\partial x_i} (x_i, t)
 +   \varepsilon \frac{\partial\omega_1^{(i)}}{\partial x_i} (x_i, t)
    \right)
\\
 +  \varepsilon^{\alpha_i} \kappa_i\Big({U}_\varepsilon^{(1)}(x, t), x_i, t \Big)
 -  \varepsilon^{\beta_i} \varphi_\varepsilon^{(i)}(x, t), \qquad i=1,2,3.
\end{multline*}

Since
$$
\omega_0^{(i)}(x_i, 0) = \omega_1^{(i)}(x_i, 0) = 0, \qquad
x_i\in I_i, \quad i=1,2,3,
$$
it follows from (\ref{junc_probl_n}) at $n=1$ that $N_1\big|_{t=0} = 0.$
As result, asymptotic approximation (\ref{asymp_expansion})
leaves no residuals in the initial condition, i.e.,
\begin{equation*}\label{initial}
   {U}_\varepsilon^{(1)}\big|_{t=0} \ = \ 0 \qquad
   \mbox{in} \ \Omega_\varepsilon.
\end{equation*}

From (\ref{nevyazka}) we derive the following integral relation:
\begin{multline}\label{int-identity_U}
     \int_{\Omega_\varepsilon}
     \partial_t {U}_\varepsilon^{(1)} v \, dx
 +   \int_{\Omega_\varepsilon}
     \nabla {U}_\varepsilon^{(1)} \cdot \nabla v \, dx
 +   \int_{\Omega_\varepsilon} k({U}_\varepsilon^{(1)}) v \, dx
 +   \varepsilon^{\alpha_0}
     \int_{\Gamma_\varepsilon^{(0)}} \kappa_0({U}_\varepsilon^{(1)}) v \, d\sigma_x
\\
 +   \sum_{i=1}^3 \varepsilon^{\alpha_i}
     \int_{\Gamma_\varepsilon^{(i)}} \kappa_i({U}_\varepsilon^{(1)}, x_i, t) v \, d\sigma_x
 -   \int_{\Omega_\varepsilon} f \, v \, dx
 -   \sum_{i=0}^3 \varepsilon^{\beta_i}
     \int_{\Gamma_\varepsilon^{(i)}} \varphi_\varepsilon^{(i)} \, v \, d\sigma_x
 =   R_\varepsilon(v),
\end{multline}
for all $v\in L^2(0, T; \mathcal{H}_\varepsilon)$ and a.e. $t\in(0, T).$
Here
$$
R_\varepsilon(v)
 =  \int_{\Omega_\varepsilon} \widehat{R}_\varepsilon \, v \,dx
 +  \sum \limits_{i=0}^3 \
    \int_{\Gamma_\varepsilon^{(i)}}
    \breve{R}_{\varepsilon}^{(i)} \, v \, d\sigma_x.
$$

From (\ref{regul_probl_2}) and (\ref{regul_probl_3}) we deduce that integral identities
\begin{multline}\label{int-identity_u_2}
     \int_{\Upsilon_i(x_i)}
     \bigg(
 -    \dfrac{\partial\omega_{0}^{(i)}}{\partial t}
 +    \frac{\partial^2\omega_0^{(i)}}{\partial x_i^2}
     \bigg) \eta \ d\overline{\xi}_i
 =   \int_{\Upsilon_i(x_i)}
     \nabla_{\overline{\xi}_i} u_2^{(i)} \cdot \nabla_{\overline{\xi}_i} \eta \ d\overline{\xi}_i
 -   \int_{\partial\Upsilon_i(x_i)}
     h_i^\prime \, \frac{\partial \omega_0^{(i)}}{\partial x_i} \eta \ dl_{\overline{\xi}_i}
\\
 +   \int_{\Upsilon_i(x_i)} k\big(\omega_0^{(i)}\big) \eta \ d\overline{\xi}_i
 +   \delta_{\alpha_i, 1} \int_{\partial\Upsilon_i(x_i)}
     \kappa_i\big(\omega_0^{(i)}, x_i, t \big) \eta \,dl_{\overline{\xi}_i}
 -   \int_{\Upsilon_i(x_i)} f_0^{(i)} \eta \ d\overline{\xi}_i
 -   \delta_{\beta_i, 1} \int_{\partial\Upsilon_i(x_i)} \varphi^{(i)} \eta \ dl_{\overline{\xi}_i}
\end{multline}
and
\begin{multline}\label{int-identity_u_3}
     \int_{\Upsilon_i(x_i)}
     \bigg(
 -    \dfrac{\partial\omega_{1}^{(i)}}{\partial t}
 +    \frac{\partial^2\omega_1^{(i)}}{\partial x_i^2}
     \bigg) \eta \ d\overline{\xi}_i
 =   \int_{\Upsilon_i(x_i)}
     \nabla_{\overline{\xi}_i} u_3^{(i)} \cdot \nabla_{\overline{\xi}_i} \eta \ d\overline{\xi}_i
 -   \int_{\partial\Upsilon_i(x_i)}
     h_i^\prime \, \frac{\partial \omega_1^{(i)}}{\partial x_i} \eta \ dl_{\overline{\xi}_i}
\\
 +   \int_{\Upsilon_i(x_i)}
     k^\prime \big(\omega_0^{(i)}\big) \omega_1^{(i)} \eta \ d\overline{\xi}_i
 +   \delta_{\alpha_i, 1} \int_{\partial\Upsilon_i(x_i)}
     \partial_s \kappa_i \big(\omega_0^{(i)}, x_i, t \big)
     \omega_1^{(i)} \eta \,dl_{\overline{\xi}_i}
 -   \int_{\Upsilon_i(x_i)} f_1^{(i)} \eta \ d\overline{\xi}_i
\\
 +   \delta_{\alpha_i, 2} \int_{\partial\Upsilon_i(x_i)}
     \kappa_i\big(\omega_0^{(i)}, x_i, t \big) \eta \,dl_{\overline{\xi}_i}
 -   \delta_{\beta_i, 2} \int_{\partial\Upsilon_i(x_i)} \varphi^{(i)} \eta \ dl_{\overline{\xi}_i}
\end{multline}
hold for all
$\eta\in H^1(\Upsilon_i(x_i))$ and for all
$(x_i, t) \in I_\varepsilon^{(i)} \times (0, T), \quad i=1,2,3.$

Using (\ref{int-identity_u_2}) and (\ref{int-identity_u_3}),  we rewrite $R_\varepsilon$ in the form
$$
R_\varepsilon(v)
 =  \sum_{j=1}^{12} R_{\varepsilon, j}(v),
$$
where
\begin{equation*}
R_{\varepsilon, 1} (v)
 =   \int_{\Omega_\varepsilon}
     \bigg(
      k \big( U_\varepsilon^{(1)}(x, t)\big)
 -    \sum_{i=1}^3 \chi_{\ell_0}^{(i)} \left(\frac{x_i}{\varepsilon^\mathfrak{a}}\right)
      \left(
       k \big( \omega_0^{(i)}(x_i, t)\big)
 +     \varepsilon \, k^\prime \big( \omega_0^{(i)}(x_i, t)\big) \omega_1^{(i)}(x_i, t)
      \right)
     \bigg)
     v (x) \,dx,
\end{equation*}
\begin{equation*}
R_{\varepsilon, 2} (v)
 =-  \int_{\Omega_\varepsilon}
     \left(
      f(x, t)
 -    \sum_{i=1}^3 \chi_{\ell_0}^{(i)} \left(\frac{x_i}{\varepsilon^\mathfrak{a}}\right)
      \left(
       f_0^{(i)}(x_i, t)
 +     \varepsilon f_1^{(i)}\Big(\frac{\overline{x}_i}{\varepsilon}, x_i, t \Big)
      \right)
     \right) v (x) \,dx,
\end{equation*}
\begin{equation*}
R_{\varepsilon, 3} (v)
 =   \varepsilon^{\alpha_0} \int_{\Gamma_\varepsilon^{(0)}}
     \Big(
      \kappa_0 \big(U_\varepsilon^{(1)}(x, t)\big)
 -    \delta_{\alpha_0, 0} \, \kappa_0\big(\omega_0^{(1)}(0, t)\big)
     \Big) v (x) \, d\sigma_x
 -   \varepsilon^{\beta_0} \int_{\Gamma_\varepsilon^{(0)}}
     ( 1 - \delta_{\beta_0, 0}) \, \varphi_\varepsilon^{(0)} (x, t) \, v (x) \,d\sigma_x,
\end{equation*}
\begin{multline*}
R_{\varepsilon, 4} (v)
 =   \sum_{i=1}^3 \varepsilon^{\alpha_i}
     \int_{\Gamma_\varepsilon^{(i)}}
     \bigg(
      \kappa_i \big( U_\varepsilon^{(1)}(x, t), x_i, t \big)
 -    \chi_{\ell_0}^{(i)}\left(\frac{x_i}{\varepsilon^\mathfrak{a}}\right)
      \Big(
       \delta_{\alpha_i, 1} \kappa_i \big( \omega_0^{(i)}(x_i, t), x_i, t \big)
\\
 +     \varepsilon \ \delta_{\alpha_i, 1}
       \partial_s \kappa_i \big( \omega_0^{(i)}(x_i, t), x_i, t \big) \omega_1^{(i)}(x_i, t)
 +     \delta_{\alpha_i, 2}
       \kappa_i \big( \omega_0^{(i)}(x_i, t), x_i, t \big)
      \Big)
     \bigg) v (x) \,d\sigma_x,
\end{multline*}
\begin{equation*}
R_{\varepsilon, 5} (v)
 = - \sum_{i=1}^3 \varepsilon^{\beta_i}
     \int_{\Gamma_\varepsilon^{(i)}}
     \bigg(
      1 - \chi_{\ell_0}^{(i)}
      \left(\frac{x_i}{\varepsilon^\mathfrak{a}}\right)
      (\delta_{\beta_i, 1} + \delta_{\beta_i, 2})
     \bigg)
     \varphi_\varepsilon^{(i)} (x, t) \, v (x) \,d\sigma_x,
\end{equation*}
\begin{equation*}
R_{\varepsilon, 6} (v)
 =  \int_{\Omega_\varepsilon}
    \bigg(1 - \sum\limits_{i=1}^3 \chi_{\ell_0}^{(i)}
     \left(\frac{x_i}{\varepsilon^\mathfrak{a}}\right)
    \bigg)
    \left(
    \dfrac{\partial \omega_{0}^{(i)}}{\partial t} (0, t)
 +  \varepsilon \dfrac{\partial N_1}{\partial t} \left( \frac{x}{\varepsilon}, t \right)
    \right)
    v(x) \ dx,
\end{equation*}
\begin{equation*}
R_{\varepsilon, 7} (v)
 =     \varepsilon \sum_{i=1}^3
       \int_{\Gamma_\varepsilon^{(i)}}
       h_i^\prime (x_i)
       \left(
        \frac{\partial\omega_0^{(i)}}{\partial x_i} (x_i, t)
 +      \varepsilon \frac{\partial\omega_1^{(i)}}{\partial x_i} (x_i, t)
       \right)
       \left(
        1 - \frac{1}{\sqrt{ 1 + \varepsilon^2 |h_i^\prime(x_i)|^2 \, }}
       \right)
       \chi_{\ell_0}^{(i)} \left(\frac{x_i}{\varepsilon^\mathfrak{a}}\right)
       v (x) \,d\sigma_x,
\end{equation*}
\begin{equation*}
R_{\varepsilon, 8} (v)
 =-  2 \varepsilon^{-\mathfrak{a}} \sum_{i=1}^3
     \int_{\Omega_\varepsilon}
     \frac{d \chi_{\ell_0}^{(i)}}{d\zeta_i} (\zeta_i)
     \bigg|_{\zeta_i=\frac{x_i}{\varepsilon^\mathfrak{a}}}
     \bigg(
      \dfrac{\partial\omega_{0}^{(i)}}{\partial x_i} (x_i, t)
 -    \dfrac{\partial\omega_{0}^{(i)}}{\partial x_i} (0, t)
 +    \varepsilon \dfrac{\partial\omega_{1}^{(i)}}{\partial x_i} (x_i, t)
     \bigg)
     v (x) \,dx,
\end{equation*}
\begin{multline*}
R_{\varepsilon, 9} (v)
 =-    \varepsilon^{-2\mathfrak{a}} \sum_{i=1}^3
       \int_{\Omega_\varepsilon}
       \frac{d^2 \chi_{\ell_0}^{(i)}}{d\zeta_i^2} (\zeta_i)
       \bigg|_{\zeta_i=\frac{x_i}{\varepsilon^\mathfrak{a}}}
\\
 \cdot \bigg(
        \omega_{0}^{(i)}(x_i, t)
 -      \omega_{0}^{(i)}(0, t)
 -      x_i \frac{\partial \omega_{0}^{(i)}}{\partial x_i}(0, t)
 +      \varepsilon \omega_{1}^{(i)}(x_i, t)
 -      \varepsilon \omega_{1}^{(i)}(0, t)
       \bigg)
       v (x) \,dx,
\end{multline*}
\begin{equation*}
R_{\varepsilon, 10} (v)
 =-  \varepsilon^{2} \sum_{i=1}^3
     \int_{I_\varepsilon^{(i)}}
     \int_{\Upsilon_i(x_i)}
     \chi_{\ell_0}^{(i)} \left(\frac{x_i}{\varepsilon^\mathfrak{a}}\right)
     \nabla_{\overline{\xi}_i} u_2^{(i)} (\overline{\xi}_i, x_i, t)
     \cdot \nabla_{\overline{\xi}_i} v (x, t) \ d\overline{\xi}_i
     \,dx_i,
\end{equation*}
\begin{equation*}
R_{\varepsilon, 11} (v)
 =-  \varepsilon^{3} \sum_{i=1}^3
     \int_{I_\varepsilon^{(i)}}
     \int_{\Upsilon_i(x_i)}
     \chi_{\ell_0}^{(i)} \left(\frac{x_i}{\varepsilon^\mathfrak{a}}\right)
     \nabla_{\overline{\xi}_i} u_3^{(i)} ( \overline{\xi}_i, x_i, t)
     \cdot \nabla_{\overline{\xi}_i} v (x, t) \ d\overline{\xi}_i
     \,dx_i,
\end{equation*}
\begin{multline*}
R_{\varepsilon, 12} (v)
 =- \sum\limits_{i=1}^3
    \int_{\Omega_\varepsilon}
    \Bigg(
   2\varepsilon^{-\mathfrak{a}}
    \frac{d\chi_{\ell_0}^{(i)}}{d\zeta_i}(\zeta_i)
    \bigg(
    \frac{\partial{N}_{1}}{\partial{\xi_i}}(\xi, t)
 -  \frac{\partial{G}_{1}}{\partial{\xi_i}}(\xi, t)
    \bigg)
\\
 +  \varepsilon^{1-2\mathfrak{a}}
    \frac{d^2\chi_{\ell_0}^{(i)}}{d\zeta_i^2}(\zeta_i)
    \Big( {N}_{1}(\xi, t) - G_{1}(\xi, t) \Big)
    \Bigg)
    \Bigg|_{\zeta_i=\frac{x_i}{\varepsilon^{\mathfrak{a}}},\, \xi=\frac{x}{\varepsilon}}
    v(x) \ dx.
\end{multline*}

Let us estimate the value $R_\varepsilon.$
Using (\ref{ineq1}), (\ref{ineq-3}) and (\ref{kappa_ineq+}), we deduce the following estimates:
\begin{equation}\label{R1R2+}
     | R_{\varepsilon, j}(v) |
 \leq \check{C}
      \left(
       \sum_{i=1}^3 \sqrt{ \pi \ell_i \max\limits_{x_i\in I_i} h_i^2(x_i) \ } \varepsilon^2
 +     \sqrt{|\Xi^{(0)}|_3
 +     3 \pi {\ell_0} \sum_{i=1}^3 h_i^2(0)\ } \varepsilon^{1+\frac{\mathfrak{a}}2}
      \right)   \| v \|_{L^2(\Omega_\varepsilon)},
\quad j=1,2,
\end{equation}
\begin{equation}\label{R3+}
     | R_{\varepsilon, 3}(v) |
 \leq \check{C} \sqrt{| \Gamma_0 \, |_2} \,
      \big(
       \varepsilon^{\alpha_0 + \delta_{\alpha_0, 0}}
 +     \varepsilon^{\beta_0} ( 1 - \delta_{\beta_0, 0})
      \big) \, \varepsilon \, \| v \|_{H^1(\Omega_\varepsilon)},
\end{equation}
\begin{equation}\label{R4+}
     | R_{\varepsilon, 4}(v) |
 \leq \check{C} \sum_{i=1}^3
      \left(
       \sqrt{2 \pi \ell_i \max\limits_{x_i\in I_i} h_i(x_i) \ }
       \varepsilon^{\alpha_i + (\delta_{\alpha_i, 1} + \delta_{\alpha_i, 2})}
 +     \sqrt{6 \pi {\ell_0} h_i(0) \ }
       \varepsilon^{\alpha_i + \frac{\mathfrak{a}}2}
      \right)   \| v \|_{H^1(\Omega_\varepsilon)},
\end{equation}
\begin{multline}\label{R5+}
     | R_{\varepsilon, 5}(v) |
 \leq \check{C} \sum_{i=1}^3
      \bigg(
       (1 - \delta_{\beta_i, 1} - \delta_{\beta_i, 2})
       \sqrt{2 \pi \ell_i \max\limits_{x_i\in I_i} h_i(x_i) \ }
       \varepsilon^{\beta_i}
\\
 +     (\delta_{\beta_i, 1} + \delta_{\beta_i, 2})
       \sqrt{6 \pi {\ell_0} h_i(0) \ }
       \varepsilon^{\beta_i + \frac{\mathfrak{a}}2}
      \bigg)   \| v \|_{H^1(\Omega_\varepsilon)},
\end{multline}
\begin{equation}\label{R6+}
       | R_{\varepsilon, 6}(v) |
 \leq \check{C}
      \sqrt{|\Xi^{(0)}|_3 + 3 \pi {\ell_0} \sum_{i=1}^3 h_i^2(0)\ } \cdot
      \varepsilon^{1+\frac{\mathfrak{a}}2} \, \| v \|_{L^2(\Omega_\varepsilon)},
\end{equation}
\begin{equation}\label{R7+}
    |  R_{\varepsilon, 7}(v) |
 \leq \check{C}
      \sum_{i=1}^3 \sqrt{2 \pi \ell_i \max\limits_{x_i\in I_i} h_i(x_i)\ } \varepsilon^3 \,
      \| v \|_{H^1(\Omega_\varepsilon)},
\end{equation}
\begin{equation}\label{R8R9+}
     | R_{\varepsilon, j}(v) |
 \leq \check{C} \sum_{i=1}^3 \sqrt{ \pi {\ell_0} h_i^2(0) \ }
      \ \varepsilon^{1 + \frac{\mathfrak{a}}2} \,
      \| v \|_{L^2(\Omega_\varepsilon)},
\quad j=8,9,
\end{equation}
\begin{equation}\label{R10R11+}
    |  R_{\varepsilon, 10}(v) |
 \leq \check{C} \varepsilon^2 \,
      \| \nabla_x v \|_{L^2(\Omega_\varepsilon)},
\qquad
     | R_{\varepsilon, 11}(v) |
 \leq \check{C} \varepsilon^3 \,
      \| \nabla_x v \|_{L^2(\Omega_\varepsilon)}.
\end{equation}
Due to the exponential decreasing of functions ${N}_{1} - {G}_{1}$
(see Remark~\ref{rem_exp-decrease}) and the fact that the support
of the derivative of $\chi_{\ell_0}^{(i)}$ belongs to the set
$\{x_i: 2 {\ell_0} \varepsilon^\mathfrak{a} \le x_i \le 3 {\ell_0} \varepsilon^\mathfrak{a}\},$
we arrive that
\begin{equation}\label{R12+}
     | R_{\varepsilon, 12}(v) |
 \leq \check{C} \varepsilon^{-1-\mathfrak{a}}
      \exp{
       \left(
       -\frac{2 {\ell_0}}{\varepsilon^{1-\mathfrak{a}}} \
        \min\limits_{i=1,2,3}{\gamma_i}
       \right)
          } \,
      \| v \|_{L^2(\Omega_\varepsilon)}.
\end{equation}

Subtracting the integral identity $(\ref{int-identity})$ from $(\ref{int-identity_U})$
and integrating over $t\in(0, \tau),$ where $\tau\in(0, T],$ we obtain
\begin{equation}\label{int-identity_U-u}
     \int_{0}^{\tau}
     \Big(
      \big\langle
       \partial_t U_\varepsilon^{(1)} - \partial_t u_\varepsilon, \, v
      \big\rangle_\varepsilon
 +    \big\langle
       \mathcal{A}_\varepsilon (t) U_\varepsilon^{(1)}
 -     \mathcal{A}_\varepsilon (t) u_\varepsilon, \, v
      \big\rangle_\varepsilon
     \Big) dt
 =   \int_{0}^{\tau} R_\varepsilon(v) \, dt
\qquad \forall v\in L^2(0, T; \mathcal{H}_\varepsilon).
\end{equation}
Now set $v={U}_\varepsilon^{(1)} - u_\varepsilon$ in (\ref{int-identity_U-u}).
Then, taking into account that $\mathcal{A_\varepsilon}$ is strongly monotone
and (\ref{R1R2+})--(\ref{R12+}),   we arrive to the inequality
\begin{equation*}
       \left\|
        {U}_\varepsilon^{(1)} (\cdot, \tau) - u_\varepsilon (\cdot, \tau)
       \right\|_{L^2(\Omega_\varepsilon)}^2
  +    \left\| {U}_\varepsilon^{(1)} - u_\varepsilon \right\|_{L^2(0, \tau; \mathcal{H}_\varepsilon)}^2
  \leq C \mu(\varepsilon)
       \left\| {U}_\varepsilon^{(1)} - u_\varepsilon \right\|_{L^2(0, \tau; H^1(\Omega_\varepsilon))},
\end{equation*}
whence thanks to  (\ref{equivalent_norm}) it follows (\ref{t0}).
\end{proof}

\begin{corollary}\label{corollary1}
The differences between the solution $u_\varepsilon$ of problem $(\ref{probl})$ and the sum
$$
{U}_\varepsilon^{(0)} (x, t)
 =  \sum\limits_{i=1}^3 \chi_{\ell_0}^{(i)} \left(\frac{x_i}{\varepsilon^\mathfrak{a}}\right)
    \omega_0^{(i)} (x_i, t)
 +  \left(1 - \sum\limits_{i=1}^3 \chi_{\ell_0}^{(i)}
     \left(\frac{x_i}{\varepsilon^\mathfrak{a}}\right)
    \right)
    \omega_0^{(1)}(0, t),
\quad x\in\Omega_\varepsilon\times(0, T)
$$
admit the following asymptotic estimate:
\begin{equation}\label{t5}
      \max\limits_{t\in[0, T]}
      \| \, u_\varepsilon( \cdot, t) - U_\varepsilon^{(0)}( \cdot, t) \|_{L^2(\Omega_\varepsilon)}
 +    \| \, u_\varepsilon - U_\varepsilon^{(0)} \|_{L^2(0, T; H^1(\Omega_\varepsilon))}
 \leq \widetilde{C}_0 \,
      \mu(\varepsilon),
\end{equation}
where $\mu(\varepsilon)$ is defined in (\ref{mu}),
and $\mathfrak{a}$ is a fixed number from the interval $\big(\frac23, 1 \big).$

In each thin cylinder
$\Omega_{\varepsilon,\mathfrak{a}}^{(i)} :=
 \Omega_\varepsilon^{(i)} \cap \big\{ x\in \Bbb{R}^3 : \
 x_i\in I_{\varepsilon, \mathfrak{a}}^{(i)}
 := (3{\ell_0}\varepsilon^\mathfrak{a}, \ell_i) \big\}, \ (i=1,2,3)$
the following estimate holds:
\begin{equation}\label{t7}
      \max\limits_{t\in[0, T]}
      \| \, u_\varepsilon( \cdot, t) - \omega_0^{(i)}( \cdot, t) \|_{L^2(\Omega_{\varepsilon,\mathfrak{a}}^{(i)})}
 +    \| \, u_\varepsilon - \omega_0^{(i)}
      \|_{L^2(0, T; H^1(\Omega_{\varepsilon,\mathfrak{a}}^{(i)}))}
 \leq \widetilde{C}_1 \,
 \mu(\varepsilon),
\end{equation}
where
$\{\omega_0^{(i)}\}_{i=1}^3$ is the solution of the limit problem~$(\ref{main}).$

In the  neighbourhood
$\Omega^{(0)}_{\varepsilon, {\ell_0}} :=
 \Omega_\varepsilon\cap \big\{ x : \ \ x_i<2{\ell_0}\varepsilon, \ i=1,2,3 \big\}$
of the node $\Omega^{(0)}_{\varepsilon},$ we get estimates
\begin{equation}\label{t-joint0}
  \| \,
       \nabla_{x}u_\varepsilon - \nabla_{\xi} \, N_1   \|_{L^2\big(\Omega^{(0)}_{\varepsilon, {\ell_0}} \times (0, T)\big)}
 \le  \widetilde{C}_4 \, \mu(\varepsilon).
\end{equation}
\end{corollary}

\begin{proof} Denote by
$\chi_{{\ell_0},\mathfrak{a},\varepsilon}^{(i)}(\cdot) := \chi_{\ell_0}^{(i)}(\tfrac{\cdot}{\varepsilon^\mathfrak{a}})$
(the function $\chi_{\ell_0}^{(i)}$ is determined in (\ref{cut-off-functions})) and
$$
     \|  v \|_{\Omega}^{*}
 :=  \max\limits_{t\in[0, T]}
     \|  v( \cdot, t) \|_{L^2(\Omega)}
 +   \|  v \|_{L^2(0, T; H^1(\Omega))}.
$$
Using the smoothness of the functions $\{\omega_1^{(i)}\}_{i=1}^3$ and the exponential decay of  the functions\\
$\{ N_1-G_1 \}, \ i=1,2,3,$ at infinity, we deduce the inequality
(\ref{t5}) from  estimate (\ref{t0}), namely
$$
     \left\| \, u_\varepsilon-U^{(0)}_{\varepsilon}\right\|_{\Omega_\varepsilon}^{*}
 \le \left\| \, u_\varepsilon-U^{(1)}_{\varepsilon}\right\|_{\Omega_\varepsilon}^{*}
   + \varepsilon \,
     \Bigg\| \,
      \sum_{i=1}^3
      \chi_{{\ell_0},\mathfrak{a},\varepsilon}^{(i)} \, \omega_1^{(i)}
   +  \bigg( 1 - \sum_{i=1}^3 \chi^{(i)}_{{\ell_0},\mathfrak{a},\varepsilon} \bigg) \, N_1
     \Bigg\|_{\Omega_\varepsilon}^{*}
$$
$$
 \le C_1 \, \mu(\varepsilon)
   + \varepsilon \sum\limits_{i=1}^3
     \bigg\| \, \chi_{{\ell_0},\mathfrak{a},\varepsilon}^{(i)} \omega_1
   +  \big(1-\chi_{{\ell_0},\mathfrak{a},\varepsilon}^{(i)}\big) N_1
     \bigg\|_{\Omega^{(i)}_\varepsilon}^{*}
   + \varepsilon   \left\| N_1 \right\|_{\Omega_\varepsilon^{(0)}}^{*}
$$
$$
 \le C_1 \, \mu(\varepsilon)
   + \sum_{i=1}^3
     \bigg\|
      \big(1-\chi_{{\ell_0},\mathfrak{a},\varepsilon}^{(i)}\big) \,
       x_i \frac{d\omega_0^{(i)}}{dx_i}(0, \cdot)
     \bigg\|_{\Omega_\varepsilon^{(i)}}^{*}
   + \varepsilon \sum_{i=1}^3
     \bigg\|
      \big(1-\chi_{{\ell_0},\mathfrak{a},\varepsilon}^{(i)}\big) \,
      \Big( \omega_1^{(i)}(0, \cdot) - \omega_1^{(i)} \Big)
     \bigg\|_{\Omega_\varepsilon^{(i)}}^{*}
$$
$$
   +  \, \varepsilon \sum_{i=1}^3
      \left\| \,
       \omega_1^{(i)}
      \right\|_{\Omega_\varepsilon^{(i)}}^{*}
   +  \varepsilon \sum_{i=1}^3
      \left\| \,
       \big(1-\chi_{{\ell_0},\mathfrak{a},\varepsilon}^{(i)}\big) \,  ( N_1 - G_1 )
      \right\|_{\Omega_\varepsilon^{(i)}}^{*}
$$$$
   + \, \varepsilon^2 \max\limits_{t\in[0, T]}
     \|  N_1( \cdot, t) \|_{L^2(\Xi^{(0)})}
   + \varepsilon^\frac32 \|  N_1 \|_{L^2(0, T; H^1(\Xi^{(0)}))}
 \le \widetilde{C}_0 \, \mu(\varepsilon).
$$

Also with the help of estimate (\ref{t0}), we derive
$$
      \left\| \,
       u_\varepsilon - \omega_0^{(i)}
      \right\|_{\Omega_{\varepsilon,\mathfrak{a}}^{(i)}}^{*}
 \leq \left\| \,
       u_\varepsilon-U^{(1)}_{\varepsilon}
      \right\|_{\Omega_\varepsilon}^{*}
    + \varepsilon
      \left\| \,
       \omega_{1}^{(i)}
      \right\|_{\Omega_{\varepsilon,\mathfrak{a}}^{(i)}}^{*}
 \leq \widetilde{C}_2 \, \mu(\varepsilon),
$$
whence we get (\ref{t7}).

The energetic estimate (\ref{t-joint0}) in a neighbourhood of the node $\Omega^{(0)}_\varepsilon$ follows directly from (\ref{t0}).
\end{proof}

Using the Cauchy-Buniakovskii-Schwarz inequality and  the continuously embedding of
the space $H^1(I_{\varepsilon, \mathfrak{a}}^{(i)})$  in $C\big([3{\ell_0}\varepsilon^\mathfrak{a}, \ell_i]\big),$
it follows from (\ref{t7}) the following corollary.

\begin{corollary}\label{corollary2}
If $h_i(x_i) \equiv h_i \equiv const, \, i=1,2,3,$ then
\begin{equation}\label{t9}
\max\limits_{t\in[0, T]}\| \, (E^{(i)}_\varepsilon u_\varepsilon)(\cdot,t) - \omega_0^{(i)}(\cdot,t) \|_{L^2(I_{\varepsilon, \mathfrak{a}}^{(i)})} +
 \| \, E^{(i)}_\varepsilon u_\varepsilon - \omega_0^{(i)}  \|_{L^2\big(0, T; C([3{\ell_0}\varepsilon^\mathfrak{a}, \ell_i])\big)}
 \leq C_5 \,  \frac{\mu(\varepsilon)}{\varepsilon},
\end{equation}
where $\mu(\varepsilon)$ is defined in (\ref{mu}) and
$$
\big(E^{(i)}_\varepsilon u_\varepsilon\big)(x_i,t)
 = \frac{1}{\pi \varepsilon^2\, h_i^2}
   \int_{\Upsilon^{(i)}_\varepsilon(0)}
   u_\varepsilon(x,t)\, d\overline{x}_i .
$$
\end{corollary}

\section{Asymptotic approximation in the case
$\alpha_0 < 0,$ $\alpha_i \ge 1,$ $i\in\{1, 2, 3\}$}\label{alpha_0<0}

Due to (\ref{est-alpha_0}) we conclude that $\omega_0^{(i)}(0,t)=0,$ $i\in\{1, 2, 3\},$
and consequently also $N_0\equiv 0.$ Thus the limit problem ~$(\ref{main})$ splits into the following three independent problems:
\begin{equation}\label{main_al0<0}
 \left\{\begin{array}{l}
    \pi h_i^2(x_i) \dfrac{\partial\omega_0^{(i)}}{\partial{t}} (x_i, t)
  - \pi \dfrac{\partial}{\partial{x_i}}
    \left(h_i^2(x_i)\dfrac{\partial\omega_0^{(i)}}{\partial{x_i}}(x_i, t)\right)
  + \pi h_i^2(x_i) \, k\Big(\omega_0^{(i)}(x_i, t)\Big)
 \\[5mm]
 \begin{array}{rclr}
  + \ 2\pi \delta_{\alpha_i, 1} h_i(x_i)
    \kappa_i\Big(\omega_0^{(i)}(x_i, t), x_i, t \Big)
    & = &
    \widehat{F}_0^{(i)}(x_i, t), \ & (x_i, t) \in I_i\times(0, T),
 \\[3mm]
    \omega_0^{(i)}(0, t) \hspace{0.26cm} = \hspace{0.26cm} \omega_0^{(i)} (\ell_i, t)
    & = &
    0, & t\in(0, T),
 \\[2mm]
    \omega_{0}^{(i)} (x_i, 0)
    & = &
    0, & x_i\in I_i,
 \end{array}
 \end{array}\right.
\end{equation}
where $\{\widehat{F}_0^{(i)}\}_{i=1}^3$ is defined in (\ref{right-hand-side}),  $i\in\{1, 2, 3\}.$

However, to construct an asymptotic approximation and to obtain asymptotic estimates in this case,
we need extra assumptions. Namely, if $\alpha_0\in[-q, -q+1), \ q\in\Bbb{N}$
we assume the following more stronger condition of zero-absorption:
\begin{equation}\label{kappa_0_conditions1}
\begin{array}{c}
\kappa_0 \in C^{q+1}(\Bbb{R}), \qquad
\dfrac{d^{q+1}\kappa_0}{ds^{q+1}} \in L^\infty(\Bbb R), \qquad
\kappa_0(0)=\dfrac{d \kappa_0}{ds}(0)=\ldots=
\dfrac{d^{q-1}\kappa_0}{ds^{q-1}}(0)=0,
\\[2mm]
\exists\, k_->0 \ \ \forall\, s\in\Bbb{R}:\qquad
\dfrac{d^{q}\kappa_0}{ds^{q}}(s)\ge k_- ;
\end{array}
\end{equation}
in addition, if $\alpha_0\neq -q$ then
\begin{multline}\label{kappa_0_conditions2}
f\in C_x^{q-1}\left(\overline{\Omega_{a_0}}\times [0,T]\right)
\,\cap\, C_{\overline{x}_i}^{q}\left(\overline{\Omega_{a_0}^{(i)}}\times [0,T]\right), \qquad
k\in C^{q+1}(\Bbb{R}), \qquad
\dfrac{d^{q_1}k}{ds^{q_1}} \in L^\infty(\Bbb R), \quad
q_1\in \{1,\ldots,q+1\},
\\
\kappa_i\in C^{q+1, q-1, 0}\big(\Bbb{R}\times [0, \ell_i] \times [0,T]\big), \qquad
\dfrac{d^{q_1}\kappa_i}{ds^{q_1}}(\cdot, x_i, t) \in L^\infty(\Bbb R), \quad
q_1\in \{1,\ldots,q+1\},
\end{multline}
uniformly with respect to $x_i\in[0, \ell_i]$ and $t\in[0, T] \ \big(i\in\{1, 2, 3\}\big).$

\begin{proposition}
Under conditions (\ref{kappa_0_conditions1})
\begin{equation}\label{est-alpha_q}
\frac{1}{\varepsilon^3} \int_{\Omega_\varepsilon^{(0)}\times(0,T)} u_\varepsilon^2\, dx \, dt
\le C_4\, \varepsilon^{\min\{1, -\frac{2\alpha_0}{q+1}\}} \ \ \longrightarrow \ 0
\quad \text{as} \ \ \varepsilon \to 0.
\end{equation}
\end{proposition}
\begin{proof}
With the help of Taylor's formula (with Lagrange form of the remainder) and
(\ref{kappa_0_conditions1}), we obtain
$$
|\kappa_0(s) \, s|\ge \frac{k_-}{q!}\,  |s|^{q+1}, \quad s\in\Bbb{R}.
$$
Knowing that $\kappa_0(s)s\ge0$ for all $s\in\Bbb{R}$
(see (\ref{kappa_ineq+})), we get
\begin{equation}\label{kappa_0_q!}
\kappa_0(s) \, s\ge \frac{k_-}{q!}\, |s|^{q+1}, \quad s\in\Bbb{R}.
\end{equation}

Similarly as in subsection \ref{a-priori_es},
from the integral identity (\ref{int-identity}) and inequalities (\ref{kappa_ineq+}),
(\ref{ineq1}), (\ref{equivalent_norm}), (\ref{ineq-3}) and (\ref{apr_ocin})
it follows
$$
\varepsilon^{\alpha_0} \int_{\Gamma_\varepsilon^{(0)}\times(0,T)}
\kappa_0(u_\varepsilon) \, u_\varepsilon \, d\sigma_x dt
\le C_1 \varepsilon^2.
$$
Thanks to (\ref{kappa_0_q!}) and H\"{o}lder's inequality, we get
$$
\int_{\Gamma_\varepsilon^{(0)}\times(0,T)} u_\varepsilon^2 \, d\sigma_x dt
 \le \Big(\varepsilon^2 |\Gamma_0|_2 \,T\Big)^{\frac{q-1}{q+1}}
     \bigg(
      \int_{\Gamma_\varepsilon^{(0)}\times(0, T)}
      |u_\varepsilon|^{q+1} \, d\sigma_x dt
     \bigg)^{\frac2{q+1}}
$$
$$
 \le \Big(\varepsilon^2 |\Gamma_0|_2 \,T\Big)^{\frac{q-1}{q+1}}
     \Big(\frac{q!}{k_-}\Big)^{\frac2{q+1}}
     \bigg(
      \int_{\Gamma_\varepsilon^{(0)}\times(0, T)}
      \kappa_0(u_\varepsilon) \, u_\varepsilon \, d\sigma_x dt
     \bigg)^{\frac{2}{q+1}}
 \le C_2 \varepsilon^{2 -\frac{2\alpha_0}{q+1}}.
$$
Now with the help of (\ref{traceineq0}) we get
\begin{equation*}
 \int_{\Omega_\varepsilon^{(0)}\times (0,T)} u_\varepsilon^2\, dx \,dt
 \le C_3
     \Bigg(
      \varepsilon^2 \int_{\Omega_\varepsilon^{(0)}\times(0,T)}
      |\nabla_{x}u_\varepsilon|^2 \, dx dt
 +    \varepsilon
      \int_{\Gamma_\varepsilon^{(0)}\times(0,T)} u_\varepsilon^2 \, d\sigma_x dt
     \Bigg)
 \le C_4 \varepsilon^\vartheta,
\end{equation*}
where $\vartheta:= \min \{4, \ 3 -\frac{2\alpha_0}{q+1}\}.$
\end{proof}

Thus, in consequence of (\ref{est-alpha_q}) we have the same three independent problems (\ref{main_al0<0}) to determine
$\omega_0^{(1)},$ $\omega_0^{(2)}$ and $\omega_0^{(3)}$ if conditions (\ref{kappa_0_conditions1}) take place instead of ${\bf C3}({\rm a}).$

To avoid cumbersome formulas and calculations, we consider the case $q=1,$ i.e. $\alpha_0\in[-1, 0),$ that is more typical and realistic.

\subsection{The case $\alpha_0\in(-1, 0)$}\label{-1<alpha_0<0}

For the regular parts of the approximation in each thin cylinder $\Omega^{(i)}_\varepsilon$ $(i \in \{1, 2, 3\}),$
we propose the following ansatz:
   \begin{equation}\label{regul_-1<al0<0}
   \sum\limits_{\mathfrak{n}\in\mathfrak{A}}
   \Bigg(
    \varepsilon^{\mathfrak{n}}\omega_\mathfrak{n}^{(i)}(x_i, t)
 +  \varepsilon^{\mathfrak{n}+2} u_{\mathfrak{n}+2}^{(i)}
    \left( \frac{\overline{x}_i}{\varepsilon}, x_i, t \right)
   \Bigg)
\end{equation}
where the index set $\mathfrak{A} = \{ 0, \ -\alpha_0, \ -2\alpha_0, \ 1+\alpha_0, \ 1, \ 1-\alpha_0 \};$
and for the inner part of the approximation in a neighborhood of the node $\Omega^{(0)}_\varepsilon$ the ansatz looks as follows
  \begin{equation}\label{junc_-1<al0<0}
    \varepsilon^{-\alpha_0} V_{-\alpha_0}(t)
 +  \varepsilon^{-2\alpha_0} V_{-2\alpha_0}(t)
 + \sum\limits_{\mathfrak{n}\in\mathfrak{I}}
   \bigg(
    \varepsilon^{\mathfrak{n}} V_\mathfrak{n}(t)
 +  \varepsilon^{\mathfrak{n}} N_{\mathfrak{n}}\left( \frac{x}{\varepsilon}, t \right)
   \bigg)
\end{equation}
where the index set $\mathfrak{I} = \{ 1, \ 1-\alpha_0, \ 2+\alpha_0, \ 2 \}.$

Similarly as was done in the subsection \ref{regul_asymp}, we obtain the linear inhomogeneous Neumann boundary-value problems
to define coefficients $\{{u}_{\mathfrak{n}+2}^{(i)}\}$:
\begin{equation}\label{regul_probl_n_-1<al0<0}
\left\{\begin{array}{rcl}
- \Delta_{\overline{\xi}_i}{u}_{\mathfrak{n}+2}^{(i)}
 & = &
 -   \, \dfrac{\partial\omega_\mathfrak{n}^{(i)}}{\partial{t}}
 +   \dfrac{\partial^2\omega_\mathfrak{n}^{(i)}}{\partial{x_i}^2}
 -   \delta_{0, \mathfrak{n}} k\big(\omega_0^{(i)}\big)
 -   (1-\delta_{\mathfrak{n}, 0})
     \bigg(
      k^\prime\big(\omega_0^{(i)}\big) \omega_{\mathfrak{n}}^{(i)}
 +    k^{\prime\prime}\big(\omega_0^{(i)}\big)
      K_{\mathfrak{n}}
      \Big(
       \{\omega_\mathfrak{j}^{(i)}\}_{\mathfrak{j}<\mathfrak{n}}
      \Big)
    \bigg)
\\[2mm]
 &   &
 +   \, \delta_{0, \mathfrak{n}} \, f_0^{(i)}
 +   \delta_{1, \mathfrak{n}} \, f_1^{(i)}
 \qquad \mbox{in} \quad \Upsilon_i (x_i),
\\[2mm]
\partial_{\boldsymbol{\nu}_{\overline{\xi}_i}}
{u}_{\mathfrak{n}+2}^{(i)}
 & = &
    h_i^\prime \dfrac{d\omega_\mathfrak{n}^{(i)}}{d{x_i}}
 -  \sum\limits_{\mathfrak{m}\in\mathfrak{A}} \delta_{\alpha_i, \mathfrak{m}+1}
    \Bigg(
     \delta_{\mathfrak{m}, \mathfrak{n}}
     \kappa_i\big(\omega_0^{(i)}, x_i, t \big)
\\[3mm]
 &   &
 +   \, (1-\delta_{\mathfrak{m}, \mathfrak{n}})
     \bigg(
      \partial_s \kappa_i\big(\omega_0^{(i)}, x_i, t \big)
      \omega_{\mathfrak{n}-\mathfrak{m}}^{(i)}
 +    \partial_{ss}^2 \kappa_i\big(\omega_0^{(i)}, x_i, t \big)
      K_{\mathfrak{n}-\mathfrak{m}}
      \Big(
       \{\omega_\mathfrak{j}^{(i)}
       \}_{\mathfrak{j}<{\mathfrak{n}-\mathfrak{m}}}
      \Big)
     \bigg)
    \Bigg)
\\[3mm]
 &   &
 +  \, \delta_{\beta_i, \mathfrak{n}+1} \, \varphi^{(i)}
 \qquad \mbox{on} \quad \partial\Upsilon_i(x_i),
\end{array}\right.
\end{equation}
where  $\mathfrak{n}, \mathfrak{j} \in\mathfrak{A}, \ i\in\{1,2,3\};$
$\partial_{ss}^2\kappa_i=\partial^2 \kappa_i /\partial s^2;$ and functions
$K_{\mathfrak{n}} := K_{\mathfrak{n}}\Big(
\{z_\mathfrak{j}\}_{\mathfrak{j}<{\mathfrak{n}}}\Big),
\ {\mathfrak{n}}\in\mathfrak{A},$
are defined by the formulas
\begin{equation*}
K_0 \equiv K_{-\alpha_0} \equiv K_{1+\alpha_0} \equiv 0, \quad
K_1 = z_{1+\alpha_0} z_{-\alpha_0}, \quad
K_{-2\alpha_0} = \tfrac12 z_{-\alpha_0}^2, \quad
K_{1-\alpha_0} = z_1 z_{-\alpha_0} + z_{1+\alpha_0} z_{-2\alpha_0}.
\end{equation*}
If $\mathfrak{n} - \mathfrak{m} \notin\mathfrak{A},$ then  $\omega_{\mathfrak{n} - \mathfrak{m}}^{(i)}\equiv0, \ K_{\mathfrak{n} - \mathfrak{m}}\equiv0.$ Also if $\alpha_i\neq{2+\alpha_0},$ then $\omega_{1+\alpha_0}^{(i)} \equiv 0.$

In (\ref{regul_probl_n_-1<al0<0}) the right-hand sides $f_0^{(i)}, f_1^{(i)}$ are defined in the subsection \ref{regul_asymp} and the variables $(x_i, t)$ are regarded as parameters from
$I_\varepsilon^{(i)} \times (0, T).$ Also, we should add conditions
$\langle u_{\mathfrak{n}+2}^{(i)} (\, \cdot \, , x_i, t) \rangle_{\Upsilon_i (x_i)} = 0$ to these problems
to guarantee the uniqueness of the solution.

By the same way as in subsection \ref{inner_asymp}, the coefficients $N_\mathfrak{n}, \ \mathfrak{n}\in\mathfrak{I},$ of the inner part of the asymptotics (\ref{junc_-1<al0<0}) are determined from the following relations:
\begin{equation}\label{junc_probl_n_-1<al0<0}
 \left\{\begin{array}{rcll}
  -\Delta_{\xi}{N_\mathfrak{n}}(\xi, t) & = &
   {F}_{\mathfrak{n}}(\xi, t), &
   \quad \xi\in\Xi,
\\[2mm]
   \partial_{{\boldsymbol \nu}_\xi}{N_{\mathfrak{n}}}(\xi, t) & = &
   B_{\mathfrak{n}}^{(0)}(\xi, t), &
   \quad \xi\in\Gamma_0,
\\[2mm]
   \partial_{{\boldsymbol \nu}_{\overline{\xi}_i}}{N_{\mathfrak{n}}}(\xi, t) & = &
   B_{\mathfrak{n}}^{(i)}(\xi, t), &
   \quad \xi\in\Gamma_i, \quad i=1,2,3,
\\[2mm]
   V_{\mathfrak{n}}(t) + N_{\mathfrak{n}}(\xi, t) & \sim &
   \omega^{(i)}_{\mathfrak{n}}(0, t) + \Psi^{(i)}_{\mathfrak{n}}(\xi, t), &
   \quad \xi_i \to +\infty, \ \ \overline{\xi}_i \in \Upsilon_i(0), \quad i=1,2,3.
 \end{array}\right.
\end{equation}
Whence, using the representation (\ref{new-solution}) (at $n=\mathfrak{n}\in\mathfrak{I}),$
we get the problem
\begin{equation}\label{junc_probl_general_-1<al0<0}
 \left\{\begin{array}{rcll}
 - \Delta_{\xi}{\widetilde{N}_{\mathfrak{n}}}(\xi, t) & = &
   \widetilde{F}_{\mathfrak{n}}(\xi, t), &
   \quad \xi\in\Xi,
\\[2mm]
   \partial_{\boldsymbol{\nu}_\xi}{\widetilde{N}_{\mathfrak{n}}}(\xi, t) & = &
   \widetilde{B}_{\mathfrak{n}}^{(0)}(\xi, t), &
   \quad \xi\in\Gamma_0,
\\[2mm]
   \partial_{\boldsymbol{\nu}_{\overline{\xi}_i}}{\widetilde{N}_{\mathfrak{n}}}(\xi, t) & = &
   \widetilde{B}_{\mathfrak{n}}^{(i)}(\xi, t), &
   \quad \xi\in\Gamma_i, \quad i=1,2,3,
 \end{array}\right.
\end{equation}
to determine ${\widetilde{N}_{\mathfrak{n}}}.$
As before, we demand that $\widetilde{N}_\mathfrak{n}$
satisfies the following stabilization conditions:
\begin{equation}\label{junc_probl_general+cond_al0}
   V_{\mathfrak{n}}(t) + \widetilde{N}_{\mathfrak{n}}(\xi, t)  \rightarrow
   \omega^{(i)}_{\mathfrak{n}}(0, t)
   \quad \text{as} \quad \xi_i \to +\infty, \ \ \overline{\xi}_i \in \Upsilon_i(0), \quad i=1,2,3.
\end{equation}
The variable $t$ in (\ref{junc_probl_n_-1<al0<0}) and (\ref{junc_probl_general_-1<al0<0})
is regarded as parameter from $(0, T).$
The right hand sides in the differential equations and boundary conditions on $\{\Gamma_i\}$
of the problems~ (\ref{junc_probl_n_-1<al0<0}), (\ref{junc_probl_general_-1<al0<0})
and the fourth conditions in (\ref{junc_probl_n_-1<al0<0}) are similarly obtained
as in subsection \ref{inner_asymp}. As a result, we get
\begin{equation*}\begin{array}{c}
\Psi_{\mathfrak{n}}^{(i)}(\xi, t)
 =   \xi_i\dfrac{\partial\omega_{\mathfrak{n}-1}^{(i)}}{\partial x_i}(0, t),
\quad \mathfrak{n}\in\mathfrak{I}\setminus\{2\},
\\[3mm]
\Psi_{2}^{(i)}(\xi, t)
 =   \dfrac{\xi_i^2}{2}\dfrac{\partial^2\omega_{0}^{(i)}}{\partial x_i^2}(0, t)
 +   \xi_i\dfrac{\partial\omega_{1}^{(i)}}{\partial x_i}(0, t)
 +   u_{2}^{(i)} (\overline{\xi}_i, 0, t),
\quad  i=1,2,3.
\end{array}\end{equation*}
$$
F_\mathfrak{n} \equiv 0, \quad \mathfrak{n}\in\mathfrak{I}\setminus\{2\},
\qquad F_2(\xi, t) = - \, k(0) + f(0, t),
$$
$$
\widetilde{F}_{\mathfrak{n}}(\xi, t)
 =  \sum\limits_{i=1}^3
    \Big(
     \xi_i\dfrac{\partial\omega_{\mathfrak{n}-1}^{(i)}}{\partial x_i}(0, t)
     \chi_i^{\prime\prime}(\xi_i)
 +  2\dfrac{\partial\omega_{\mathfrak{n}-1}^{(i)}}{\partial x_i}(0, t)
     \chi_i^{\prime}(\xi_i)
    \Big),
\quad \mathfrak{n}\in\mathfrak{I}\setminus\{2\},
$$
\begin{multline*}
\widetilde{F}_2(\xi, t)
 =  \sum\limits_{i=1}^3
    \Bigg[
    \bigg(
     \dfrac{\xi_i^2}{2}\dfrac{d^2\omega_{0}^{(i)}}{dx_i^2}(0, t)
 +   \xi_i\dfrac{\partial\omega_{1}^{(i)}}{\partial x_i}(0, t)
 +   u_{2}^{(i)} (\overline{\xi}_i, 0, t)
    \bigg)
    \chi_i^{\prime\prime}(\xi_i)
\\
 + 2\bigg(
     \xi_i\dfrac{\partial^2\omega_{0}^{(i)}}{\partial x_i^2}(0, t)
 +   \dfrac{\partial\omega_{1}^{(i)}}{\partial x_i}(0, t)
    \bigg)
    \chi_i^{\prime}(\xi_i)
    \Bigg]
 +  \bigg( 1 - \sum\limits_{i=1}^3 \chi_i(\xi_i) \bigg)
    \Big( f(0, t) - k(0) \Big),
\end{multline*}
\begin{gather*}
B_{1}^{(0)}(\xi, t) = \widetilde{B}_{1}^{(0)}(\xi, t)
 = - \kappa_0^\prime(0) V_{-\alpha_0}(t) + \delta_{\beta_0, 0} \varphi^{(0)}(\xi, t),
\\[1mm]
B_{1-\alpha_0}^{(0)}(\xi, t) = \widetilde{B}_{1-\alpha_0}^{(0)}(\xi, t)
 = - \kappa_0^\prime(0) V_{-2\alpha_0}(t)
 -   \frac12 \kappa_0^{\prime\prime}(0) V_{-\alpha_0}^2(t)
 +   \delta_{\beta_0, -\alpha_0} \varphi^{(0)}(\xi, t),
\\[1mm]
B_{2+\alpha_0}^{(0)}(\xi, t) = \widetilde{B}_{2+\alpha_0}^{(0)}(\xi, t)
 = - \kappa_0^\prime(0) \big( V_{1}(t) + N_{1}(\xi, t) \big)
 +   \delta_{\beta_0, 1+\alpha_0} \varphi^{(0)}(\xi, t),
\\[1mm]
B_{2}^{(0)}(\xi, t) = \widetilde{B}_{2}^{(0)}(\xi, t)
 = - \kappa_0^\prime(0) \big( V_{1-\alpha_0}(t) + N_{1-\alpha_0}(\xi, t) \big)
 -   \kappa_0^{\prime\prime}(0) \big( V_{1}(t) + N_{1}(\xi, t) \big) V_{-\alpha_0}(t)
 +   \delta_{\beta_0, 1} \varphi^{(0)}(\xi, t),
\end{gather*}
\begin{gather*}
B_{\mathfrak{n}}^{(i)} \equiv \widetilde{B}_{\mathfrak{n}}^{(i)}
\equiv 0, \quad  \mathfrak{n}\in\mathfrak{I}\setminus\{2\},
\qquad
B_{2}^{(i)}(\xi, t)
 = - \, \delta_{\alpha_i, 1} \, \kappa_i(0, 0, t)
 +   \delta_{\beta_i, 1} \, \varphi^{(i)} (\overline{\xi}_i, 0, t),
\\[1mm]
\widetilde{B}_{2}^{(i)} (\xi, t)
 =  \left(
 -   \, \delta_{\alpha_i, 1} \kappa_i \big(0, 0, t \big)
 +   \delta_{\beta_i, 1} \, \varphi^{(i)} (\overline{\xi}_i, 0, t)
    \right) \big( 1-\chi_i(\xi_i) \big),
\quad  i=1,2,3.
\end{gather*}

The existence of a solution of the problem (\ref{junc_probl_general_-1<al0<0})
in $\mathcal{H}$  follows from Proposition \ref{tverd1}.
In order to satisfy solvability conditions (\ref{solvability})
of the problem (\ref{junc_probl_general_-1<al0<0}) we choose the values
$V_{\mathfrak{n}-1-\alpha_0}, \ \mathfrak{n}\in\mathfrak{I}$ as follows: $V_{1+\alpha_0}\equiv 0,$
\begin{equation}\label{V}\begin{array}{l}
V_{-\alpha_0}(t)
 =   \dfrac{1}{\kappa_0^{\prime}(0) \, |\Gamma_0|_2}
     \bigg( \,
      \sum\limits_{i=1}^3 \dfrac{\partial \omega_0^{(i)}}{\partial x_i} (0, t)
 +    \delta_{\beta_0, 0} \int\limits_{\Gamma_0} \varphi^{(0)}(\xi, t) \, d\sigma_{\xi}
     \bigg),
\\[3mm]
V_{-2\alpha_0}(t)
 =   \dfrac{1}{\kappa_0^{\prime}(0) \, |\Gamma_0|_2}
     \bigg( \,
      \sum\limits_{i=1}^3 \dfrac{\partial \omega_{-\alpha_0}^{(i)}}{\partial x_i} (0, t)
 -    \frac12 \kappa_0^{\prime\prime}(0) |\Gamma_0|_2 V_{-\alpha_0}^2(t)
 +    \delta_{\beta_0, -\alpha_0}
      \int\limits_{\Gamma_0} \varphi^{(0)}(\xi, t) \, d\sigma_{\xi}
     \bigg),
\\[3mm]
V_{1}(t)
 =   \dfrac{1}{\kappa_0^{\prime}(0) \, |\Gamma_0|_2}
     \bigg( \,
      \sum\limits_{i=1}^3 \dfrac{\partial \omega_{1+\alpha_0}^{(i)}}{\partial x_i} (0, t)
 -    \frac12 \kappa_0^{\prime\prime}(0) \int\limits_{\Gamma_0} N_{1}(\xi, t) \, d\sigma_{\xi}
 +    \delta_{\beta_0, 1+\alpha_0} \int\limits_{\Gamma_0} \varphi^{(0)}(\xi, t) \, d\sigma_{\xi}
     \bigg),
\\[3mm]
V_{1-\alpha_0}(t)
 =   \dfrac{1}{\kappa_0^{\prime}(0) \, |\Gamma_0|_2}
     \bigg( \,
      \sum\limits_{i=1}^3 \dfrac{\partial \omega_{1}^{(i)}}{\partial x_i} (0, t)
 -    \kappa_0^{\prime}(0) \int\limits_{\Gamma_0} N_{1-\alpha_0}(\xi, t) \, d\sigma_{\xi}
 -    \kappa_0^{\prime\prime}(0) V_{-\alpha_0}(t)
      \int\limits_{\Gamma_0} \big( V_{1}(t) + N_{1}(\xi, t) \big) \, d\sigma_{\xi}
\\[2mm] \hspace{1.40cm}
 +    \, \ell_0 \sum\limits_{i=1}^3
      \Big(
       \pi h_i^2(0) \big( k(0)-f(0, t) \big)
 +     2\pi \, \delta_{\alpha_i, 1} \, h_i(0) \kappa_i(0, 0, t)
 -     \delta_{\beta_i, 1}
       \int\limits_{\Upsilon_i(0)} \varphi^{(i)}(\overline{\xi}_i, 0, t) \, dl_{\overline{\xi}_i}
      \Big)
\\[2mm] \hspace{1.40cm}
 -    \, |\Xi^{(0)}|_3 \big( k(0)-f(0, t) \big)
 +    \delta_{\beta_0, 1} \int\limits_{\Gamma_0} \varphi^{(0)}(\xi, t) \, d\sigma_{\xi}
     \bigg).
\end{array}\end{equation}
Again, according to Proposition \ref{tverd1}, the solution can be chosen in a unique way
to guarantee the asymptotics (\ref{inner_asympt_general}) with values
$\boldsymbol{\delta_{\mathfrak{n}}^{(2)}}$ and $\boldsymbol{\delta_{\mathfrak{n}}^{(3)}}$
(at $n=\mathfrak{n}\in\mathfrak{I}).$

It remains to satisfy the stabilization conditions (\ref{junc_probl_general+cond_al0}) at
$\mathfrak{n}\in\{1, \ 1-\alpha_0\}.$
Taking into account the asymptotics (\ref{inner_asympt_general}), we have to put
\begin{equation}\label{trans1_-1<al0<0}
\omega_\mathfrak{n}^{(1)}(0, t) = V_\mathfrak{n}(t), \qquad
\omega_\mathfrak{n}^{(2)}(0, t) = V_\mathfrak{n}(t) + \boldsymbol{\delta_\mathfrak{n}^{(2)}}(t),
\qquad
\omega_\mathfrak{n}^{(3)}(0, t) = V_\mathfrak{n}(t) + \boldsymbol{\delta_\mathfrak{n}^{(3)}}(t),
\quad \mathfrak{n}\in\{1, \ 1-\alpha_0\}.
\end{equation}
As a result, we get the solution of the problem (\ref{junc_probl_n_-1<al0<0})
with the following asymptotics:
\begin{equation}\label{inner_asympt_-1<al0<0}
{N}_{\mathfrak{n}}(\xi, t)
 =- {V}_{\mathfrak{n}}(t) + \omega_{\mathfrak{n}}^{(i)}(0, t)
 +  \Psi_\mathfrak{n}^{(i)}(\xi, t)
 +  {\mathcal O}(\exp(-\gamma_i\xi_i))
\quad \mbox{as} \ \ \xi_i\to+\infty, \qquad i=1,2,3.
\end{equation}

To complete matching the regular and inner asymptotics, we put
\begin{equation}\label{trans2_-1<al0<0}
\omega_{1+\alpha_0}^{(i)}(0, t) = 0, \qquad
\omega_{-\alpha_0}^{(i)}(0, t) = V_{-\alpha_0}(t), \qquad
\omega_{-2\alpha_0}^{(i)}(0, t) = V_{-2\alpha_0}(t), \qquad i=1,2,3.
\end{equation}
With the help of the necessary and sufficient condition for the solvability
of the problem (\ref{regul_probl_n_-1<al0<0}) and conditions (\ref{trans1_-1<al0<0}),
(\ref{trans2_-1<al0<0}), we get the following problems for
$\omega_\mathfrak{n}^{(1)},$ $\omega_\mathfrak{n}^{(2)}$ and $\omega_\mathfrak{n}^{(3)}$  $(\mathfrak{n}\in\mathfrak{A}\setminus\{0\}):$
\begin{equation}\label{omega_probl*_-1<al0<0}
 \left\{\begin{array}{l}
    \pi h_i^2(x_i) \dfrac{\partial\omega_{\mathfrak{n}}^{(i)}}{\partial{t}} (x_i, t)
  - \pi \dfrac{\partial}{\partial{x_i}}
    \left(h_i^2(x_i)\dfrac{\partial\omega_{\mathfrak{n}}^{(i)}}{\partial{x_i}}(x_i, t)\right)
  + \pi h_i^2(x_i) k^\prime\Big(\omega_0^{(i)}(x_i, t)\Big)
    \omega_\mathfrak{n}^{(i)}(x_i, t)
 \\[5mm]
 \begin{array}{rclr}
  + \ 2\pi \, \delta_{\alpha_i, 1} \, h_i(x_i)
    \partial_s \kappa_i\Big(\omega_0^{(i)}(x_i, t), x_i, t \Big)
    \omega_\mathfrak{n}^{(i)}(x_i, t) & = &
    \widehat{F}_\mathfrak{n}^{(i)}(x_i, t), \ & (x_i, t)\in I_i \times (0, T),
 \\[3mm]
    \omega_\mathfrak{n}^{(i)}(0, t) \hspace{0.26cm} = \hspace{0.26cm}
    V_\mathfrak{n}(t) + \boldsymbol{\delta_\mathfrak{n}^{(i)}}(t), \qquad
    \omega_\mathfrak{n}^{(i)} (\ell_i, t) & = &
    0, & t\in(0, T),
 \\[2mm]
    \omega_{\mathfrak{n}}^{(i)} (x_i, 0) & = &
    0, & x_i\in I_i,
 \end{array}
 \end{array}\right.
\end{equation}
for each $i\in\{1, 2, 3\}.$ Here the values $V_{\mathfrak{n}}$ are defined in (\ref{V}),
\begin{multline*}
\widehat{F}_\mathfrak{n}^{(i)}(x_i, t)
 = - \pi h_i^2(x_i) k^{\prime\prime}\Big(\omega_0^{(i)}(x_i, t)\Big)
     K_\mathfrak{n}
     \Big(
      \{\omega_\mathfrak{j}^{(i)}(x_i, t)
      \}_{\mathfrak{j}<\mathfrak{n}}
     \Big)
 -  2 \pi h_i(x_i)
    \sum\limits_{\mathfrak{m}\in\mathfrak{R}} \delta_{\alpha_i, \mathfrak{m}+1}
    \Bigg(
     \delta_{\mathfrak{m}, \mathfrak{n}}
     \kappa_i\Big(\omega_0^{(i)}(x_i, t), x_i, t \Big)
\\
 +   (1-\delta_{\mathfrak{m}, \mathfrak{n}})
     \bigg(
      \partial_s \kappa_i\Big(\omega_0^{(i)}(x_i, t), x_i, t \Big)
      \omega_{\mathfrak{n}-\mathfrak{m}}^{(i)}
 +    \partial_{ss}^2 \kappa_i\Big(\omega_0^{(i)}(x_i, t), x_i, t \Big)
      K_{\mathfrak{n}-\mathfrak{m}}
      \Big(
       \{\omega_\mathfrak{j}^{(i)}(x_i, t)
       \}_{\mathfrak{j}<{\mathfrak{n}-\mathfrak{m}}}
      \Big)
     \bigg)
    \Bigg)
\\
 +   \delta_{1, \mathfrak{n}}
     \int\limits_{\Upsilon_i(x_i)} f_1^{(i)} (\overline{\xi}_i, x_i, t) \, d{\overline{\xi}_i}
 +   \delta_{\beta_i, \mathfrak{n}+1} \int\limits_{\partial\Upsilon_i(x_i)}
     \varphi^{(i)}(\overline{\xi}_i, x_i, t) \, dl_{\overline{\xi}_i},
\quad (x_i, t)\in I_i \times (0, T), \quad i=1,2,3;
\end{multline*}
the values
$\boldsymbol{\delta_{1+\alpha_0}^{(i)}} =\boldsymbol{\delta_{-\alpha_0}^{(i)}}= \boldsymbol{\delta_{-2\alpha_0}^{(i)}} =0, \ i\in\{1,2,3\},$
$\boldsymbol{\delta_1^{(1)}} =  \boldsymbol{\delta_{1-\alpha_0}^{(1)}} = 0,$ and
$\boldsymbol{\delta_1^{(2)}}, \ \boldsymbol{\delta_1^{(3)}}$ and
$\boldsymbol{\delta_{1-\alpha_0}^{(2)}}, \ \boldsymbol{\delta_{1-\alpha_0}^{(3)}}$
are uniquely determined (see Remark~\ref{remark_constant}) by formulas
\begin{multline}\label{t-delta-1}
\boldsymbol{\delta_1^{(i)}}(t)
 =  \int\limits_{\Xi} \mathfrak{N}_i (\xi) \,
    \sum\limits_{j=1}^3
    \bigg(
     \xi_j\dfrac{\partial\omega_0^{(j)}}{\partial{x_j}}(0, t) \chi_j^{\prime\prime}(\xi_j)
 +  2\dfrac{\partial\omega_0^{(j)}}{\partial{x_j}}(0, t) \chi_j^{\prime}(\xi_j)
    \bigg) \, d\xi
\\
 -  \kappa_0^{\prime}(0) V_{-\alpha_0}(t)
    \int\limits_{\Gamma_0} \mathfrak{N}_i(\xi) \, d\sigma_\xi
 +  \delta_{\beta_0, 0}
    \int\limits_{\Gamma_0} \mathfrak{N}_i(\xi) \, \varphi^{(0)} (\xi, t) \, d\sigma_\xi,
\quad i=2,3,
\end{multline}
\begin{multline}\label{t-delta-2}
\boldsymbol{\delta_{1-\alpha_0}^{(i)}}(t)
 =  \int\limits_{\Xi} \mathfrak{N}_i (\xi) \,
    \sum\limits_{j=1}^3
    \bigg(
     \xi_j\dfrac{\partial\omega_{-\alpha_0}^{(j)}}{\partial{x_j}}(0, t)
     \chi_j^{\prime\prime}(\xi_j)
 +   2\dfrac{\partial\omega_{-\alpha_0}^{(j)}}{\partial{x_j}}(0, t) \chi_j^{\prime}(\xi_j)
    \bigg) \, d\xi
\\
 -  \Big(
     \kappa_0^{\prime}(0) V_{-2\alpha_0}(t)
 +   \tfrac12 \kappa_0^{\prime\prime}(0) V_{-\alpha_0}^2(t)
    \Big)
    \int\limits_{\Gamma_0} \mathfrak{N}_i(\xi) \, d\sigma_\xi
 +  \delta_{\beta_0, -\alpha_0}
    \int\limits_{\Gamma_0} \mathfrak{N}_i(\xi) \, \varphi^{(0)} (\xi, t) \, d\sigma_\xi,
\quad i=2,3,
\end{multline}
where $\mathfrak{N}_2$ and $\mathfrak{N}_3$ are defined in Proposition \ref{tverd2}.

The determination  of the terms of the asymptotics is carried out according to the following scheme:
\begin{center}
\begin{tikzpicture}
  \matrix (m) [matrix of math nodes,row sep=3em,column sep=2em,minimum width=2em]
  {
     &  &  &  & \{\omega_{1+\alpha_0}^{(i)}\}_{i=1}^3 & N_{2+\alpha_0} \\
     & \{\omega_{0}^{(i)}\}_{i=1}^3 & N_{1} & V_{1} & \{\omega_{1}^{(i)}\}_{i=1}^3 & N_{2} \\
   V_{-\alpha_0} & \{\omega_{-\alpha_0}^{(i)}\}_{i=1}^3 & N_{1-\alpha_0} & V_{1-\alpha_0} &
   \{\omega_{1-\alpha_0}^{(i)}\}_{i=1}^3 &  \\
   V_{-2\alpha_0} & \{\omega_{-2\alpha_0}^{(i)}\}_{i=1}^3 &  &  &  &  \\
  };
  \path[-stealth]
    (m-1-5) edge [dashed] (m-2-4)
    (m-2-2) edge [dashed] (m-1-5) edge (m-3-1)
    (m-2-3) edge (m-2-4)
    (m-2-4) edge (m-1-6) edge (m-2-5)
    (m-2-5) edge (m-3-4)
    (m-3-1) edge (m-2-3) edge (m-3-2)
    (m-3-2) edge [dashed] (m-2-5) edge (m-4-1)
    (m-3-3) edge (m-3-4)
    (m-3-4) edge (m-2-6) edge (m-3-5)
    (m-4-1) edge (m-3-3) edge (m-4-2)
    (m-4-2) edge [dashed] (m-3-5);
\end{tikzpicture}
\end{center}
\begin{center}
\begin{minipage}[t]{14.3cm}
{\sf Comments to the scheme.} The arrows indicate the order for determining the terms of the asymptotics.
We start with elements $\{\omega_{0}^{(i)}\}_{i=1}^3$
(see (\ref{main_al0<0})) and move across the arrows.
Here the terms$\{\omega_\mathfrak{n}^{(i)}\}_{i=1}^3, \ \mathfrak{n}\in\mathfrak{A}\setminus\{0\}$
and $N_\mathfrak{n}, \ \mathfrak{n}\in\mathfrak{I}$ are determined from the problems
(\ref{omega_probl*_-1<al0<0}) and (\ref{junc_probl_n_-1<al0<0}), respectively;
the values $V_{\mathfrak{n}-1-\alpha_0}, \ \mathfrak{n}\in\mathfrak{I}$ are defined in (\ref{V}).
If $\alpha_j \neq 2+\alpha_0$ for some $j\in\{1,2,3\},$ then
$\omega_{1+\alpha_0}^{(j)}\equiv0$ (see (\ref{regul_probl_n_-1<al0<0}) and comments below)
and term $\omega_{1-\alpha_0}^{(j)}$ does not depend on $\omega_{-2\alpha_0}^{(j)}.$ If
$\alpha_j \neq 2+\alpha_0$ for all $j\in\{1,2,3\},$ then the dashed arrows disappear
and we don't need to find the elements $\{\omega_{-2\alpha_0}^{(i)}\}_{i=1}^3.$
The approximation does not contain the terms $N_{2+\alpha_0}$ and $N_2,$ they are only needed to
find the values $V_1$ and $V_{1-\alpha_0}.$
\end{minipage}
\end{center}

\medskip

Thus, the asymptotic approximation in the case $\alpha_0\in (-1, 0)$ has the following form:
\begin{multline}\label{asymp_expansion_-1<al0<0}
{U}_\varepsilon^{(1-\alpha_0)} (x, t)
 = \sum\limits_{i=1}^3 \chi_{\ell_0}^{(i)} \left(\frac{x_i}{\varepsilon^\mathfrak{a}}\right)
    \left(
     \omega_0^{(i)} (x_i, t)
 +   \varepsilon^{-\alpha_0} \, \omega_{-\alpha_0}^{(i)} (x_i, t)
 +   \varepsilon \, \omega_1^{(i)} (x_i, t)
 +   \varepsilon^{1-\alpha_0} \, \omega_{1-\alpha_0}^{(i)} (x_i, t)
    \right)
\\
 +  \left(1 - \sum\limits_{i=1}^3 \chi_{\ell_0}^{(i)}
     \left(\frac{x_i}{\varepsilon^\mathfrak{a}}\right)
    \right)
    \Bigg(
     \varepsilon^{-\alpha_0} \, V_{-\alpha_0}(t)
 +   \varepsilon \bigg(V_1(t) + N_1 \left( \frac{x}{\varepsilon}, t \right)\bigg)
 +   \varepsilon^{1-\alpha_0}
     \bigg(V_{1-\alpha_0}(t) + N_{1-\alpha_0} \left( \frac{x}{\varepsilon}, t \right)\bigg)
    \Bigg),
\\
(x, t)\in\Omega_\varepsilon\times(0, T),
\end{multline}
where $\mathfrak{a}$ is a fixed number from the interval $\big(\frac23, 1 \big),$
and $\{\chi_{\ell_0}^{(i)}\}_{i=1}^3$ are defined in (\ref{cut-off-functions}).

\begin{theorem}\label{mainTheorem_-1<al0<0}
Let assumptions made in the statement of the problem (\ref{probl}) and (\ref{kappa_0_conditions1}), (\ref{kappa_0_conditions2}) at $q=1$ are satisfied. Then the sum (\ref{asymp_expansion_-1<al0<0}) is the asymptotic approximation for the solution
$u_\varepsilon$ to the boundary-value problem~$(\ref{probl}),$ i.e., \quad
$
      \exists \, {C}_0 >0 \ \ \exists \, \varepsilon_0>0 \ \
      \forall\, \varepsilon\in(0, \varepsilon_0):
$
\begin{equation}\label{t0_-1<al0<0}
      \max\limits_{t\in[0,T]}
      \left\|
       {U}_\varepsilon^{(1-\alpha_0)} (\cdot, t) - u_\varepsilon (\cdot, t)
      \right\|_{L^2(\Omega_\varepsilon)}
  +   \left\|
       {U}_\varepsilon^{(1-\alpha_0)} - u_\varepsilon
      \right\|_{L^2(0, T; {H}^1(\Omega_\varepsilon))}
 \leq {C}_0 \, \mu_0(\varepsilon),
\end{equation}
where $\mu_0(\varepsilon) = o(\varepsilon)$ as $\varepsilon \to 0$ and
\begin{multline}\label{mu_-1<al0<0}
\mu_0(\varepsilon)
 = \bigg(
    \varepsilon^{1+\frac{\mathfrak{a}}2}
 +  \sum_{i=1}^3
     \Big(
      (1-\delta_{\alpha_{i}, 1}-\delta_{\alpha_{i}, 1-\alpha_0}) \varepsilon^{\alpha_i}
 +    (1-\delta_{\beta_{i}, 1}-\delta_{\beta_{i}, 1-\alpha_0}) \varepsilon^{\beta_i}
     \Big)
\\
 +   \varepsilon^{2+\alpha_0} + \varepsilon^{1-\alpha_0}
 +   (1-\delta_{\beta_{0}, 0}-\delta_{\beta_{0}, -\alpha_0})\varepsilon^{\beta_0+1}
   \bigg).
\end{multline}
\end{theorem}
\begin{proof}
The proof of Theorem \ref{mainTheorem_-1<al0<0} repeats the proof of Theorem \ref{mainTheorem}.
To avoid huge amount of calculations we note the main differences.

The residual $\widehat{R}_{\varepsilon}$ in the differential equation in the whole domain
$\Omega_\varepsilon$ and the residuals $\breve{R}_{\varepsilon}^{(i)}$ in the boundary
conditions on the surfaces $\Gamma_i$ of the thin cylinders
$\Omega_\varepsilon^{(i)} \ (i\in\{1,2,3\})$ can be similarly obtained and estimated.

Let us consider the residual that asymptotic approximation (\ref{asymp_expansion_-1<al0<0})
leaves in the boundary condition on the node. We get
\begin{equation*}
   \partial_\nu {U}_\varepsilon^{(1-\alpha_0)}
 + \varepsilon^{\alpha_0} \kappa_0\Big( {U}_\varepsilon^{(1-\alpha_0)}\Big)
 - \varepsilon^{\beta_0} \varphi_\varepsilon^{(0)}
 = \breve{R}_{\varepsilon}^{(0)}
\quad \mbox{on} \quad \Gamma_\varepsilon^{(0)}\times(0, T),
\end{equation*}
where
\begin{multline*}
\breve{R}_{\varepsilon}^{(0)}(x, t)
 =  \varepsilon^{\alpha_0} \kappa_0\Big({U}_\varepsilon^{(1-\alpha_0)}(x, t)\Big)
 -  \kappa_0^{\prime}(0) V_{-\alpha_0}(t)
 -  \varepsilon^{-\alpha_0} \kappa_0^{\prime}(0) V_{-2\alpha_0}(t)
 -  \varepsilon^{-\alpha_0} \tfrac12 \, \kappa_0^{\prime\prime}(0) \, V_{-\alpha_0}^2(t)
\\
 +  (\delta_{\beta_0, 0} + \delta_{\beta_0, -\alpha_0} - 1) \,
    \varepsilon^{\beta_0} \varphi_\varepsilon^{(0)}(x, t),
\qquad (x, t)\in\Gamma_\varepsilon^{(0)}\times(0, T).
\end{multline*}
Denote by
$$
\boldsymbol{N}_\varepsilon^{(1-\alpha_0)}(x, t):=
   \varepsilon^{-\alpha_0}V_{-\alpha_0}(t)
 + \varepsilon\bigg(V_1(t)+N_1\Big(\frac{x}{\varepsilon}, t \Big)\bigg)
 + \varepsilon^{1-\alpha_0}
   \bigg(V_{1-\alpha_0}(t)+N_{1-\alpha_0}\Big(\frac{x}{\varepsilon}, t\Big)\bigg).
$$
Taking into account that
${U}_\varepsilon^{(1-\alpha_0)}=\boldsymbol{N}_\varepsilon^{(1-\alpha_0)}$
on $\Gamma_\varepsilon^{(0)}$ and using Taylor's formula
\begin{equation*}
\kappa_0\Big({U}_\varepsilon^{(1-\alpha_0)}\Big)
 = \kappa_0^{\prime}(0) \boldsymbol{N}_\varepsilon^{(1-\alpha_0)}
 + \int\limits_{0}^{\boldsymbol{N}_\varepsilon^{(1-\alpha_0)}}
   \Big(\boldsymbol{N}_\varepsilon^{(1-\alpha_0)} - s\Big)
   \kappa_0^{\prime\prime}(s) \, ds,
\end{equation*}
we rewrite $\breve{R}_{\varepsilon}^{(0)}$ in the following form:
\begin{multline*}
\breve{R}_{\varepsilon}^{(0)}(x, t)
 =  \varepsilon^{1+\alpha_0} \kappa_0^{\prime}(0)
    \bigg(V_1(t)+N_1\Big(\frac{x}{\varepsilon}, t \Big)\bigg)
 +  \varepsilon \kappa_0^{\prime}(0)
    \bigg(V_{1-\alpha_0}(t)+N_{1-\alpha_0}\Big(\frac{x}{\varepsilon}, t\Big)\bigg)
\\
 +  \varepsilon^{\alpha_0} \int\limits_{0}^{\boldsymbol{N}_\varepsilon^{(1-\alpha_0)}(x, t)}
    \Big(\boldsymbol{N}_\varepsilon^{(1-\alpha_0)}(x, t) - s\Big)
    \kappa_0^{\prime\prime}(s) \, ds
 -  \varepsilon^{-\alpha_0} \kappa_0^{\prime}(0) V_{-2\alpha_0}(t)
 -  \varepsilon^{-\alpha_0} \tfrac12 \kappa_0^{\prime\prime}(0) V_{-\alpha_0}^2(t)
\\
 +  (\delta_{\beta_0, 0} + \delta_{\beta_0, -\alpha_0} - 1) \,
    \varepsilon^{\beta_0} \varphi_\varepsilon^{(0)}(x, t),
\qquad (x, t)\in\Gamma_\varepsilon^{(0)}\times(0, T).
\end{multline*}
With the help of (\ref{ineq-3}), we obtain
\begin{equation*}
|R_{\varepsilon, 3} (v)|
 =   \left| \int_{\Gamma_{\varepsilon}^{(0)}}\breve{R}_{\varepsilon}^{(0)} v \, d\sigma_x
     \right|
 \le \check{C} \sqrt{| \Gamma_0 \, |_2} \,
     \big(
      \varepsilon^{1+\alpha_0} + \varepsilon^{-\alpha_0}
 +    \varepsilon^{\beta_0} ( 1 - \delta_{\beta_0, 0} - \delta_{\beta_0, -\alpha_0})
     \big) \, \varepsilon \, \| v \|_{H^1(\Omega_\varepsilon)},
\end{equation*}
for all $v\in L^2(0, T; \mathcal{H}_\varepsilon)$ and a.e. $t\in(0, T).$
\end{proof}

\subsection{The case $\alpha_0=-1$}\label{alpha_0=-1}

In this case we take ansatzes (\ref{regul}) for the approximation in each thin cylinder $\Omega^{(i)}_\varepsilon$ $(i \in \{1, 2, 3\})$
and entirely repeat all calculations from the subsection~\ref{regul_asymp}.
In a neighborhood of the node $\Omega^{(0)}_\varepsilon$ we consider only one term
$$
\varepsilon N_1\left(\frac{x}{\varepsilon}, t\right).
$$

Similarly as in subsection \ref{inner_asymp} we derive the following relations for $N_1$:
\begin{equation}\label{junc_probl_1_al0=-1}
 \left\{\begin{array}{rcll}
  -\Delta_{\xi}{N_1}(\xi, t) & = &
   0, &
   \quad \xi\in\Xi,
\\[2mm]
   \partial_{{\boldsymbol \nu}_\xi}{N_1}(\xi, t)
 + \kappa_0^{\prime}(0){N_1}(\xi, t) & = &
   B_{1}^{(0)}(\xi, t), &
   \quad \xi\in\Gamma_0,
\\[2mm]
   \partial_{{\boldsymbol \nu}_{\overline{\xi}_i}}{N_1}(\xi, t) & = &
   0, &
   \quad \xi\in\Gamma_i, \quad i=1,2,3,
\\[2mm]
   N_1(\xi, t) & \sim &
   \omega^{(i)}_{1}(0, t) + \Psi^{(i)}_{1}(\xi, t), &
   \quad \xi_i \to +\infty, \ \ \overline{\xi}_i \in \Upsilon_i(0), \quad i=1,2,3.
 \end{array}\right.
\end{equation}
With the help of the representation (\ref{new-solution}) (at $n=1),$ we obtain the problem
\begin{equation}\label{junc_probl_general_al0=-1}
 \left\{\begin{array}{rcll}
 - \Delta_{\xi}{\widetilde{N}_1}(\xi, t) & = &
   \widetilde{F}_1(\xi, t), &
   \quad \xi\in\Xi,
\\[2mm]
   \partial_{\boldsymbol{\nu}_\xi}{\widetilde{N}_1}(\xi, t)
 + \kappa_0^{\prime}(0){\widetilde{N}_1}(\xi, t) & = &
   \widetilde{B}_{1}^{(0)}(\xi, t), &
   \quad \xi\in\Gamma_0,
\\[2mm]
   \partial_{\boldsymbol{\nu}_{\overline{\xi}_i}}{\widetilde{N}_1}(\xi, t) & = &
   0, &
   \quad \xi\in\Gamma_i, \quad i=1,2,3,
\\[2mm]
   \widetilde{N}_1(\xi, t) & \rightarrow &
   \omega^{(i)}_{1}(0, t) &
   \quad \text{as} \quad \xi_i \to +\infty, \ \ \overline{\xi}_i \in \Upsilon_i(0), \quad i=1,2,3,
 \end{array}\right.
\end{equation}
to determine ${\widetilde{N}_1}.$
Here $\{\Psi^{(i)}_{1}\}_{i=1}^3, \ \widetilde{F}_1$
are the same as in subsection \ref{inner_asymp}, and
$
B_{1}^{(0)} = \widetilde{B}_{1}^{(0)} = \delta_{\beta_0, 0} \, \varphi^{(0)}.
$

Similarly as in subsection \ref{inner_asymp}, we introduce the space $\mathcal{H}$ and prove
the existence of a unique weak solution to the problem (\ref{junc_probl_general_al0=-1}). But in contrast to the problem (\ref{junc_probl_general})
we have the Robin condition on $\Gamma_0.$
\begin{definition}
A function $\widetilde{N}_1$ from the space $\mathcal{H}$ is called a weak solution of the problem
(\ref{junc_probl_general_al0=-1}) if the identity
\begin{equation*}
    \int\limits_{\Xi} \nabla \widetilde{N}_1 \cdot \nabla v \, d\xi
 +  \kappa_0^{\prime}(0)\int\limits_{\Gamma_0} \widetilde{N}_1 \, v \, d\sigma_\xi
 =  \int\limits_{\Xi} \widetilde{F}_1 \, v \, d\xi
 +  \int\limits_{\Gamma_0} \widetilde{B}^{(0)}_1 \, v \, d\sigma_\xi
\end{equation*}
holds for all $v\in\mathcal{H}$.
\end{definition}
\begin{proposition}\label{tverd1_al0=-1}
   Let  $\rho^{-1} \widetilde{F}_1 (\cdot\,, t) \in L^2(\Xi), \
    \widetilde{B}^{(0)}_{1} (\cdot\,, t) \in L^2(\Gamma_0)$ for a.e.  $t\in (0, T).$
Then there exist a unique weak solution of problem (\ref{junc_probl_general_al0=-1})
with the following differentiable asymptotics:
\begin{equation}\label{inner_asympt_general_al0=-1}
\widetilde{N}_1(\xi,t)=\left\{
\begin{array}{rl}
    \boldsymbol{\delta_1^{(1)}}(t)
 +  {\mathcal O}\big(\exp( - \gamma_1 \xi_1)\big) & \mbox{as} \ \ \xi_1\to+\infty,
\\[2mm]
    \boldsymbol{\delta_1^{(2)}}(t)
 +  {\mathcal O}\big(\exp( - \gamma_2 \xi_2)\big) & \mbox{as} \ \ \xi_2\to+\infty,
\\[2mm]
    \boldsymbol{\delta_1^{(3)}}(t)
 +  {\mathcal O}\big(\exp( - \gamma_3 \xi_3)\big) & \mbox{as} \ \ \xi_3\to+\infty,
\end{array}
\right.
\end{equation}
where $\gamma_i, \ i=1,2,3 $ are positive constants.
\end{proposition}
The values $\{\boldsymbol{\delta_1^{(i)}}\}_{i=1}^3$
in (\ref{inner_asympt_general_al0=-1}) are defined as follows:
\begin{equation}\label{const_d_0_al0=-1}
\boldsymbol{\delta_1^{(i)}} (t)
 =  \int_{\Xi} \mathfrak{N}_i(\xi) \, \widetilde{F}_1(\xi, t) \, d\xi
 +  \int_{\Gamma_0} \mathfrak{N}_i(\xi) \, \widetilde{B}_1^{(0)}(\xi, t) \, d\sigma_\xi ,
\quad i=1,2,3,
\end{equation}
where $\{\mathfrak{N}_i\}_{i=1}^3$ are special solutions to
the corresponding homogeneous problem
\begin{equation}\label{hom_probl_al0=-1}
  -\Delta_{\xi}\mathfrak{N} = 0 \ \ \text{in} \ \ \Xi, \qquad
  \partial_\nu \mathfrak{N} + \kappa_0^{\prime}(0)\mathfrak{N}
                            = 0 \ \ \text{on} \ \ \Gamma_0, \qquad
  \partial_\nu \mathfrak{N} = 0 \ \ \text{on} \ \ \partial\Xi\setminus\Gamma_0
\end{equation}
for the problem (\ref{junc_probl_general_al0=-1}).

\begin{proposition}\label{tverd2_al0=-1}
The  problem (\ref{hom_probl_al0=-1}) has three linearly independent solutions
$\{\mathfrak{N}_i\}_{i=1}^3$ that do not belong to the space
${\mathcal H}$ and they have the following differentiable asymptotics:
\begin{equation}\label{inner_asympt_hom_solution_al0=-1}
\mathfrak{N}_i(\xi) = \left\{
\begin{array}{rl}
    C_i^{(1)} + \delta_{i,1} \ \dfrac{\xi_1}{\pi h_1^2(0)}
 +  {\mathcal O}\big(\exp( - \gamma_1 \xi_1)\big) & \mbox{as} \ \ \xi_1\to+\infty,
 \\[3mm]
    C_i^{(2)} + \delta_{i,2} \ \dfrac{\xi_2}{\pi h_2^2(0)}
 +  {\mathcal O}\big(\exp( - \gamma_2 \xi_2)\big) & \mbox{as} \ \ \xi_2\to+\infty,
 \\[3mm]
    C_i^{(3)} + \delta_{i,3} \ \dfrac{\xi_3}{\pi h_3^2(0)}
 +  {\mathcal O}\big(\exp( - \gamma_3 \xi_3)\big) & \mbox{as} \ \ \xi_3\to+\infty,
\end{array}
\right.
\qquad i=1,2,3.
\end{equation}
Any other solution to the homogeneous problem, which has polynomial growth at infinity,
can be presented as a linear combination
$ \mathfrak{c}_1 \mathfrak{N}_1 + \mathfrak{c}_2 \mathfrak{N}_2 + \mathfrak{c}_3 \mathfrak{N}_3.$
\end{proposition}

In order to satisfy the forth condition in (\ref{junc_probl_general_al0=-1}), we have to put
\begin{equation}\label{trans1_al0=-1}
\omega_1^{(i)}(0, t) = \boldsymbol{\delta_1^{(i)}}(t), \qquad i=1,2,3.
\end{equation}
As a result, we get the solution of the problem (\ref{junc_probl_1_al0=-1}) with the following asymptotics:
\begin{equation}\label{inner_asympt_al0=-1}
{N}_{1}(\xi, t)
 =  \omega_{1}^{(i)} (0, t) + \Psi_1^{(i)} (\xi, t)
 +  {\mathcal O}(\exp(-\gamma_i\xi_i))
\quad \mbox{as} \ \ \xi_i\to+\infty, \quad i=1,2,3.
\end{equation}

Taking into account (\ref{trans1_al0=-1}) we derive for each $i\in\{1,2,3\}$ the problem
\begin{equation}\label{omega_probl*_al0=-1}
 \left\{\begin{array}{l}
    \pi h_i^2(x_i) \dfrac{\partial\omega_1^{(i)}}{\partial{t}} (x_i, t)
  - \pi \dfrac{\partial}{\partial{x_i}}
    \left(h_i^2(x_i)\dfrac{\partial\omega_1^{(i)}}{\partial{x_i}}(x_i, t)\right)
  + \pi h_i^2(x_i) k^\prime\Big(\omega_0^{(i)}(x_i, t)\Big)
    \omega_1^{(i)}(x_i, t)
 \\[5mm]
 \begin{array}{rclr}
  + \ 2\pi \, \delta_{\alpha_i, 1} \, h_i(x_i)
    \partial_s \kappa_i\Big(\omega_0^{(i)}(x_i, t), x_i, t \Big) \omega_1^{(i)}(x_i, t)& = &
    \widehat{F}_1^{(i)}(x_i, t), & (x_i, t)\in I_i \times (0, T),
 \\[4mm]
    \omega_1^{(i)}(0, t) \hspace{0.26cm} = \hspace{0.26cm}
    \boldsymbol{\delta_1^{(i)}}(t), \qquad
    \omega_1^{(i)} (\ell_i, t) & = &
    0, & t\in(0, T),
 \\[3mm]
    \omega_{1}^{(i)} (x_i, 0)
    & = &
    0, & x_i\in I_i,
 \end{array}
 \end{array}\right.
\end{equation}
to determine uniquely $\omega_1^{(i)}.$
 Here $\widehat{F}_1^{(i)}$ are defined in (\ref{omegaF1}), and
\begin{equation}\label{t-delta-3}
\boldsymbol{\delta_1^{(i)}}(t)
 =  \int_{\Xi} \mathfrak{N}_i (\xi) \,
    \sum\limits_{j=1}^3 \dfrac{\partial\omega_0^{(j)}}{\partial{x_j}}(0, t)
    \Big( \xi_j \chi_j^{\prime\prime}(\xi_j) + 2 \chi_j^{\prime}(\xi_j) \Big) \, d\xi
 +  \delta_{\beta_0, 0}
    \int_{\Gamma_0} \mathfrak{N}_i (\xi) \, \varphi^{(0)} (\xi, t) \, d\sigma_\xi,
\quad i=1,2,3,
\end{equation}
where $\{\mathfrak{N}_i\}_{i=1}^3$ are defined in Proposition \ref{tverd2_al0=-1}.

With the help of  $\{\omega_0, \ \omega_1\}_{i=1}^3, \ N_1$
(see (\ref{main_al0<0}), (\ref{omega_probl*_al0=-1}), (\ref{junc_probl_1_al0=-1}), respectively)
we construct the following asymptotic approximation:
\begin{multline}\label{asymp_expansion_al0=-1}
{U}_\varepsilon^{(1)} (x, t)
 = \sum\limits_{i=1}^3 \chi_{\ell_0}^{(i)} \left(\frac{x_i}{\varepsilon^\mathfrak{a}}\right)
    \left( \omega_0^{(i)} (x_i, t) + \varepsilon \, \omega_1^{(i)} (x_i, t) \right)
 +  \left(1 - \sum\limits_{i=1}^3 \chi_{\ell_0}^{(i)}
     \left(\frac{x_i}{\varepsilon^\mathfrak{a}}\right)
    \right)
    \varepsilon N_1 \left( \frac{x}{\varepsilon}, t \right),
\\
\quad (x, t)\in\Omega_\varepsilon\times(0, T),
\end{multline}
where $\mathfrak{a}$ is a fixed number from the interval $\big(\frac23, 1 \big),$
and $\{\chi_{\ell_0}^{(i)}\}_{i=1}^3$ are defined in (\ref{cut-off-functions}).

\begin{theorem}\label{mainTheorem_al0=-1}
Let assumptions made in the statement of the problem (\ref{probl}) and
(\ref{kappa_0_conditions1}) at $q=1$ are satisfied.
Then the sum (\ref{asymp_expansion_al0=-1}) is the asymptotic approximation for the solution
$u_\varepsilon$ to the boundary-value problem~$(\ref{probl}),$ i.e., \quad
$
      \exists \, {C}_0 >0 \ \ \exists \, \varepsilon_0>0 \ \
      \forall\, \varepsilon\in(0, \varepsilon_0):
$
\begin{equation}\label{t0_al0=-1}
      \max\limits_{t\in[0,T]}
      \left\|
       {U}_\varepsilon^{(1)} (\cdot, t) - u_\varepsilon (\cdot, t)
      \right\|_{L^2(\Omega_\varepsilon)}
  +   \left\|
       {U}_\varepsilon^{(1)} - u_\varepsilon
      \right\|_{L^2(0, T; {H}^1(\Omega_\varepsilon))}
 \leq {C}_0 \, \mu_0(\varepsilon),
\end{equation}
where $\mu_0(\varepsilon) = o(\varepsilon)$ as $\varepsilon \to 0$ and
\begin{equation}\label{mu_al0=-1}
\mu_0(\varepsilon)
 = \bigg(
    \varepsilon^{1+\frac{\mathfrak{a}}2}
 +  \sum_{i=1}^3
     \Big(
      (1-\delta_{\alpha_{i}, 1}) \varepsilon^{\alpha_i}
 +    (1-\delta_{\beta_{i}, 1}) \varepsilon^{\beta_i}
     \Big)
 +   (1-\delta_{\beta_{0}, 0})\varepsilon^{\beta_0+1}
  \bigg).
\end{equation}
\end{theorem}

\begin{proof}
The proof of Theorem \ref{mainTheorem_al0=-1} repeats the proof of Theorem \ref{mainTheorem}.
The only difference is the residual on the boundary of the node, namely
\begin{equation*}
\breve{R}_{\varepsilon}^{(0)}(x, t)
 =  \varepsilon^{-1} \kappa_0\Big({U}_\varepsilon^{(1)}(x, t)\Big)
 -  \kappa_0^{\prime}(0){N_1}\Big(\frac{x}{\varepsilon}, t\Big)
 +  (\delta_{\beta_0, 0} - 1) \, \varepsilon^{\beta_0} \varphi_\varepsilon^{(0)}(x, t),
\qquad (x, t)\in\Gamma_\varepsilon^{(0)}\times(0, T).
\end{equation*}
We estimate the value
$$
R_{\varepsilon, 3} (v)
 = \int_{\Gamma_\varepsilon^{(0)}} \breve{R}_{\varepsilon}^{(0)} v \, d\sigma_x
$$
with the help of (\ref{ineq-3})  and Taylor's formula. As a result, we get
\begin{equation*}
|R_{\varepsilon, 3} (v)|
 \le \check{C} \sqrt{| \Gamma_0 \, |_2} \,
     \big( \varepsilon + \varepsilon^{\beta_0} ( 1 - \delta_{\beta_0, 0}) \big)
     \, \varepsilon \, \| v \|_{H^1(\Omega_\varepsilon)},
\end{equation*}
for all $v\in L^2(0, T; \mathcal{H}_\varepsilon)$ and a.e. $t\in(0, T).$
\end{proof}

\begin{remark}\label{rem_6_1}
As we have argued inequalities (\ref{t7}), (\ref{t-joint0}) and (\ref{t9}), we can prove similar inequalities in the case $\alpha_0 \in [-1, 0)$ using
(\ref{t0_-1<al0<0}) and (\ref{t0_al0=-1}).
\end{remark}

\section{Comments}\label{comments}

{\bf 1.} At first glance it may seem that there is no difference between the nonlinear Robin condition
(\ref{in-1}) in the problem (\ref{probl}) and the corresponding linear Neumann condition, since the term $\kappa_i(u_\varepsilon,x_i,t)$ is multiplied by~$\varepsilon^{\alpha_i}$
$(i\in\{1, 2, 3\}).$ However, this is true only if $\alpha_i > 1.$ If $\alpha_i = 1,$ then the
new blow-up term
$$
2 \pi \,  h_i(x_i)\, \kappa_i\Big(\omega_0^{(i)}(x_i, t), x_i, t \Big),
$$
which takes into account the curvilinearity of the thin cylinder $\Omega_\varepsilon^{(i)}$ through the function $h_i,$
appears in the differential equation of the corresponding limit problem (see (\ref{main}) and (\ref{main_al0<0})).

What happens when  $\alpha_i < 1$ for some $i\in\{1, 2, 3\};$ to be specific we put  $\alpha_1 < 1.$
 As in the case ${\bf C3}({\rm a})$ we additionally suppose that there is a constant $k_-$ such that
 $0< k_- \leq \kappa_1^\prime(s,x_1,t)$ for all $s\in\Bbb{R}$ uniformly with respect to $x_1\in[0, \ell_1]$ and $t\in[0, T],$
  and  $\kappa_1(0,x_1,t)=0.$ Then from the integral identity (\ref{int-identity}) and inequalities (\ref{kappa_ineq+}),  (\ref{ineq1}), (\ref{equivalent_norm}), (\ref{ineq-3}) and (\ref{apr_ocin})  it follows
\begin{multline*}
\varepsilon^{\alpha_1} \int_{\Gamma_\varepsilon^{(1)}\times(0,T)} u_\varepsilon^2\, d\sigma_x dt
\le C_1 \Big(|k(0)| \varepsilon + \varepsilon^{\alpha_0 + 1} |\kappa_0(0)| + \sum_{i=2}^3 \varepsilon^{\alpha_i} \max\limits_{[0, \ell_i]\times[0,T]} |\kappa_i(0, x_i, t)|
\\
+
\|f\|_{L^2(\Omega_\varepsilon \times (0,T))} + \varepsilon^{\beta_0} \|\varphi_\varepsilon^{(0)}\|_{L^2(\Gamma_\varepsilon^{(0)}\times (0,T)) }
+
\sum\limits_{i=1}^{3} \varepsilon^{\beta_i - \frac12} \|\varphi_\varepsilon^{(i)}\|_{L^2(\Gamma_\varepsilon^{(i)}\times (0,T)) }\Big)
\|u_\varepsilon\|_{L^2(0,T; \mathcal{H}_\varepsilon)} \le C_2 \varepsilon^2.
\end{multline*}
Now with the help of (\ref{ineq-2}) we get
\begin{equation*}
 \int_{\Omega_\varepsilon^{(1)}\times(0,T)} u_\varepsilon^2\, dx dt \le
 C_3 \Bigg( \varepsilon^2 \int_{\Omega_\varepsilon^{(1)}\times (0,T)} |\nabla_{x}u_\varepsilon|^2 \, dx dt +
  \varepsilon^{1 -\alpha_1} \varepsilon^{\alpha_1} \int_{\Gamma_\varepsilon^{(1)}\times (0,T)} u_\varepsilon^2 \, d\sigma_x dt \Bigg) \le C_4 \varepsilon^\vartheta,
\end{equation*}
where $\vartheta:= \min \{4, \ 3 - \alpha_1\}.$ This means that
\begin{equation}\label{est-alpha_2}
\frac{1}{\varepsilon^2} \int_{\Omega_\varepsilon^{(1)}\times(0,T)} u_\varepsilon^2\, dx dt \le C_4 \, \varepsilon^{\min\{2, 1 - \alpha_1\}} \longrightarrow 0\quad \text{as} \ \ \varepsilon \to 0.
\end{equation}
Due to (\ref{est-alpha_2}) we conclude that $\omega_0^{(1)}\equiv 0.$
If $\alpha_1<0,$ then we can state that there are two  independent problems (\ref{main_al0<0}) $(i= 2$ and $i=3)$ to determine
$\omega_0^{(2)}$ and $\omega_0^{(3)}.$  The view of the limit problem is still unknown for $\alpha_1\in [0, 1);$
we guess that the limit problem will depend on the parameter~$\alpha_0$ in addition.

\medskip

{\bf 2.} From obtained results it follows that the asymptotic behaviour of the solution essentially depends
 on the parameter $\alpha_0$ characterizing the intensity of processes at the boundary of the node.
If $\alpha_0>0$ and $\beta_0>0,$ then the limit problem (\ref{main}) does not feel both  those processes and the node geometry.
In this case, in order to take into account all these factors on the global level, we propose to consider a system  consisting of
the limit problem (\ref{main}) and (\ref{omega_probl*}) on the graph. The coefficients $\boldsymbol{d_1^*},$ $\boldsymbol{\delta_1^{(2)}}$ and $\boldsymbol{\delta_1^{(3)}}$ in the Kirchhoff transmission conditions of the problem (\ref{omega_probl*})
pay respect to all parameters $\{\alpha_i\}_{i=0}^3,$ $\{\beta_i\}_{i=0}^3,$ and many other features (see formulas (\ref{const_d_*}) and (\ref{delta_1})).
This proposition is justified by Theorem \ref{justification}.

The same observation holds for the cases $\alpha_0 \in (-1, 0)$ and $\alpha_0=-1$ despite
the limit problem is split into three problems (\ref{main_al0<0}) and
the problems  (\ref{omega_probl*_-1<al0<0}) and (\ref{omega_probl*_al0=-1})  are independent at first glance.
In fact the Dirichlet condition at the vertex $x=0$ in the problems (\ref{omega_probl*_-1<al0<0})
indicate the dependence of these problems both on previous solutions
$\omega_\mathfrak{n}^{(1)},$ $\omega_\mathfrak{n}^{(2)}$ and $\omega_\mathfrak{n}^{(3)}$ $(\mathfrak{n}\in\mathfrak{A})$ and on other factors
through the values $V_\mathfrak{n}$ and $\boldsymbol{\delta_\mathfrak{n}^{(i)}}$ (see (\ref{V}),  (\ref{t-delta-1}) and (\ref{t-delta-2})) for $\alpha_0 \in (-1, 0);$ \ in the case $\alpha_0=-1$ see the problems (\ref{omega_probl*_al0=-1}) and formulas (\ref{t-delta-3}).

\medskip

{\bf 3.}
Thanks to estimates (\ref{t7}) and (\ref{t-joint0}), we get the zero-order approximation of the gradient (flux) of the solution
$$
\nabla u_\varepsilon(x, t)  \sim \frac{\partial\omega^{(i)}_0}{\partial x_i}(x_i, t)  \quad \text{as}\quad \varepsilon \to 0
$$
in each curvilinear cylinders~$\Omega_{\varepsilon,\mathfrak{a}}^{(i)}$ $(i=1,2, 3)$ and
$$
\nabla u_\varepsilon(x, t) \sim \nabla_{\xi}\big({N}_{1}(\xi, t)\big)\Big|_{\xi=\frac{x}{\varepsilon}} \quad \text{as}\quad \varepsilon \to 0.
$$
in the neighbourhood $\Omega^{(0)}_{\varepsilon, \ell_0}$ of the node.

The estimate (\ref{t-joint0})  is very important if we investigate
processes occurring in a neighbourhood of the node. In this case, in terms of practical application, we propose to apply numerical methods
not to original problems in thin star-shaped junctions, as was done for instance in  \cite{Mardal} without enough accuracy (see the Introduction), and to the corresponding problem for $N_1$ (see (\ref{junc_probl_n}), (\ref{junc_probl_n_-1<al0<0}) at $n=1$ and (\ref{junc_probl_1_al0=-1})).

\medskip

{\bf 4.}
An important problem of existing multi-scale methods is their stability and accuracy. The proof of the
error estimate between the constructed approximation and the exact solution is a general principle that has been
applied to the analysis of the efficiency of a multi-scale method. In our paper, we have constructed and justified the asymptotic
approximation for the solution to problem (\ref{probl}) and proved the corresponding estimates for different values of the parameters
$\{\alpha_i\}$ and $\{\beta_i\}.$ It should be noted here that we do not assume any orthogonality conditions for the right-hand sides in the equation and in the nonlinear Robin boundary conditions.

\smallskip

The results obtained in Theorems~\ref{mainTheorem}, \ref{mainTheorem_-1<al0<0}, \ref{mainTheorem_al0=-1} and Corollaries \ref{corollary1}, \ref{corollary2} argue that  in depending on $\{\alpha_i\}$ and $\{\beta_i\}$ it is possible to replace the complex boundary-value problem (\ref{probl}) with
the corresponding limit problem (\ref{main}) ((\ref{main_al0<0})) on the graph $\mathcal{I}$  with sufficient accuracy measured by the parameter $\varepsilon$ characterizing the thickness and the local geometrical irregularity of the thin star-shaped junction $\Omega_\varepsilon.$

\end{document}